%% file: gammalim_asym.tex
\documentclass[11pt]{article}

\usepackage[utf8]{inputenc}	
\usepackage[T1]{fontenc}
\usepackage[ngerman,english]{babel}	
\usepackage[a4paper,DIV=11]{typearea}
\usepackage{parskip}
\usepackage{amsmath}		
\usepackage{amsfonts}		
\usepackage{amssymb}		
\usepackage{amsthm}		  
\usepackage{mathtools}  
\usepackage{color}      
\usepackage{pstricks}   
\usepackage{pst-plot}
\usepackage{pst-math}
\usepackage{graphicx}
\usepackage{booktabs, verbatim}

\newcommand{\R}{\ensuremath{\mathbb{R}}}	
\newcommand{\Z}{\ensuremath{\mathbb{Z}}}	

\newcommand{\Sp}{\ensuremath{\mathbb{S}}}	

\newcommand{\Xt}{X^\theta}
\newcommand{\Xs}{X^*}

\newcommand{\ES}{E_\text{sym}}
\newcommand{\EA}{E_\text{asym}}

\newcommand{\mds}{m_2^*}
\newcommand{\mts}{m_3^*}

\newcommand{\beq}{\begin{equation}}
\newcommand{\eeq}{\end{equation}}

\newcommand{\loc}{\text{loc}}

\newcommand{\defas}{\mathrel{\mathop{:}\!\!=}}  

\newcommand{\with}{\;\middle\vert\;}  

\newcommand{\wto}{\xrightharpoonup{}}          

\newcommand{\lm}{\mathcal{L}}
\newcommand{\hm}{\mathcal{H}}
\newcommand{\co}{\mathcal{O}}
\newcommand{\so}{o}

\def\Xint#1{\mathchoice
{\XXint\displaystyle\textstyle{#1}}%
{\XXint\textstyle\scriptstyle{#1}}%
{\XXint\scriptstyle\scriptscriptstyle{#1}}%
{\XXint\scriptscriptstyle\scriptscriptstyle{#1}}%
\!\int}
\def\XXint#1#2#3{{\setbox0=\hbox{$#1{#2#3}{\int}$}
\vcenter{\hbox{$#2#3$}}\kern-.5\wd0}}

\def\dashint{\Xint-}

\DeclareMathOperator{\id}{id}
\DeclareMathOperator{\sgn}{sgn}

\newcommand{\tod}{\downarrow}
\newcommand{\tou}{\uparrow}

\newcommand{\vect}[1]{\begin{smallmatrix}#1\end{smallmatrix}}

\theoremstyle{plain}
\newtheorem{thm}{Theorem}
\newtheorem*{thm*}{Theorem}
\newtheorem{prop}{Proposition}
\newtheorem*{prop*}{Proposition}
\newtheorem{lem}{Lemma}
\newtheorem{cor}{Corollary}

\theoremstyle{definition}

\theoremstyle{plain}
\newtheorem{rem}{Remark}

\title{Asymmetric domain walls of small angle\\in soft ferromagnetic films}
\author{Lukas D\"oring\thanks{RWTH Aachen, Lehrstuhl I f\"ur Mathematik, Pontdriesch 14-16, 52056 Aachen (email: L.Doering@math1.rwth-aachen.de)} \quad Radu Ignat\thanks{Institut de Math\'ematiques de Toulouse, Universit\'e Paul Sabatier, 31062 Toulouse, France (email: Radu.Ignat@math.univ-toulouse.fr)}}

\begin{document}
\maketitle

\begin{abstract}
We focus on a special type of domain walls appearing in the Landau-Lifshitz theory for soft ferromagnetic films. These domain walls are divergence-free $\Sp^2$-valued transition layers that connect two directions $m_\theta^\pm\in\Sp^2$ (differing by an angle $2\theta$) and minimize the Dirichlet energy. Our main result is the rigorous derivation of the asymptotic structure and energy of such ``asymmetric'' domain walls in the limit $\theta\tod 0$. As an application, we deduce that a supercritical bifurcation causes the transition from symmetric to asymmetric walls in the full micromagnetic model.

\end{abstract}
\noindent{\scriptsize\textbf{Keywords:} micromagnetics, transition layer, $\Gamma$-convergence, $\Sp^2$-harmonic maps, spectral gap, Eikonal equation.}\\
{\scriptsize\textbf{MSC:} 49S05, 49J45, 78A30, 35B32, 35B36}

\section{Introduction}\label{sec:introduction}
\input{intro}

\section{Properties of the coefficients $E_0$ and $E_1$}\label{sec:propertiese0e1}
\input{propertiesE0E1}

\section{Compactness and lower bounds}\label{sec:lower}
\input{lower}

\section{Upper bound}\label{sec:upper}
\input{upper}

\appendix
\section*{Appendix}

\section{Classical inequalities}
\input{inequalities}

\section{Characteristics of the eikonal equation}\label{sec:chareik}
\input{eikonal}

\textbf{Acknowledgments}: The authors thank Felix Otto for his strong support and very helpful discussions and suggestions. LD acknowledges the hospitality of the universities Paris-Sud and Paul Sabatier and the MPI MIS, as well as support of its IMPRS. RI gratefully acknowledges the hospitality of the MPI MIS, where part of this work was carried out; he also acknowledges partial support by the ANR project ANR-14-CE25-0009-01.

\bibliography{references}
\bibliographystyle{abbrv}
\end{document}

%% file: intro.tex
\subsection{Model}
We consider the following model for asymmetric domain walls: The magnetization is described by a unit-length vector field
\begin{align*}
m=(m_1, m_2, m_3) \colon \Omega \to \mathbb{S}^2,
\end{align*}
where the two-dimensional domain 
\begin{align*}
x=(x_1,x_3) \in \Omega = \R \times (-1,1)
\end{align*}
corresponds to a cross-section of a ferromagnetic sample that is parallel to the $x_1x_3$-plane. The following ``boundary conditions at $x_1=\pm\infty$'' are imposed so that a transition from the angle $-\theta$ to $\theta\in(0,\tfrac{\pi}{2}]$ is generated and a domain wall forms parallel to the $x_2x_3$-plane (see Figure \ref{fig:samplegeometry}):
\begin{align}\label{eq:bcinfty}
m(\pm\infty,\cdot) = m^\pm_\theta := (\cos\theta,\pm\sin\theta,0),
\end{align}
with the convention
\begin{align}
\label{convent}
  f(\pm\infty,\cdot) = a_\pm \iff \int_{\Omega_+} \lvert f - a_+ \rvert^2 \, dx + \int_{\Omega_-} \lvert f - a_- \rvert^2 \, dx < \infty,
\end{align}
where $\Omega_+=\Omega\cap\{x_1\geq0\}$ and $\Omega_-=\Omega\cap\{x_1\leq0\}$. Throughout the paper, we use the variables $x=(x_1,x_3)\in\Omega$ together with the differential operator $\nabla=(\partial_1 ,\partial_3 )$, and we denote by 
\[ m'=(m_1,m_3) \]
the projection of $m$ on the $x_1x_3$-plane.
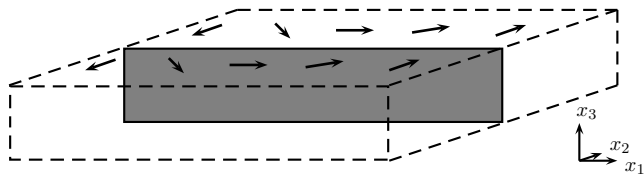
\begin{figure}[b]
\centering
\begin{pspicture}(0,0)(8,2.5)
\psframe[fillstyle=solid,fillcolor=gray](1.5,0.5)(6.5,1.5)
\psframe[linestyle=dashed](0,0)(5,1)
\psline[linestyle=dashed](0,1)(3,2)
\psline[linestyle=dashed](5,1)(8,2)
\psline[linestyle=dashed](5,0)(8,1)
\psline[linestyle=dashed](3,2)(8,2)
\psline[linestyle=dashed](8,1)(8,2)
\psline{->}(7.5,0)(8,0)
\rput[l](8.1,-0.1){\psscalebox{0.7}{$x_1$}}
\psline{->}(7.5,0)(7.8,0.1)
\rput[lb](7.9,0.1){\psscalebox{0.7}{$x_2$}}
\psline{->}(7.5,0)(7.5,0.5)
\rput[b](7.6,0.55){\psscalebox{0.7}{$x_3$}}
\psline[linewidth=1pt]{<-}(2.4,1.66)(2.8,1.8)
\psline[linewidth=1pt]{<-}(3.7,1.62)(3.5,1.82)
\psline[linewidth=1pt]{->}(4.3,1.73)(4.8,1.73)
\psline[linewidth=1pt]{->}(5.3,1.7)(5.8,1.78)
\psline[linewidth=1pt]{->}(6.4,1.66)(6.8,1.8)
\psline[linewidth=1pt]{<-}(1,1.2)(1.4,1.34)
\psline[linewidth=1pt]{<-}(2.3,1.16)(2.1,1.36)
\psline[linewidth=1pt]{->}(2.9,1.27)(3.4,1.27)
\psline[linewidth=1pt]{->}(3.9,1.24)(4.4,1.32)
\psline[linewidth=1pt]{->}(5,1.2)(5.4,1.34)
\end{pspicture}
\caption{The cross-section $\Omega$ in a ferromagnetic sample on a mesoscopic level.}
\label{fig:samplegeometry}
\end{figure}

The set $X^\theta$ of all magnetization configurations of wall angle $\theta\in[0,\pi]$ is defined as
\begin{align}
X^\theta := \left\{m \in \dot{H}^1(\Omega,\Sp^2) \with m(\pm\infty,\cdot) \stackrel{\eqref{convent}}{=} m^\pm_\theta\right\}.
\end{align}
The asymmetry of ``asymmetric'' domain walls is a consequence of $m$ trying to avoid ``magnetic charges'' in the bulk and on the surface of the sample, so that no magnetic stray-field is generated via Maxwell's equations, see Remark~1(ii) below. In other words, the main feature of asymmetric walls is the flux-closure constraint
\begin{align}\label{flux-clos}
\nabla\cdot (m' \mathbf{1}_\Omega)=0 \quad \text{in} \quad \mathcal{D}'(\R^2).
\end{align}
Observe that for any $m\in\dot{H}^1(\Omega,\mathbb{S}^2)$ satisfying \eqref{flux-clos}, i.e., 
\begin{align}\label{flux-clos-classical}
\nabla\cdot m'=0 \text{ in $\Omega$} \quad \text{and} \quad m_3=0 \text{ on $\partial \Omega$},
\end{align}
there exists a unique constant angle $\theta_m\in[0,\pi]$ such that the $x_3$-average (which will always be denoted by a bar $\bar{\quad}$) satisfies
\begin{align}\label{deftheta}
  \bar{m}_1 (x_1):= \dashint_{-1}^{1} \! m_1(x_1, x_3) \, dx_3 = \cos\theta_m \quad  \text{for all } x_1\in \R.
\end{align}
Moreover, such vector fields have the property\footnote{It is a direct consequence of Poincar\'e's inequalities \eqref{eq:poincm1} and \eqref{eq:poincm3}, together with Remark \ref{rem_zero} (see below) and \cite[Lemma~3]{dioreducedmodel13}.}
\[ m'(\pm\infty,\cdot)=(\cos \theta_m, 0) \quad \text{and} \quad \lvert m_2\rvert(\pm\infty,\cdot)=\sin\theta_m \]
in the sense of \eqref{convent}. We define the set $X_0$ as the non-empty (see Proposition~\ref{prop:upperbd} below and \cite[Appendix~A]{dioreducedmodel13}) set of such configurations $m$ that additionally change sign as $\lvert x_1 \rvert \to \infty$, namely $m_2(\pm\infty,\cdot)=\pm\sin\theta_m$ in the sense of \eqref{convent}:
\begin{align}\label{defx0}
X_0 := \left\{ m\in\dot{H}^1(\Omega,\Sp^2) \with \nabla \cdot m'=0 \text{ in }\Omega,\, m_3=0 \text{ on }\partial\Omega,\,m(\pm\infty,\cdot) \stackrel{\eqref{convent}}{=} m^\pm_{\theta_m} \right\}
\end{align}

\medskip
\begin{rem}
\label{rem_zero}
\begin{enumerate}
\item Observe that if $\theta_m\in \{0,\pi\}$ for $m\in\dot{H}^1(\Omega,\Sp^2)$ with $\nabla\cdot(m'\mathbf{1}_\Omega)=0$ in $\mathcal{D}'(\R^2)$ -- in particular if $m\in X_0$ --, then $m\in\{\pm \mathbf{e}_1\}$: Indeed, since $\lvert\bar{m}_1\rvert\equiv 1$ in $\R$ and $\lvert m \rvert=1$ in $\Omega$, we deduce $\lvert m_1 \rvert \equiv 1$ and $m_2 \equiv m_3\equiv 0$ in $\Omega$.
\item Observe that the set $X_0$ as defined in \eqref{defx0} does not contain symmetric magnetization configurations $m=m(x_1)$ provided $m\not\in\{\pm \mathbf{e}_1\}$. Indeed, if $m=m(x_1)$, then \eqref{flux-clos-classical} implies $\partial_1 m_1 = 0$, i.e. $m_1\equiv\cos\theta_m$, and $m_3\equiv 0$. This is incompatible with the requirement in \eqref{defx0} that $m_2$ changes sign, unless $\theta_m\in\{0,\pi\}$, i.e. $m\in\{\pm\mathbf{e}_1\}$.
\end{enumerate}
\end{rem}

Our aim is to study the following minimization problem:\footnote{We refer to Section~\ref{sec:origineasym} for a brief motivation of $\EA(\theta)$ and to \cite{dioreducedmodel13} for its rigorous derivation from the full Landau-Lifshitz energy.}
\begin{align}\label{eq:defsffree}
\EA(\theta) \defas \min_{m\in X_0\cap X^\theta} \int_{\Omega} \lvert \nabla m \rvert^2 dx
\end{align}
For every $\theta\in [0, \pi]$, the minimum in \eqref{eq:defsffree} is indeed attained, which is essentially due to a concentration-compactness result that copes with the change of sign of $m_2(\pm\infty,\cdot)$ (see \cite[Theorem~3]{dioreducedmodel13}).
The minimizers stand for asymmetric domain walls and we are going to characterize their structure and energy as the angle $\theta\tod0$. The variational problem \eqref{eq:defsffree} closely resembles the $\Sp^2$-harmonic map problem with an additional divergence constraint.

\subsection{Asymmetric domain walls}
In the physics literature \cite{labonte69,hubert69}, two different types of asymmetric domain walls have been found via the construction of models and also numerical simulation: asymmetric Bloch walls and asymmetric N\'eel walls. These transition layers have a width that is comparable to the film thickness and ensure \eqref{flux-clos-classical} at the expense of non-vanishing $\partial_3 m$ in $\Omega$, see Remark 1(ii). This makes asymmetric walls favored over other types of transition layers only in sufficiently thick films. Both asymmetric N\'eel and Bloch walls can also be obtained numerically as critical points\footnote{Actually, we conjecture them to be (local) minimizers of $\EA(\theta)$, at least for certain ranges of wall angles $\theta$.} of a discretized $\EA(\theta)$ (cf. Figures~\ref{asym}~and~\ref{fig:asymenergy}).
\begin{figure}[ht] %
  \begin{minipage}{0.48\linewidth}
    \centering
    \includegraphics[height=4cm]{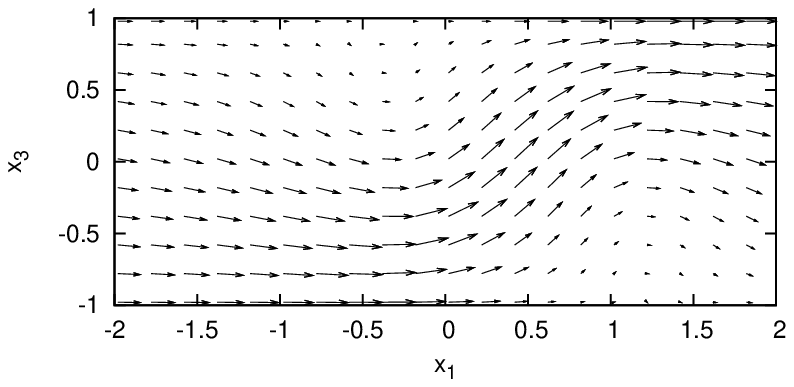} 
    \rput(3.2,2.8){
      \pscircle*[linecolor=white](0,0){0.31cm}
      \pscircle(0,0){0.3cm}
      \rput{45}(0,0){\psline(-0.3,0)(0.3,0)}
      \rput{-45}(0,0){\psline(-0.3,0)(0.3,0)}
    }
    \rput(-2.0,2.8){
      \pscircle*[linecolor=white](0,0){0.31cm}
      \pscircle(0,0){0.3cm}
      \pscircle*(0,0){0.05cm}
    }
  \end{minipage}
  \;
  \begin{minipage}{0.48\linewidth}
    \centering
    \includegraphics[height=4cm]{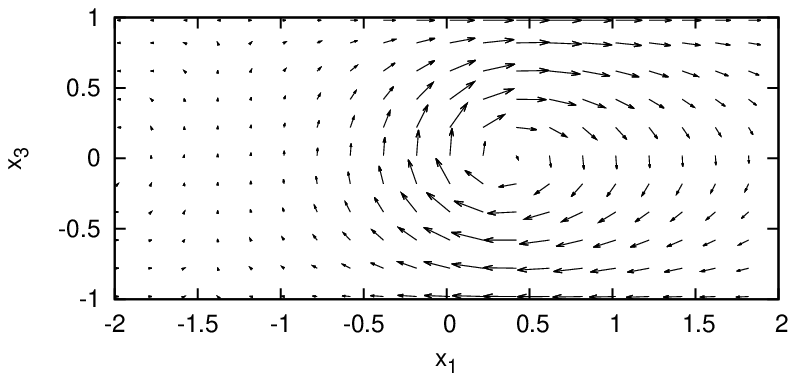} 
    \rput(3.2,2.85){
      \pscircle*[linecolor=white](0,0){0.31cm}
      \pscircle(0,0){0.3cm}
      \rput{45}(0,0){\psline(-0.3,0)(0.3,0)}
      \rput{-45}(0,0){\psline(-0.3,0)(0.3,0)}
    }
    \rput(1.35,2.85){\psscalebox{0.7}{
      \pscircle*[linecolor=white](0,0){0.31cm}
      \pscircle(0,0){0.3cm}
      \rput{45}(0,0){\psline(-0.3,0)(0.3,0)}
      \rput{-45}(0,0){\psline(-0.3,0)(0.3,0)}
    }}
    \rput(-2.0,2.85){
      \pscircle*[linecolor=white](0,0){0.31cm}
      \pscircle(0,0){0.3cm}
      \pscircle*(0,0){0.05cm}
    }
  \end{minipage}
  \vspace{-4ex}
  \caption{Asymmetric N\'eel wall (on the left) and asymmetric Bloch wall (on the right). Numerics.}
  \label{asym}
\end{figure}
\begin{figure}
\centering
\input{wallenergy_asym.tex}
\rput(-6.6,2.5){\psscalebox{0.8}{$E_0 \sin^2\theta + E_1 \sin^4\theta$}}
\psline[linewidth=0.5pt](-6.6,2.7)(-4.0,3.3)
\rput(-7,3.7){\psscalebox{0.8}{asym.-Bloch type}}
\rput(-7,3.4){\psscalebox{0.8}{critical point}}
\psline[linewidth=0.5pt](-7,4.0)(-6.2,4.3)
\rput(-2.5,2.5){\psscalebox{0.8}{asym.-N\'eel type}}
\rput(-2.5,2.2){\psscalebox{0.8}{critical point}}
\psline[linewidth=0.5pt](-2.5,2.8)(-2.9,3.6)
\caption{The energies of asymmetric-N\'eel (solid line) and asymmetric-Bloch (dotted line) type critical points of a discretized $\EA(\theta)$, as well as the expansion of $\EA(\theta)$ that we give in \eqref{eq:expansioneasym} (dashed line). According to the numerics, the expansion \eqref{eq:expansioneasym} is a good approximation of the energy of asymmetric N\'eel walls, with a relative error less than $15\%$ across the whole range of wall angles.}
\label{fig:asymenergy}
\end{figure}
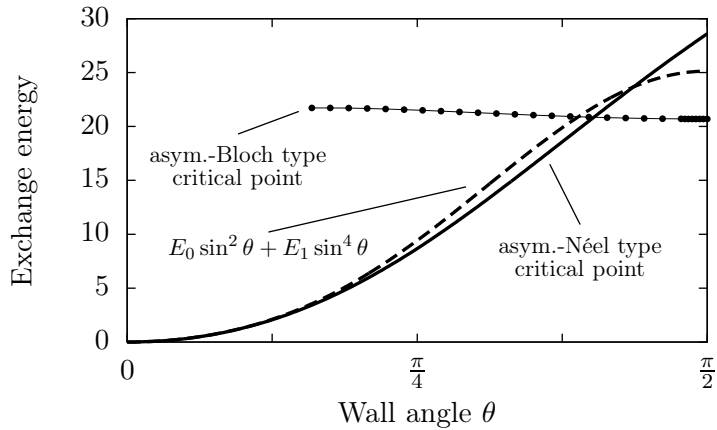
We refer to \cite{hubertschaefer98} for experimental pictures of magnetic domains and the domain walls in-between, and to \cite{ottocrossover02} for a rigorous derivation of a precise regime in which asymmetric walls minimize the Landau-Lifshitz energy.

Judging from the models and the numerical results, there are (at least) three ways to distinguish asymmetric N\'eel from asymmetric Bloch walls:
\begin{itemize}
\item On the film surface, the asymmetric N\'eel wall rotates in a non-monotonic way (i.e., considered as a map $m\big\vert_{\partial\Omega} \colon \partial\Omega\simeq\Sp^1\to\Sp^1$, its phase is non-monotonic), while the asymmetric Bloch wall rotates monotonically; this feature is actually used to experimentally distinguish the asymmetric N\'eel wall from other wall types in images obtained by Kerr microscopy \cite[Sec.~5.5.3~(B)]{hubertschaefer98}.
\item The asymmetric N\'eel wall (up to a translation in $x_1$) is invariant under the symmetry $x\mapsto -x$, $m_2 \mapsto -m_2$, while (except for the special case $\theta=\tfrac{\pi}{2}$) the asymmetric Bloch wall does not respect any of the symmetries of the energy functional $\EA(\theta)$.
\item Defining the winding number $W := \deg(m\big\vert_{\partial\Omega}\colon \Sp^1 \to \Sp^1)$, the asymmetric N\'eel wall has a trivial $W=0$, while the asymmetric Bloch wall satisfies $\lvert W \rvert=1$ (see also Section~\ref{sec:topology}).
\end{itemize}

Since the magnetization of an asymmetric Bloch wall points into the opposite direction on the top film surface with respect to the bottom surface, it is expected to be energetically more costly than the asymmetric N\'eel wall when sufficiently strong magnetic fields are applied along the $x_1$ direction, i.e. when the wall angle $\theta$ decreases from $\tfrac{\pi}{2}$ to $0$ (cf. \cite{berkovetal93}, the quotation from \cite{hubertschaefer98} in Section~\ref{sec:origineasym}, and Figure~\ref{fig:asymenergy}).

\subsection{Main results}
Our main goal is to establish the following asymptotic expansion of $\EA$ in the wall angle~$\theta$:
\begin{align}
  E_\text{asym}(\theta) = C_0 \, \theta^2 + C_1 \, \theta^4 + \so(\theta^4) \quad \text{as }\theta\tod 0,\label{eq:expansionC}
\end{align}
with some positive constants $C_0>0$ and $C_1>0$ that we compute explicitly.
In fact it will turn out to be more convenient to expand $\EA$ in terms of $\sin\theta$, i.e. derive  
\begin{align}\label{eq:expansioneasym}
  E_\text{asym}(\theta) = E_0 \, \sin^2\theta + E_1 \, \sin^4\theta + \so(\theta^4) \quad \text{as }\theta\tod 0.
\end{align}
Both expansions are related via $C_0=E_0$ and $C_1=-E_0/3+E_1$. Our method is based on an asymptotic development by $\Gamma$-convergence. In deriving \eqref{eq:expansioneasym}, we will obtain an asymptotic expansion up to order $\so(\theta^2)$ also of minimizers $m_\theta$ of $\EA(\theta)$. The expansion indicates that the variational problem \eqref{eq:defsffree} has -- up to translation and for small $0<\theta\ll 1$ -- exactly two global minimizers that are related by the reflection $x_3\leadsto-x_3$, $m_3\leadsto-m_3$. Moreover, both minimizers rotate non-monotonically on the sample surface, satisfy -- at least up to order $\theta^2$ -- the symmetry $x\leadsto-x$, $m_2\leadsto -m_2$, and are topologically trivial, see Proposition~\ref{prop:topology}.

In conjunction with a reduced model for extended tails of asymmetric domain walls that was derived in \cite{dioreducedmodel13}, the asymptotic expansion \eqref{eq:expansionC} allows us to prove that symmetric N\'eel walls turn into their asymmetric variant as the global wall angle increases via a supercritical bifurcation (see Sections~\ref{sec:origineasym}~and~\ref{sec:crossover} for details).

\subsubsection*{The leading-order coefficient $E_0$}
The leading-order coefficient in \eqref{eq:expansioneasym} is obtained as a consequence of deriving the asymptotic behavior of a minimizer $m_\theta$ of $\EA(\theta)$ up to order $\theta^2$ (cf. Lemma~\ref{lem:compsecondleadingm3} below):
\begin{align*}
m_\theta=\Bigl(\begin{smallmatrix}\cos \theta\\\sin\theta \, m_2^*\\0\end{smallmatrix}\Bigr)+\co(\theta^2) \quad \text{as } \theta\tod 0,
\end{align*}
where $m_2^*$ is a minimizing transition layer of the variational problem
\begin{align}\label{eq:defe0}
  E_0 = \min_{f\in \Xs} \int_\Omega \bigl\lvert \nabla f \bigr\rvert^2 dx.
\end{align}
The set of admissible configurations is defined as
\begin{align}
\label{defXstar}
\Xs=\left\{ f \in \dot{H}^1(\Omega, \R) \, \with \, f(\pm \infty, \cdot) \stackrel{\eqref{convent}}{=} \pm 1,\, \dashint_{-1}^1 \! f^2(\cdot, x_3) \, dx_3 = 1,\, \bar{f}(0)=0\right\}.
\end{align}
Observe that due to translation invariance of the minimization problem \eqref{eq:defe0} and the boundary conditions of admissible $f$ at $x_1=\pm\infty$, the constraint $\bar{f}(0)=\dashint_{-1}^1 \! f(0,x_3) \, dx_3 = 0$ is not a restriction.
The requirement that the average $\overline{f^2} \equiv 1$ follows from the property $\bar{m}_1 \equiv \cos\theta_m$ of any $m\in X_0$ by letting $\theta\tod 0$.

By matching upper and lower bounds on $\EA(\theta)$ in the spirit of $\Gamma$-convergence at the level of minimizers we prove:

\medskip
\begin{thm}\label{thm:mine0}
The leading-order coefficient of $\EA$ is given by $E_0$ as defined in \eqref{eq:defe0}, i.e.
\begin{align*}
\lim_{\theta \tod 0} \bigl(\sin^{-2}\theta \;\EA(\theta)\bigr) = E_0.
\end{align*}
Problem \eqref{eq:defe0} has exactly two minimizers, determined by:
\begin{align}\label{eq:defm2star}
  m_2^*(x) = \tanh(\tfrac{\pi}{2} x_1) + \sigma \sqrt{2} \sin(\tfrac{\pi}{2} x_3) \sqrt{1-\tanh^2(\tfrac{\pi}{2} x_1)}, \quad x=(x_1, x_3)\in \Omega,
\end{align}
corresponding to the choice of $\sigma \in \{\pm 1\}$. Moreover, one computes $E_0 = 4\pi$.
\end{thm}

The above theorem already justifies the physical prediction for the asymmetric N\'eel wall: First of all, observe that $m_2^*$ is a non-monotonic function on the surface $\partial\Omega=\{|x_3|=1\}$, so that the same behavior is conserved by the second component of the asymmetric N\'eel wall. Second, observe that $m_2^*$ is odd with respect to the origin, so that the second component of an asymmetric N\'eel wall approximately conserves the same symmetry. Indeed, by Lemma~\ref{lem:compsecondleadingm3} $m_{2,\theta}^*\to m_2^*$ in $\dot{H}^1(\Omega)$, in particular uniformly on a.e. vertical line $\{ x_1=a \}$. Due to the symmetry $x\leadsto -x$, $m_2\leadsto -m_2$ , we expect that for small $0<\theta\ll1$ the variational problem \eqref{eq:defsffree} has only two global minimizers $m_\theta$.

\subsubsection*{Second-leading order coefficient $E_1$}
The second-leading order coefficient of the asymmetric-wall energy is obtained by expanding the minimizer $m_\theta$ of $\EA(\theta)$ to the next order:
\begin{gather}\label{asym_doi}
\left.\begin{aligned}
m'_\theta &=\bigl(\begin{smallmatrix}\cos \theta\\0\end{smallmatrix}\bigr)+\sin^2\theta\,  \hat{m}' + \so(\theta^2)\\
m_{2,\theta} &= \sin\theta \,  m^*_2 + \so(\theta^2)
\end{aligned} \quad \right\} \quad \text{in $\dot{H}^1(\Omega)$ as } \theta \tod 0,
\end{gather}
where $m_2^*$ is a minimizing transition layer of $E_0$ (cf. \eqref{eq:defm2star}) and $\hat{m}'$ is given by
\begin{align}\label{eq:defhatm}
\hat{m}_1 := \tfrac{1-(m^*_2)^2}{2} \quad \text{and} \quad \partial_3 \hat{m}_3 := -\partial_1  \hat{m}_1 = \partial_1  \tfrac{(m^*_2)^2}{2} \quad \text{in }\Omega.
\end{align} 
Note that $\hat{m}_3$ is indeed uniquely determined due to its boundary condition $\hat{m}_3=0$ on $\partial \Omega$.
Using the expansion \eqref{asym_doi} and the fact that $m_2^*$ minimizes \eqref{eq:defe0}, one computes that the expansion \eqref{eq:expansioneasym} of $\EA(\theta)$ holds with the exact constant (see Section~\ref{sec:valueofe1})
\begin{align}\label{eq:defe1}
E_1 = \int_\Omega \Bigl( \lvert \nabla \hat{m}' \rvert^2 - \mu(x_1)\,\lvert \hat{m}' \rvert^2 \Bigr)  dx =  \tfrac{148}{35} \pi,
\end{align}
where $\mu\colon\R\to\R$ is the Lagrange multiplier corresponding to the constraint $\dashint_{-1}^1 f^2 dx_3 = 1$ in \eqref{eq:defe0}. It is given by
\begin{align}\label{eq:defmu}
\mu(x_1) = \dashint_{-1}^1 \lvert \nabla m^*_2 (x_1, x_3)\rvert^2 dx_3 = \pi \tfrac{d}{dx_1} \ell(x_1) = \tfrac{\pi^2}{2}\bigl(1-\ell^2(x_1)\bigr),
\end{align}
with
\begin{align}\label{eq:defell}
\ell(x_1) := \tanh(\tfrac{\pi}{2}x_1) \quad \text{for }x_1\in\R.
\end{align}
We again rigorously establish the second-order term of the expansion \eqref{eq:expansioneasym} by finding matching upper and lower bounds on the energy $\EA(\theta)$. 

\medskip
\begin{thm}\label{thm:mine1}
The second-leading order coefficient of $\EA$ is given by $E_1$ as in \eqref{eq:defe1}, i.e.
\begin{align*}
\lim_{\theta\tod 0} \Bigl( \sin^{-4}\theta \bigl(\EA(\theta) - E_0\sin^2\theta \bigr)\Bigr) = E_1.
\end{align*}
Moreover, if $\{m_\theta\}_\theta$ is a family of minimizers of $\{\EA(\theta)\}_\theta$, then the asymptotic expansion up to the second order \eqref{asym_doi} holds in $\dot{H}^1(\Omega)$. That is, up to a translation in $x_1$-direction and a subsequence we have
\begin{align*}
\tfrac{1}{\sin^2 \theta} \biggl(m_\theta-\biggl(\begin{smallmatrix}\cos \theta+\sin^2\theta \, \hat{m}_1\\\sin \theta \, \mds\\\sin^2\theta \, \hat{m}_3\end{smallmatrix}\biggr) \biggr)\to 0 \quad \text{strongly in } \dot{H}^1(\Omega) \text{ as } \theta\tod 0,
\end{align*}
where $\hat{m}'$ is defined in \eqref{eq:defhatm} for one of the two minimizers $\mds$ given in \eqref{eq:defm2star}. 
\end{thm}

From Theorems~\ref{thm:mine0}~and~\ref{thm:mine1} 
one finally deduces the expansion \eqref{eq:expansionC} of $\EA(\theta)$ in terms of $\theta$ instead of $\sin\theta$:
\[ \EA(\theta) = 4\pi\,\theta^2 + \tfrac{304}{105}\pi \, \theta^4 + \so(\theta^4). \]

\subsection{Topological properties of small-angle walls}\label{sec:topology}
There are (at least) two ways of defining topological invariants for smooth vector fields $m \in X_0$:
\begin{itemize}
\item Since $m\big\vert_{\partial\Omega}\colon \partial\Omega \approx \Sp^1 \to \Sp^1$, we can define a \emph{winding number} $W(m)= \deg(m\big\vert_{\partial\Omega})$.
\item Since $m_3=0$ on $\partial\Omega$, we can extend $m$ from $\Omega$ to $\tilde{\Omega}:=\R\times[-3,1)$ by even reflection in $(m_1,m_2)$ and odd reflection in $m_3$. The reflected configuration remains divergence free.
Identifying $\Omega\subset\tilde{\Omega}$ with the upper hemisphere of $\Sp^2$ and the reflected version $\tilde{\Omega}\setminus\Omega$ with the lower hemisphere, we see that (due to the boundary conditions at $x_1=\pm\infty$) the extension $\tilde{m} \colon \tilde{\Omega}\to \Sp^2$ of $m$ induces a vector field $\tilde{m}\colon \Sp^2 \to \Sp^2$, to which we can associate a \emph{degree} $D(m)=\deg(\tilde{m})$.
\end{itemize}
Both degrees can be represented in terms of the Jacobian determinants
\begin{align*}
J(\vect{m_1\\m_2}) = \partial_1 (\vect{m_1\\m_2}) \times \partial_3 (\vect{m_1\\m_2}),\quad \omega(m) = m\cdot (\partial_1 m \times \partial_3 m).
\end{align*}
We have (note \cite[Rem.~1.5.10]{nirenberg74})
\begin{align}
W(m) = \!\!\!\!\!\sum_{x \in \left(\vect{m_1\\m_2}\right)^{-1}\left(\vect{y_1\\y_2}\right)} \!\!\!\!\!\!\!\! \sgn\bigl(J(\vect{m_1\\m_2})(x)\bigr) \in \Z,\label{def:W}
\end{align}
for any regular value $\bigl(\vect{y_1\\y_2} \bigr)$ of $\left(\vect{m_1\\m_2}\right)$, where $\sgn(a)$ denotes the sign of $a\neq 0$. The integer $W$ is independent of the choice of the regular value $\bigl(\vect{y_1\\y_2} \bigr)$ (cf. \cite[Prop.~1.4.1]{nirenberg74}).

Moreover, for any two regular values $y$ and $z$ of $m\colon\Omega\to\Sp^2$ with $y_3 > 0$ and $z_3<0$, we have
\begin{align}
D(m) = \sum_{x \in m^{-1}(y)} \sgn\bigl(\omega(m)(x)\bigr) + \sum_{x \in m^{-1}(z)} \sgn\bigl(\omega(m)(x)\bigr) \in \Z.\label{def:D}
\end{align}
The sums are constant in $y$ and $z$, respectively (cf. \cite[Prop.~1.4.1]{nirenberg74}).

With these definitions the following relation holds (see also \cite[pg.~1021]{kmms14}):
\begin{align*}
W(m)\equiv D(m) \pmod 2.
\end{align*}
Furthermore, recall that there are alternative characterizations of $W$ and $D$ available by integrating $J$ and $\omega$, respectively:
\begin{align*}
\Z \ni W(m) &= \tfrac{1}{\pi} \int_\Omega J(\vect{m_1\\m_2}) \, dx = \tfrac{1}{\pi} \int_{\partial\Omega} m_1\partial_\tau m_2 \, d\hm^1(x),\\
\Z \ni D(m) &= \tfrac{1}{2\pi} \int_\Omega \omega(m) \, dx.
\end{align*}
The winding number $W$ is a classical quantity in the study of Ginzburg-Landau vortices \cite{bbh94} and has been generalized also to less regular vector fields \cite{brezisnirenberg95,boutetdemonvelgeorgescupurice91}.

By Young's inequality both $W(m)$ and $D(m)$ relate topological properties of $m$ to its exchange energy:
\begin{align}\label{eq:WDEineq}
\lvert W(m) \rvert \leq \tfrac{1}{2\pi} \int_\Omega \lvert \nabla m \rvert^2 dx, \quad \lvert D(m) \rvert \leq \tfrac{1}{4\pi} \int_\Omega \lvert \nabla m \rvert^2 dx.
\end{align}
Finally, note that an odd reflection in $m_3$ and even reflection in $(m_1,m_2)$ across one of the components of $\partial\Omega$, like in the definition of $\tilde{m}\colon\tilde{\Omega}\to\Sp^2$, sets $W$ to zero in $\tilde{\Omega}$, while $D$ is doubled in $\tilde{\Omega}$.

With these definitions, one immediately obtains:

\medskip
\begin{prop}\label{prop:topology}
For $0<\theta\ll 1$, any global minimizer $m_\theta$ of $\EA(\theta)$ is topologically trivial in the sense that $m$ has vanishing winding number and degree, i.e. $W(m_\theta)=D(m_\theta)=0$.
\end{prop}
\begin{proof}
Using \eqref{eq:expansionC}, choose $0<\theta\ll 1$ small enough so that $0\leq \EA(\theta) < 2\pi$. Then, the proposition directly follows from the fact that $W$ and $D$ are integers and the classical inequalities \eqref{eq:WDEineq} hold.
\end{proof}

\subsection{The origin of $\EA(\theta)$}\label{sec:origineasym}
The minimization problem \eqref{eq:defsffree} appears naturally in the asymptotic analysis of the micromagnetic energy in the limit $\eta \tod 0$, with a fixed parameter $\lambda >0$:\footnote{Recall the physical interpretation of the parameters: We have $\eta:=Q\frac{t^2}{d^2}$ and $\lambda:=\frac{t^2}{d^2} \ln^{-1} \frac{1}{\eta}$ where $Q$ is the quality factor (of the crystalline anisotropy), $d$ is the exchange length and $2t$ is the thickness of the ferromagnetic film. The regime ($\eta\ll 1$ and $\lambda\sim 1$) corresponds to ($Q\ll 1$ and $\ln \frac{1}{Q}\sim (\frac{t}{d})^2$), where the cross-over from symmetric to asymmetric walls is known to occur. In particular, $\lambda$ can be interpreted as measure of the film thickness, relative to the critical film thickness of the cross-over. The value $\cos\alpha$ can be seen as strength of the reduced external magnetic field. We refer to \cite{dioreducedmodel13,ottocrossover02} for further details.}
\begin{align*}
E_\eta(m)= \int_{\Omega} \lvert\nabla m\rvert^2 dx + \lambda \ln \tfrac{1}{\eta} \int_{\R^2} \lvert h\rvert^2 dx + \eta \int_{\Omega} (m_1 - \cos \alpha )^2 + m^2_3\, dx
\end{align*}
Here $m \colon \Omega \to \Sp^2$ describes a general, i.e. not stray-field free, transition layer that connects the two directions $m(\pm\infty,\cdot)\stackrel{\eqref{convent}}{=}m_\alpha^\pm$ for a fixed angle $\alpha\in[0,\tfrac{\pi}{2}]$ (cf. \eqref{eq:bcinfty}). The stray field $h=-\nabla u\colon\R^2\to \R^2$ is generated in the $x_1x_3$-plane due to non-vanishing ``magnetic charges'' $\nabla \cdot (m' \mathbf{1}_\Omega)$ via the static Maxwell equation:
\begin{align*}
\nabla\cdot (h+m'\mathbf{1}_{\Omega})=0 \quad \text{in } \mathcal{D}'(\R^2).
\end{align*}
In \cite{dioreducedmodel13}, rigorous asymptotic analysis, based on the $\Gamma$-convergence method, was used to derive a reduced model for the minimal energy of such transition layers:
\begin{align}\label{eq:mineeta}
\min_{m \in X^\alpha} E_\eta(m) \stackrel{\eta\tod 0}{\longrightarrow} \min_{\theta\in[0,\frac{\pi}{2}]}\Bigl(\EA(\theta) + \lambda \ES(\alpha-\theta)\Bigr), \quad \text{given any }\alpha\in[0,\tfrac{\pi}{2}],
\end{align}
for $\ES(\alpha-\theta)= 2\pi (\cos\theta-\cos\alpha)^2$. This reduced model confirms and renders more precisely a statement on extended tails of asymmetric N\'eel walls that is found in the physics literature \cite[Page 250]{hubertschaefer98}:
\begin{quote}
  ``The magnetization of an asymmetric N\'eel wall points in the same direction at both surfaces, which is [\ldots] favourable for an applied field along this direction. This property is also the reason why the wall can gain some energy by splitting off an extended tail, reducing the core energy in the field generated by the tail. [\ldots] The tail part of the wall profile increases in relative importance with an applied field, so that less of the vortex structure becomes visible with decreasing wall angle. At a critical value of the applied field the asymmetric disappears in favour of a symmetric N\'eel wall structure.''
\end{quote}
The symmetric N\'eel wall has been studied extensively by many authors (e.g. in \cite{melcher03,melcher04,dko06,ignatotto08,ignat09,dkmorepulsive03}). It occurs in very thin films (cf. \cite{ottocrossover02} for a precise regime), where the exchange energy suppresses variation along the thickness direction of the sample. Moreover, the stray-field energy suppresses an out-of-plane component of the magnetization on the sample surface, hence in the whole sample. More precisely, to leading order in $t/d$, the symmetric N\'eel wall is a smooth, one-dimensional transition layer with values in $\Sp^1$ that connects the boundary values $m_\alpha^\pm$ and minimizes $E_\eta$.
To this end, it exhibits two internal length scales: It combines a symmetric core of width $w_\text{core}\sim \lambda^{-1} \ln^{-1}\frac{1}{\eta}$ with two logarithmically decaying tails $w_\text{core}\lesssim \lvert x_1 \rvert \lesssim w_\text{tails}\sim \lambda\ln\frac{1}{\eta}/\eta$. The symmetric N\'eel wall is invariant with respect to all the symmetries of the variational problem (besides translation invariance). Its specific energy is to leading order given by the energy of the stray field generated in the tails. It depends quartically on the wall angle $\alpha\ll 1$.

The first part of the above quotation suggests that asymmetric domain walls can replace the symmetric wall core of the symmetric N\'eel wall, in an optimal way. The result \eqref{eq:mineeta} confirms this on the level of the energy: Indeed, asymptotically, $\ES(\alpha-\theta)$ is the micromagnetic energy of logarithmic wall tails that connect the boundary values $m_\theta^\pm$ of an asymmetric wall core, i.e. a minimizer of $\EA(\theta)$, with the global boundary conditions $m_\alpha^\pm$. Thus, \eqref{eq:mineeta} states that, asymptotically, $\min E_\eta$ splits into contributions from an asymmetric core and symmetric tails in an optimal way. The core wall angle $\theta$ serves as an indicator for the actual wall type (asymmetric N\'eel/Bloch for $\theta>0$ or symmetric N\'eel for $\theta=0$) and the relative amount of rotation in the wall core.

\subsection{Bifurcation from symmetric to asymmetric N\'eel walls}\label{sec:crossover}
We now address the second part of the quotation from \cite{hubertschaefer98} in the previous section, i.e. the core size $\theta$ as a function of the global wall angle $\alpha$ and the relative film thickness $\lambda$: Consider the reduced energy
\begin{align*}
[0,\tfrac{\pi}{2}] \ni \theta \mapsto \EA(\theta) + \lambda \ES(\alpha-\theta), \quad \text{for }\alpha\in[0,\tfrac{\pi}{2}]\text{ fixed}.
\end{align*}
Assuming smoothness of $\EA$ in $\theta$ and using \eqref{eq:expansionC}, critical points $\theta$ of the above reduced energy solve the equation
\begin{align*}
0 = C_0 \theta + 2C_1 \theta^3 - 2\pi \, \lambda (\cos\theta-\cos\alpha)\sin\theta + \so(\theta^3),
\end{align*}
which always has the trivial solution $\theta=0$. For $\lambda>2$ and $\cos\alpha<1-\frac{C_0}{2\pi\lambda}=1-\tfrac{2}{\lambda}$, expanding $\sin\theta$ and $\cos\theta$ up to order $\theta^4$, we obtain another branch of positive solutions\footnote{Note that at this point we need the expansion of $\EA$ up to order $\theta^4$.}
\begin{align*}
\theta \approx \sqrt{\tfrac{2\pi\lambda(1-\cos\alpha) - C_0}{2C_1 + \pi\lambda(1+\frac{1-\cos\alpha}{3})}} = \sqrt{\tfrac{\lambda(1-\cos\alpha) - 2}{\frac{304}{105} + \lambda \frac{4-\cos\alpha}{6}}} \quad \text{as }\alpha \tod \alpha^\text{crit}:=\arccos(1-\tfrac{2}{\lambda}).
\end{align*}
Evaluating the second derivative of the reduced model at these critical points, we see that for $\alpha>\alpha^\text{crit}$ (provided $\lambda>2$) the trivial zero $\theta=0$ becomes unstable, while the non-trivial branch -- corresponding to an asymmetric N\'eel wall for $0<\theta \ll 1$ -- is a minimizer.
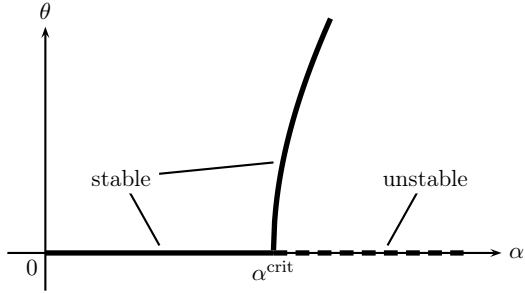
\begin{figure}
\centering
\begin{pspicture}(-1,-0.5)(7,3.5)
\psline{->}(-0.5,0)(6,0)
\psline{->}(0,-0.5)(0,3)
\psline[linewidth=2pt](0,0)(3,0)
\psline[linewidth=2pt,linestyle=dashed](3,0)(5.5,0)
\parametricplot[plotstyle=line,linewidth=2pt,plotpoints=51]{0.3}{0.375}{t 10 mul 2 3.1415 mul 44.78 mul 1 t 3.1415 2 mul div 360 mul cos neg add mul 4 3.1415 mul neg add 6 3.1415 mul 3.1415 44.78 mul 4 t 3.1415 2 mul div 360 mul cos neg add 3 div mul add div sqrt 15 mul}
\psline(1,1)(1.5,0.1)
\psline(1,1)(3,1.2)
\rput*(1,1){\psscalebox{0.8}{stable}}
\psline(5,1)(4.5,0.1)
\rput*(5,1){\psscalebox{0.8}{unstable}}
\rput[b](0,3.1){\psscalebox{0.8}{$\theta$}}
\rput[l](6.1,0){\psscalebox{0.8}{$\alpha$}}
\rput[t](3,-0.1){\psscalebox{0.8}{$\alpha^\text{crit}$}}
\rput[tr](-0.1,-0.1){\psscalebox{0.8}{$0$}}
\end{pspicture}
\caption{Bifurcation diagram for the angle $\theta$ of the asymmetric core part in the case $\lambda>2$, depending on the global wall angle $\alpha$.}
\label{bifurc}
\end{figure}

In other words: In sufficiently thin ferromagnetic films, corresponding to $0<\lambda<2$ small, only the symmetric N\'eel wall is the stable minimizer of the reduced model. In thicker films, on the other hand, i.e. for $\lambda>2$, there exists a critical wall angle $\alpha^\text{crit}$ at which the symmetric N\'eel wall, which is a minimizer up to that point, becomes unstable with respect to perturbations that nucleate an asymmetric core. With decreasing field, i.e. increasing angle~$\alpha$, the core wall angle grows.

Hence, a supercritical bifurcation as in Figure~\ref{bifurc} is at the origin of the cross-over from symmetric to asymmetric N\'eel wall (with extended tails).

\subsection{Discussion on methods}

\input{suggestion_gamma_exp}

\subsubsection*{Implicit function theorem}
One expects that the expansion \eqref{eq:expansioneasym} of $\EA(\theta)$ and its minimizers $m_\theta$ can also be derived by more standard methods such as an implicit function theorem, applied to the Euler-Lagrange equation $G_\theta(m_\theta,\lambda_\theta,p_\theta)=0$ of $\EA(\theta)$, which is formally given by (for simplicity we use the notation $\nabla p = (\partial_1p,0,\partial_3p)$)
\begin{align*}
G_\theta\bigl(m,\lambda,p\bigr) := 
\left(
\begin{gathered}
-\Delta m + \nabla p - \lambda m\\
\nabla \cdot m'\\
(\partial_3  m_{1}, \partial_3  m_{2}, m_{3})\big\vert_{\partial\Omega}\\
\tfrac{1}{2}(\lvert m \rvert^2 - 1)\\
m(\pm\infty,\cdot) - m^\pm_\theta
\end{gathered}
\right).
\end{align*}
The function $\lambda \colon \Omega\to\R$ is the Lagrange multiplier that corresponds to the unit-length constraint $\lvert m \rvert^2 =1$ in the variational problem, while the function $p\colon\Omega\to\R$ is the Lagrange multiplier that corresponds to the constraint $\nabla\cdot m'=0$ on the divergence of $m'$, just as the pressure in Stokes' equations.

Denote by $m_\theta\in X_0$ a curve of minimizers of $\EA(\theta)$ that smoothly depends on the parameter $\theta$ in a neighborhood of $\theta=0$. For $\theta=0$, we have $m_0=(1,0,0)$, which solves $G_0(m_0,\lambda_0,p_0)=0$ for -- a priori -- an arbitrary choice of the multipliers $\lambda_0$ and $p_0$, provided $\partial_1 p_0 = \lambda_0$ and $\partial_3 p_0=0$. However, it turns out that one has $p_0=\pi \ell$ with $\ell$ given in \eqref{eq:defell} (in fact, this means $\lambda_0=\mu$ as given in \eqref{eq:defmu}).

In the spirit of the implicit function theorem method, one identifies the first-order correction $(\delta m,\delta \lambda, \delta p)$ to $(m_0,\lambda_0,p_0)$, i.e. $(m_\theta,\lambda_\theta,p_\theta)=(m_0,\lambda_0,p_0)+\theta (\delta m,\delta \lambda, \delta p)+\so(\theta)$, by solving the linear equation
\[ 0= \tfrac{d}{d\theta}\big\vert_{\theta=0} G_\theta(m_\theta,\lambda_\theta,p_\theta) = \partial_\theta\vert_{\theta=0} G_\theta (m_0,\lambda_0,p_0) + DG_0(m_0,\lambda_0,p_0) (\delta m,\delta \lambda,\delta p) \]
for $(\delta m,\delta \lambda, \delta p)$. The differential of $G_0$ is given by
\begin{align*}
DG_0(m_0,\lambda_0,p_0)(\delta m, \delta \lambda, \delta p) = \left(
\begin{gathered}
-\Delta \delta m + \nabla \delta p - \lambda_0 \, \delta m - m_0 \, \delta \lambda\\
\nabla \cdot \delta m'\\
(\partial_3  \delta m_{1}, \partial_3  \delta m_{2}, \delta m_{3})\big\vert_{\partial\Omega}\\
m_0\cdot \delta m\\
\delta m(\pm\infty,\cdot)
\end{gathered}
\right),
\end{align*}
and the only non-vanishing term in $\partial_\theta\vert_{\theta=0} G_\theta(m_0,\lambda_0,p_0)$ is $\partial_{\theta}\vert_{\theta=0} m^\pm_{2,\theta} = \pm 1$.

However, according to our results we can expect to have two branches of solutions, so that $DG_0(m_0,\lambda_0,p_0)$ has a non-trivial kernel, containing the linear space spanned by the difference of the functions $m_2^*$ for $\sigma=\pm 1$ (see \eqref{eq:defm2star}). But even after restricting to a suitable subspace on which this degeneracy is ruled out, there is another degeneracy in the Lagrange multipliers: $\delta p$ and $\delta \lambda$ can be taken arbitrarily as long as $\partial_1\delta p = \delta \lambda$ and $\partial_3 \delta p = 0$. It is possible to (formally) identify $\delta\lambda=0$ and $\delta p \equiv \text{const}$ by taking into account also higher-order terms in the expansion of $(m_\theta,\lambda_\theta,p_\theta)$ around $\theta=0$, but this involves similar technical problems as the $\Gamma$-convergence approach that we have pursued.

%% file: wallenergy_asym.tex
\begingroup
  \makeatletter
  \providecommand\color[2][]{%
    \GenericError{(gnuplot) \space\space\space\@spaces}{%
      Package color not loaded in conjunction with
      terminal option `colourtext'%
    }{See the gnuplot documentation for explanation.%
    }{Either use 'blacktext' in gnuplot or load the package
      color.sty in LaTeX.}%
    \renewcommand\color[2][]{}%
  }%
  \providecommand\includegraphics[2][]{%
    \GenericError{(gnuplot) \space\space\space\@spaces}{%
      Package graphicx or graphics not loaded%
    }{See the gnuplot documentation for explanation.%
    }{The gnuplot epslatex terminal needs graphicx.sty or graphics.sty.}%
    \renewcommand\includegraphics[2][]{}%
  }%
  \providecommand\rotatebox[2]{#2}%
  \@ifundefined{ifGPcolor}{%
    \newif\ifGPcolor
    \GPcolorfalse
  }{}%
  \@ifundefined{ifGPblacktext}{%
    \newif\ifGPblacktext
    \GPblacktexttrue
  }{}%
  \let\gplgaddtomacro\g@addto@macro
  \gdef\gplbacktext{}%
  \gdef\gplfronttext{}%
  \makeatother
  \ifGPblacktext
    \def\colorrgb#1{}%
    \def\colorgray#1{}%
  \else
    \ifGPcolor
      \def\colorrgb#1{\color[rgb]{#1}}%
      \def\colorgray#1{\color[gray]{#1}}%
      \expandafter\def\csname LTw\endcsname{\color{white}}%
      \expandafter\def\csname LTb\endcsname{\color{black}}%
      \expandafter\def\csname LTa\endcsname{\color{black}}%
      \expandafter\def\csname LT0\endcsname{\color[rgb]{1,0,0}}%
      \expandafter\def\csname LT1\endcsname{\color[rgb]{0,1,0}}%
      \expandafter\def\csname LT2\endcsname{\color[rgb]{0,0,1}}%
      \expandafter\def\csname LT3\endcsname{\color[rgb]{1,0,1}}%
      \expandafter\def\csname LT4\endcsname{\color[rgb]{0,1,1}}%
      \expandafter\def\csname LT5\endcsname{\color[rgb]{1,1,0}}%
      \expandafter\def\csname LT6\endcsname{\color[rgb]{0,0,0}}%
      \expandafter\def\csname LT7\endcsname{\color[rgb]{1,0.3,0}}%
      \expandafter\def\csname LT8\endcsname{\color[rgb]{0.5,0.5,0.5}}%
    \else
      \def\colorrgb#1{\color{black}}%
      \def\colorgray#1{\color[gray]{#1}}%
      \expandafter\def\csname LTw\endcsname{\color{white}}%
      \expandafter\def\csname LTb\endcsname{\color{black}}%
      \expandafter\def\csname LTa\endcsname{\color{black}}%
      \expandafter\def\csname LT0\endcsname{\color{black}}%
      \expandafter\def\csname LT1\endcsname{\color{black}}%
      \expandafter\def\csname LT2\endcsname{\color{black}}%
      \expandafter\def\csname LT3\endcsname{\color{black}}%
      \expandafter\def\csname LT4\endcsname{\color{black}}%
      \expandafter\def\csname LT5\endcsname{\color{black}}%
      \expandafter\def\csname LT6\endcsname{\color{black}}%
      \expandafter\def\csname LT7\endcsname{\color{black}}%
      \expandafter\def\csname LT8\endcsname{\color{black}}%
    \fi
  \fi
  \setlength{\unitlength}{0.0500bp}%
  \begin{picture}(5668.00,3400.00)%
    \gplgaddtomacro\gplbacktext{%
      \csname LTb\endcsname%
      \put(814,704){\makebox(0,0)[r]{\strut{} 0}}%
      \put(814,1109){\makebox(0,0)[r]{\strut{} 5}}%
      \put(814,1514){\makebox(0,0)[r]{\strut{} 10}}%
      \put(814,1920){\makebox(0,0)[r]{\strut{} 15}}%
      \put(814,2325){\makebox(0,0)[r]{\strut{} 20}}%
      \put(814,2730){\makebox(0,0)[r]{\strut{} 25}}%
      \put(814,3135){\makebox(0,0)[r]{\strut{} 30}}%
      \put(946,484){\makebox(0,0){\strut{}$0$}}%
      \put(2027,484){\makebox(0,0){\strut{}}}%
      \put(3109,484){\makebox(0,0){\strut{}$\tfrac{\pi}{4}$}}%
      \put(4190,484){\makebox(0,0){\strut{}}}%
      \put(5271,484){\makebox(0,0){\strut{}$\tfrac{\pi}{2}$}}%
      \put(176,1919){\rotatebox{-270}{\makebox(0,0){\strut{}Exchange energy}}}%
      \put(3108,154){\makebox(0,0){\strut{}Wall angle $\theta$}}%
    }%
    \gplgaddtomacro\gplfronttext{%
    }%
    \gplbacktext
    \put(0,0){\includegraphics{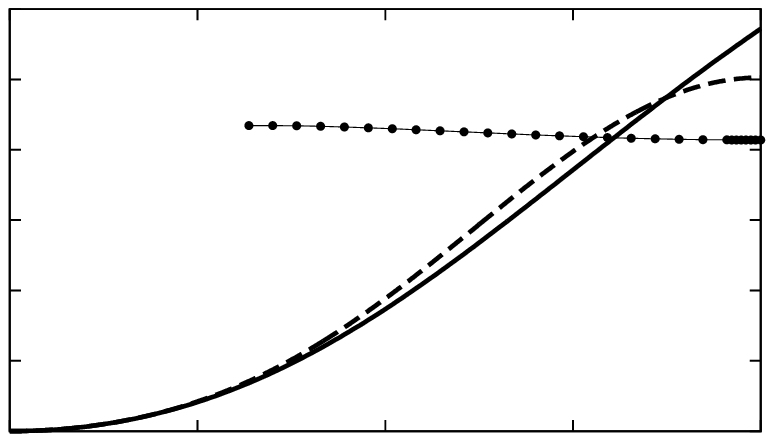}}%
    \gplfronttext
  \end{picture}%
\endgroup

%% file: suggestion_gamma_exp.tex
\subsubsection*{Asymptotic expansions by $\Gamma$-convergence}

Our proof of Theorems~\ref{thm:mine0}~and~\ref{thm:mine1} is based on the method of $\Gamma$-convergence, at the level of minimizers of $\EA(\theta)$. Proposition~\ref{prop:upperbd} gives an asymptotically optimal upper bound for the energy $\EA(\theta)$, while in Proposition~\ref{prop:compactnessasym} below we prove a compactness result for minimizers of $\EA(\theta)$ and an asymptotically sharp lower bound for their energy. In combination, these propositions immediately yield Theorems~\ref{thm:mine0}~and~\ref{thm:mine1} (see also Corollary~\ref{cor:compactnessasym} below).

This approach fits into the framework of asymptotic expansions by $\Gamma$-convergence (or the process of $\Gamma$-development), see e.g. \cite{AnzBal,BraTru}. In fact, for small $\theta$, \eqref{eq:defsffree} appears to be a scalar variational problem. To see this, we reformulate \eqref{eq:defsffree} in the following way: Let
\[ F_\theta(m) := \sin^{-2}\theta \int_\Omega \lvert \nabla m \rvert^2 dx \quad \text{for }m\in X_0\cap X^\theta. \]
We expect that -- at least in the regime $F_\theta(m) \leq C <\infty$ of small energy, which is the relevant one for the study of small-angle asymmetric N\'eel walls -- we can reconstruct the first and third component $m'$ of any such $m\in X_0\cap X^\theta$ from its second component $m_2$ by solving the Eikonal equation (cf. Section~\ref{sec:upper}):
\begin{align}\label{eq:eikasympt}
\lvert \nabla \psi \rvert^2 = 1-m_2^2 \;\; \text{in }\Omega, \quad \psi=0 \;\; \text{for }x_3=-1, \quad \psi=-2\cos\theta \;\; \text{for }x_3=1
\end{align}
for $\psi\in \dot{H}^2(\Omega)$ and setting $m':=\nabla^\perp \psi$. Hence, we may formally rewrite $F_\theta$ as functional on the second component $m_2$ alone, while the other components are slaved to $m_2$ via \eqref{eq:eikasympt}:
\[  F_\theta(m_2):= F_\theta\bigl(m_1,m_2,m_3\bigr), \quad \text{for }m':=\nabla^\perp \psi. \]
In order to define all functionals $F_\theta(m_2)$ on the same space, we introduce $f:=m_2/\sin\theta$, so that we may consider $F_\theta$ as function $F_\theta\colon \dot{H}^1(\Omega,\R) \to [0,\infty]$, given by
\begin{align*}
F_\theta(f) := 
\begin{cases}
\sin^{-2}\theta\int_\Omega\lvert \nabla m \rvert^2 dx, & \parbox[t]{0.5\textwidth}{if $f$ is such that $f(\pm\infty,\cdot)=\pm 1$ and the\\function $m_2:=\sin\theta \, f$ admits a unique\\extension $m'$ so that $m\in X_0\cap X^\theta$,\\[-2ex]}\\
+\infty, &\text{otherwise.}
\end{cases}
\end{align*}
Then, the leading-order $\Gamma$-limit $F^{(0)} \colon \dot{H}^1(\Omega,\R)\to[0,\infty]$ of $F_\theta$ is expected to be given by
\[ F^{(0)}(f) :=
\begin{cases}
\int_\Omega \lvert \nabla f \rvert^2 dx, &\text{if }f\in X^*,\\
+\infty, &\text{otherwise}.
\end{cases} \]
By Proposition~\ref{prop:minofe0} below, the energy $F^{(0)}$ admits only two minimizers $m_2^*$ of energy $E_0=4\pi$, which are related by the symmetry $x \leadsto -x$, $m_2\leadsto -m_2$. Since the whole energy $F_\theta$ is invariant under this symmetry, we cannot expect that any higher-order expansion will select one of the two minimizers, i.e. the first order has already locked the minimizer(s).

In line with the general results on $\Gamma$-expansions (see \cite[Eq.~(0.5)]{AnzBal} and \cite[Remark~1.8]{BraTru}), the next coefficient in the expansion of $\min F_\theta$ is just given in terms of a function $F^{(1)} \colon \dot{H}^1(\Omega,\R)\to[0,\infty]$
\[ F^{(1)}(f) :=
\begin{cases}
E_1, &\text{if }f=m_2^* \text{ for a $\sigma \in\{\pm 1\}$},\\
+\infty, &\text{otherwise},
\end{cases} \]
that attains a finite value only on the set of minimizers $\{m_2^* \,\vert\, \sigma=\pm 1 \}$ of the leading-order limit.
The coefficient $E_1=F^{(1)}(m_2^*)$ has been defined in \eqref{eq:defe1}.

Viewed as a formal $\Gamma$-expansion, we may write
\[ F_\theta \stackrel{\Gamma}{=} F^{(0)} + \sin^2\theta \, F^{(1)} + \so(\theta^2) \quad \text{weakly in } \dot{H}^1 \text{ as }\theta\tod 0. \]
In fact, at the first order, we conjecture that $F_\theta$ $\Gamma$-converges to $F^{(0)}$ in the weak $\dot{H}^1$-topology as $\theta \downarrow 0$. Note that Lemma~\ref{lem:compleadingm1} proves compactness and the liminf inequality of the $\Gamma$-convergence program. A variant of the construction argument in Proposition~\ref{prop:upperbd} might also carry over to general limit configurations $f\in X^*$, so that the limsup inequality should hold true.

%% file: propertiesE0E1.tex
\subsection{Identification of minimizers of $E_0$}\label{sec:idmine0}
In this section, we show that a minimizer of $E_0$ exists and has the form \eqref{eq:defm2star}. The main idea is to write an admissible function $f\in \Xs$ in \eqref{eq:defe0} as cosine series in $x_3$-direction and find a lower bound independent of $f$ with help of the Modica-Mortola trick. The lower bound will be attained by the two configurations given in \eqref{eq:p3m2}. 

\medskip
\begin{prop}\label{prop:minofe0}
The variational problem \eqref{eq:defe0} has exactly two minimizers $m_2^*$ given by \eqref{eq:defm2star}. In particular, $E_0=4\pi$.
\end{prop}

\begin{proof}[Proof of Proposition \ref{prop:minofe0}:] We split the proof into several steps:

\textbf{Step 1:} \textit{Fourier cosine representation of \eqref{eq:defe0}.}

Let $f \in \Xs$, i.e. be admissible in \eqref{eq:defe0}. Since $f\in \dot{H}^1(\Omega)$ and $f(\pm \infty, \cdot) = \pm 1$, there exist Fourier coefficients
$a_0\in \dot{H}^1(\R)$ and $a_n \in H^1(\R)$ for every $n \geq 1$ so that $f$ is represented as cosine series (in $x_3$-direction):
\begin{align*}
  f(x_1,x_3) - \tfrac{a_0(x_1)}{\sqrt{2}}= \sum_{n\geq 1} a_n(x_1) \cos\bigl(\tfrac{\pi}{2} \, n \, (x_3+1)\bigr) \quad \textrm{in }\, H^1(\Omega).
\end{align*}
Therefore, one computes:
\begin{align}
  \int_{\Omega} \bigl\lvert \nabla f \bigr\rvert^2 dx &= \int_{\R} \Bigl( \bigl\lvert \partial_1  a_0 \bigr\rvert^2 + \sum_{n \geq 1} \bigl\lvert \partial_1  a_n \bigr\rvert^2  + \sum_{n \geq 1} (\tfrac{\pi}{2} n)^2 \lvert a_n \rvert^2 \Bigr) \, dx_1,\label{eq:direncosine}
 \end{align}
and 
\begin{align}
\label{equ1}
\int_{\Omega_\pm} |f(x_1, x_3)\mp 1|^2\, dx=\int_{\R_\pm} |a_0\mp \sqrt{2}|^2\, dx_1+ \sum_{n\geq 1}\int_{\R_\pm} |a_n|^2\, dx_1.
\end{align}
Hence, $a_0 \mp \sqrt{2} \in H^1(\R_\pm)$, $a_0$, $a_n$ are continuous in $\R$, and we have $\lim_{x_1 \to \pm \infty}a_0(x_1)=\pm \sqrt{2}$ and $\lim_{|x_1| \to \infty} a_n(x_1) = 0$ for every $n\geq 1$.
Finally, one may write the constraint $\dashint_{-1}^1 f^2 dx_3 = 1$ in terms of the Fourier coefficients as:
\begin{align}
 1 = \dashint_{-1}^1 \! f^2(x_1, x_3) \, dx_3 = \tfrac{1}{2} \bigl( \lvert a_0(x_1) \rvert^2 + \sum_{n \geq 1} \lvert a_n(x_1) \rvert^2 \bigr) \quad \forall x_1 \in \R,\label{eq:avconstrcosine}
\end{align}
while $\bar{f}(0)=0$ is equivalent to $a_0(0)=0$.

\textbf{Step 2:} \textit{Lower bound: $E_0\geq4\pi$.}

Indeed, with the notation introduced at Step 1 for an arbitrary $f\in \Xs$, we define $g\colon\R\to\R_+$ by
\begin{align}
g(x_1) := \biggl(\sum_{n \geq 1} \lvert a_n(x_1) \rvert^2\biggr)^{\frac{1}{2}} \stackrel{\eqref{eq:avconstrcosine}}{=} \biggl(2 - \lvert a_0(x_1) \rvert^2\biggr)^{\frac{1}{2}} \quad \forall x_1\in \R.\label{eq:defg}
\end{align}
By \eqref{eq:direncosine} and \eqref{equ1}, we have that $g\in H^1(\R)$ since the Cauchy-Schwarz inequality yields:
\begin{align}
\lvert \partial_1  g \rvert^2 = \tfrac{(\sum_{n\geq 1} a_n \partial_1  a_n)^2}{\sum_{n\geq 1} \lvert a_n \rvert^2} \leq \sum_{n\geq 1} \lvert \partial_1  a_n \rvert^2
\quad \textrm{in }\, L^1(\R).\label{eq:partialg}
\end{align}
Let us introduce the continuous profile $$\ell=\tfrac{a_0}{\sqrt{2}}\in \dot{H}^1(\R).$$ Note that $\lim_{x_1 \to \pm \infty}\ell(x_1)=\pm 1$. By \eqref{eq:direncosine}, \eqref{eq:defg} and \eqref{eq:partialg}, we have the lower bound:
\begin{gather}\label{eq:mmargument}
\begin{aligned}
\int_{\Omega} \lvert \nabla f \rvert^2 dx & \geq \int_{\R} \Bigl( \lvert \partial_1  a_0 \rvert^2 + \lvert \partial_1  g \rvert^2 + (\tfrac{\pi}{2})^2 g^2 \Bigr) \, dx_1\\
  &= 2 \int_{\R} \Bigl( \bigl\lvert \partial_1  \ell \bigr\rvert^2 + \bigl\lvert \partial_1  \sqrt{1 - \ell^2} \bigr\rvert^2 + (\tfrac{\pi}{2})^2 \bigl(1-\ell^2\bigr) \Bigr) \, dx_1\\
  &= 2 \int_{\R} \Bigl( \tfrac{\lvert \partial_1  \ell \rvert^2}{1-\ell^2} + (\tfrac{\pi}{2})^2 \bigl(1-\ell^2\bigr) \Bigr) \, dx_1\\
 &\geq 2 \pi \int_{\R} \bigl\lvert \partial_1  \ell \bigr\rvert \, dx_1 \geq 2 \pi \int_{\R} \partial_1  \ell \, dx_1 = 4 \pi,
\end{aligned}
\end{gather}
where we applied Young's inequality to obtain the first term in the last line.

\textbf{Step 3:} \textit{The configurations $m^*_2$ in \eqref{eq:defm2star} are the only minimizers of $E_0$.}

First of all, a direct computation shows that $m^*_2$ in \eqref{eq:defe0} belongs to $\Xs$ with $\int_{\Omega} \lvert \nabla m^*_2 \rvert^2 dx=4\pi$. In order to prove that these two configurations are the unique minimizers, let $f \colon \Omega \to \R$ be an arbitrary minimizer of $E_0$, i.e. let $f\in \Xs$ satisfy $\int_\Omega \lvert \nabla f \rvert^2 dx = 4\pi$. By \eqref{eq:mmargument}, this implies that $\partial_1 \ell \geq 0$ and
\begin{align}
  \partial_1 \ell &= \tfrac{\pi}{2} (1-\ell^2) \quad \textrm{in }\, \R,\label{eq:odeell}\\
  g^2 &= \sum_{n \geq 1} n^2 \lvert a_n \rvert^2 \quad \textrm{in }\, \R.\label{eq:highercoeffs}
\end{align}
Comparing \eqref{eq:highercoeffs} to \eqref{eq:defg}, one deduces that the Fourier modes $a_n$ with $n\geq 2$ vanish, and the unique solution of \eqref{eq:odeell} with $\ell(0) = 
a_0(0) = 0$ is given by $\ell$ as it has been defined in \eqref{eq:defell}, so that $$a_0(x_1)=\sqrt{2}\tanh(\tfrac{\pi}{2}x_1).$$
Therefore, \eqref{eq:defg} turns into the relation
$a_1^2(x_1) = 2\bigl(1-\tanh^2(\tfrac{\pi}{2}x_1)\bigr)>0$ for every $x_1\in \R$ and the continuity of $a_1$ implies the existence of $\sigma \in \{\pm 1\}$ such that
$$a_1(x_1)=\sigma\sqrt{2} \sqrt{1-\tanh^2(\tfrac{\pi}{2}x_1)} \quad \textrm{in }\, \R.$$
Hence, using $\cos(\tfrac{\pi}{2}(x_3+1)) = -\sin(\tfrac{\pi}{2} x_3)$, we conclude that $f$ has the form given in \eqref{eq:defm2star}. 
\end{proof}

\subsection{Derivation of the Euler-Lagrange equation of $E_0$}\label{sec:ele0}
Let $m_2^*\in \Xs$ be one of the minimizers of $E_0$.
We aim to prove that $m_2^*$ satisfies the equation
\begin{align}\label{eq:ele0}
\int_\Omega \nabla m^*_2 \cdot \nabla f \, dx = \int_\Omega m^*_2  \, f \, \mu(x_1) dx \quad \forall f \in \dot{H}^1(\Omega), \, \bar{f}(0)=0,
\end{align}
for the Lagrange multiplier $\mu$ as defined in \eqref{eq:defmu}.

Let $f \in \dot{H}^1(\Omega)$ be such that $\bar{f}(0)=0$ and $f=0$ if $\lvert x_1 \rvert\gg 1$. We define a variation $\gamma_t \in \Xs$ of $m_2^*$ for $\lvert t \rvert \ll 1$ as
\begin{align}\label{eq:gamma}
\gamma_t=\tfrac{m^*_2+t f}{\sqrt{\dashint_{-1}^1 \! (m^*_2+t f)^2 dx_3}} \quad \text{in }\Omega.
\end{align}
Since $\gamma_0 = m_2^*\in \Xs$ minimizes $E_0$, we deduce that
\begin{align*}
\tfrac{d}{dt}\big\vert_{t=0} \int_\Omega \lvert \nabla \gamma_t \rvert^2 dx = 0.
\end{align*}
Using the fact that 
\begin{align*}
0=\tfrac{d}{dx_1} \dashint_{-1}^1 (m^*_2)^2(\cdot, x_3)\,  dx_3 =2 \dashint_{-1}^1 m^*_2(\cdot, x_3) \tfrac{d}{dx_1}m_2^*(\cdot, x_3)\,  dx_3,
\end{align*}
a short computation yields that $m_2^*$ satisfies the Euler-Lagrange equation \eqref{eq:ele0}, provided the test functions have compact support.

In order to pass from compactly supported functions to arbitrary $f\in\dot{H}^1(\Omega)$ with $\bar{f}(0)=0$, we consider $f_n := \chi_n f$ for $\chi_n(x_1)=\chi(\tfrac{x_1}{n})$, where $\chi \colon \R \to [0,1]$ is a smooth, compactly supported cutoff function with $\chi\equiv 1$ on $[-1,1]$. Then $\nabla f_n \to \nabla f$ in $L^2(\Omega)$, hence, by Hardy's inequality \eqref{eq:hardy}, also $\mu^{\frac{1}{2}}(x_1) f_n\to \mu^{\frac{1}{2}}(x_1) f$ in $L^2(\Omega)$. Therefore, passing to the limit $n\tou \infty$ in \eqref{eq:ele0} for $f=f_n$, one deduces that indeed \eqref{eq:ele0} holds for every $f\in \dot{H}^1(\Omega)$ with $\bar{f}(0)=0$.

\subsection{The Hessian of $E_0$ and its spectral gap}\label{sec:hesse0}
As above, let $m_2^*\in \Xs$ be one of the minimizers of $E_0$ and $f \in \dot{H}^1(\Omega)$ such that $\bar{f}(0)=0$ and $f=0$ if $\lvert x_1 \rvert\gg 1$. Again, we consider the variation $\gamma_t$ of $m_2^*$ defined in \eqref{eq:gamma}, under the additional assumption that the function $f$ be tangential to the constraint $\dashint_{-1}^1 (m_2^*)^2 \, dx_3 = 1$. To this end, we require $f$ to satisfy $\int_{-1}^1 m_2^* \, f \, dx_3 = 0$.

Since $\gamma_0=m_2^*\in \Xs$ minimizes $E_0$, one has that 
\begin{align*}
2B(f,f):=\tfrac{d^2}{dt^2}\big\vert_{t=0} \int_{\Omega} \lvert \nabla \gamma_t \rvert^2 dx \geq 0.
\end{align*}
Using \eqref{eq:ele0} one can explicitly compute the bilinear form
\begin{align}
\label{eq:hesse0}
B(f,\tilde f) = \int_\Omega \nabla f \cdot  \nabla \tilde f  - \mu(x_1) \, f \tilde f\, dx.
\end{align}
Hence, the same density argument as in the previous section shows that we trivially have
\begin{align*}
B(f,f)\geq 0\text{ for every $f\in \dot{H}^1(\Omega)$ with $\bar{f}(0)=0$ and $\int_{-1}^1 m_2^* \, f \, dx_3 = 0$}.
\end{align*}

However, the Hessian in fact has a spectral gap as we prove below.

\medskip
\begin{prop}
\label{prop:spectralgap}
Let $m_2^*$ be given by \eqref{eq:defm2star} for some $\sigma\in \{\pm 1\}$. For all $f \in \dot{H}^1(\Omega)$ with
\begin{align}\label{eq:constraintsonf}
\bar{f}(0)=0 \quad\text{and}\quad \int_{-1}^1 m_2^* \, f \, dx_3 = 0,
\end{align}
we have
\begin{align}\label{eq:sghess}
B(f,f)=\int_\Omega \lvert \nabla f \rvert^2 - \mu(x_1) \, f^2 \, dx\geq \tfrac{1}{5} \int_\Omega \lvert \nabla f \rvert^2 + \mu\,f^2 \, dx,
\end{align}
where $\mu$ has been defined in \eqref{eq:defmu}. Moreover, $B(f,f)\geq \tfrac{\pi^2}{4} \lVert f \rVert^2_{L^2}$, where $\tfrac{\pi^2}{4}$ is sharp.
\end{prop}

\begin{proof}
We divide the proof into three steps:

\textbf{Step 1:} \textit{Any $f\in\dot{H}^1(\Omega)$ with \eqref{eq:constraintsonf} satisfies $f\in L^2(\Omega)$.}

Indeed, by the second constraint in \eqref{eq:constraintsonf}, we may apply Poincar\'e's inequality \eqref{eq:poincm1} to $m_2^* f$. Hence,
\begin{align*}
\int_{\Omega} f^2\, dx&\leq 2\int_{\Omega} f^2(m_2^*)^2 + f^2 \underbrace{(1-\lvert m_2^* \rvert)^2}_{\leq C\mu(x_1)}\, dx \stackrel{\mathclap{\substack{\eqref{eq:constraintsonf}\\\eqref{eq:poincm1}}}}{\leq} C \int_{\Omega} \underbrace{\lvert \partial_3 (fm_2^*) \rvert^2}_{\mathclap{\leq C(\lvert \partial_3 f \rvert^2 + f^2 \mu)}} + f^2 \, \mu \, dx\\
&\leq C \int_{\Omega} \lvert \nabla f \rvert^2+ f^2 \, \mu \, dx \leq  C \int_{\Omega} \lvert \nabla f \rvert^2 \, dx,
\end{align*}
where we used that $\lvert  m_2^* \rvert\leq C$ and $\lvert \partial_3  m_2^*\rvert\leq C \sqrt{\mu(x_1)}$ in $\Omega$.

\textbf{Step 2:} \textit{For all $f\in H^1(\Omega)$ with \eqref{eq:constraintsonf}, we have
\begin{align*}
B(f,f) \geq \tfrac{\pi^2}{4} \int_\Omega f^2 \, dx.
\end{align*}
The constant $\tfrac{\pi^2}{4}$ is sharp.
}

As in the proof of Proposition~\ref{prop:minofe0}, we represent $f\in H^1(\Omega)$ as a Fourier cosine series: There exist coefficients $a_n \in H^1(\R)$, $n\geq 0$, such that
\begin{align*}
f(x_1,x_3) = \tfrac{a_0(x_1)}{\sqrt{2}} + \sum_{n\geq 1} a_n(x_1) \cos\bigl(\tfrac{\pi}{2} n (x_3+1)\bigr) \quad \text{in } H^1(\Omega).
\end{align*}
The constraints \eqref{eq:constraintsonf} turn into
\begin{align}\label{eq:constraintsfourier}
a_0(0) = 0,\qquad \tfrac{a_0}{\sqrt{1-\ell^2}} = \sigma \tfrac{a_1}{\ell} =: g \in H^1(\R).
\end{align}
Then, one computes
\begin{align*}
(\partial_1 a_0)^2 &= (1-\ell^2)  (\partial_1 g)^2 + \tfrac{\pi^2}{4} (1-\ell^2) \ell^2 \, g^2 - \tfrac{\pi}{2} (1-\ell^2) \ell \, \partial_1 g^2,\\
(\partial_1 a_1)^2 &=  \ell^2 (\partial_1 g)^2 + \tfrac{\pi^2}{4} (1-\ell^2)^2 \, g^2 + \tfrac{\pi}{2} (1-\ell^2) \ell \, \partial_1 g^2,\\
\end{align*}
as well as $g^2 = a_0^2 + a_1^2$ and
\begin{align}\label{eq:coeffs0and1}
(\partial_1 a_0)^2 + (\partial_1 a_1)^2 + \tfrac{\pi^2}{4} a_1^2 &= (\partial_1 g)^2 + \tfrac{\pi^2}{4} g^2.
\end{align}
Therefore, using the cosine-series representation \eqref{eq:direncosine} of the exchange energy, and
\begin{align*}
\int_{-1}^1 \lvert f(\cdot, x_3) \rvert^2 dx_3 = \sum_{n\geq 0} \lvert a_n \rvert^2 \quad \text{in } L^1(\R),
\end{align*}
we find
\begin{align*}
B(f,f) &= \int_{\R} \biggl((\partial_1 a_0)^2  + (\partial_1 a_1)^2 - \mu \, a_0^2 + (\tfrac{\pi^2}{4}-\mu) a_1^2\\
&\qquad\qquad+ \sum_{n\geq 2} \Bigl( \underbrace{(\partial_1 a_n)^2 + \bigl( (\tfrac{\pi}{2} n)^2 - \mu \bigr) a_n^2}_{\geq 0 + \tfrac{\pi^2}{4} a_n^2} \Bigr) \biggr) \, dx_1\\
&\stackrel{\mathclap{\eqref{eq:coeffs0and1}}}{\geq} \; \int_\R (\partial_1 g)^2 - \mu \, g^2 \, dx_1 + \tfrac{\pi^2}{4} \underbrace{\int_\R g^2 + \sum_{n\geq 2} a_n^2 \, dx}_{\smash[b]{=\int_\Omega f^2 \, dx}}.
\end{align*}
This concludes Step~2, provided
\begin{align}\label{eq:hardymu}
\int_{\R} (\partial_1 g)^2 - \mu \, g^2 \, dx_1 \geq 0 \quad \forall g\in \dot{H}^1(\R) \text{ with } g(0)=0.
\end{align}
For the proof of \eqref{eq:hardymu}, note that by an approximation argument we may assume that $g$ is smooth and $g=0$ in a neighborhood of $x_1=0$. Then, \eqref{eq:hardymu} easily follows using the Hardy decomposition $g = \ell \tfrac{g}{\ell}$, which entails $\partial_1 g = \partial_1 \ell \, \frac{g}{\ell} + \ell \partial_1 (\frac{g}{\ell})$ and hence
\begin{align*}
(\partial_1 g)^2 &= \ell^2 \bigl(\partial_1 (\tfrac{g}{\ell})\bigr)^2 + (\partial_1\ell)^2 (\tfrac{g}{\ell})^2 + \overbrace{\partial_1 (\partial_1 \ell \tfrac{g^2}{\ell}) - (\partial_1 \ell)^2 (\tfrac{g}{\ell})^2 - \ell \partial_1^2 \ell(\tfrac{g}{\ell})^2}^{\smash[t]{=\ell \partial_1 \ell \partial_1 (\frac{g}{\ell})^2}}\\
&\geq \partial_1 (\partial_1 \ell \tfrac{g^2}{\ell}) - \ell \underbrace{\partial_1^2 \ell}_{\mathclap{=-\ell \frac{\pi^2}{2}(1-\ell^2) = -\ell \mu}} (\tfrac{g}{\ell})^2 = \partial_1 (\partial_1 \ell \tfrac{g^2}{\ell}) + \mu \, g^2.
\end{align*}
After integrating over $\R$, the first term on the right hand side vanishes.

Finally, we note that the spectral gap estimate would be saturated for a function $f$ that is defined in Fourier space by $a_n=0$ for $n\geq 2$ and $a_0,a_1$ such that the corresponding $g\in H^1(\R)$ satisfies equality in \eqref{eq:hardymu}. While $g=\ell$ saturates \eqref{eq:hardymu}, it is not in $H^1(\R)$. However, it can be approximated by compactly supported functions $\ell_n$ such that $B(\ell_n,\ell_n) \to 0$ as $n\tou \infty$. Hence, the constant $\tfrac{\pi^2}{4}$ is sharp.

\textbf{Step 3:} \textit{Conclusion}

By Step~2, we have
\begin{align*}
\int_\Omega  \lvert \nabla f \rvert^2 + f^2 \, \mu \, dx= B(f,f)+ 2 \underbrace{\int_\Omega f^2\, \mu \, dx}_{\leq 2 B(f,f)} \leq 5 B(f,f).
\end{align*}
\end{proof}

\subsection{Computation of $E_1$}\label{sec:valueofe1}
Observe that the two minimizers $m^*_2$ as in \eqref{eq:defm2star} and hence also the two possible definitions of $\hat{m}'$ are related\footnote{In case of $\hat{m}_3$ observe that $\int_{-1}^s \hat{m}_1\,dx_3 = -\int_s^1 \hat{m}_1\,dx_3$ for any $s\in[-1,1]$.} via the symmetry $x_3 \leadsto -x_3$. By definition \eqref{eq:defe1}, the energy $E_1$ is invariant under this transformation. Therefore, we may restrict our attention to the case $\sigma = 1$.

As before, we denote $\ell(x_1) := \tanh(\tfrac{\pi}{2}x_1)$. Then
\begin{align}\label{eq:valm2}
m^*_2(x_1,x_3) = \ell(x_1) + \sqrt{2} \sqrt{1-\ell^2(x_1)} \sin(\tfrac{\pi}{2}x_3),
\end{align}
and using the relation $\ell' = \tfrac{\pi}{2}(1-\ell^2)$ one can compute:
\begin{align}
\partial_1  m_2^* &\stackrel{\eqref{eq:valm2}}{=} \tfrac{\pi}{2} (1-\ell^2) - \tfrac{\pi}{\sqrt{2}} \ell \sqrt{1-\ell^2} \sin(\tfrac{\pi}{2}x_3),\label{eq:p1m2}\\
\partial_3  m^*_2 &\stackrel{\eqref{eq:valm2}}{=} \tfrac{\pi}{\sqrt{2}} \sqrt{1-\ell^2} \cos(\tfrac{\pi}{2} x_3).\label{eq:p3m2}
\end{align}

Hence,
\begin{align*}
\dashint_{-1}^1 \lvert \partial_1  m_2^* \rvert^2 dx_3 &\stackrel{\eqref{eq:p1m2}}{=} \tfrac{\pi^2}{4} (1-\ell^2)^2 + \tfrac{\pi^2}{4} \ell^2 (1-\ell^2) = \tfrac{\pi^2}{4} (1-\ell^2),\\
\dashint_{-1}^1 \lvert \partial_3  m_2^* \rvert^2 dx_3 &\stackrel{\eqref{eq:p3m2}}{=} \tfrac{\pi^2}{4} (1-\ell^2),
\end{align*}
which entails \eqref{eq:defmu}:
\begin{align}\label{eq:lagrangemult}
\mu = \dashint_{-1}^1 \lvert \nabla m^*_2 \rvert^2 dx_3 = \tfrac{\pi^2}{2} (1-\ell^2).
\end{align}

Moreover, we have
\begin{align}
\hat{m}_1 &\stackrel{\eqref{eq:defhatm}}{=} \tfrac{1-(m^*_2)^2}{2} \stackrel{\eqref{eq:valm2}}{=} \tfrac{1}{2} (1-\ell^2)\cos(\pi x_3) - \sqrt{2} \ell \sqrt{1-\ell^2} \sin(\tfrac{\pi}{2}x_3),\label{eq:valm1hat}\\
\partial_3  \hat{m}_3 &\stackrel{\eqref{eq:defhatm}}{=} -\partial_1  \hat{m}_1 \stackrel{\eqref{eq:valm1hat}}{=} \tfrac{\pi}{2} \ell (1-\ell^2) \cos(\pi x_3) + \tfrac{\pi}{\sqrt{2}} \sqrt{1-\ell^2} (1-2\ell^2) \sin(\tfrac{\pi}{2} x_3),\label{eq:valp3m3hat}\\
\partial_3  \hat{m}_1 &\stackrel{\eqref{eq:valm1hat}}{=} -\tfrac{\pi}{2} (1-\ell^2) \sin(\pi x_3) - \tfrac{\pi}{\sqrt{2}} \ell \sqrt{1-\ell^2} \cos(\tfrac{\pi}{2} x_3),\label{eq:valp3m1hat}
\end{align}
and
\begin{align}
\hat{m}_3 &= \int_{-1}^{x_3} \partial_3  \hat{m}_3 \, dy_3\notag\\
&\stackrel{\mathclap{\eqref{eq:valp3m3hat}}}{=} \;\, \tfrac{1}{2}\ell(1-\ell^2) \sin(\pi x_3) - \sqrt{2} \sqrt{1-\ell^2}(1-2\ell^2) \cos(\tfrac{\pi}{2} x_3),\label{eq:valm3hat}\\
\partial_1  \hat{m}_3 &\stackrel{\mathclap{\eqref{eq:valm3hat}}}{=} \;\, \tfrac{\pi}{4} (1-\ell^2)(1-3\ell^2) \sin(\pi x_3) + \tfrac{\pi}{\sqrt{2}} \ell \sqrt{1-\ell^2} (5-6\ell^2) \cos(\tfrac{\pi}{2}x_3).\label{eq:valp1m3hat}
\end{align}

This yields
\begin{align*}
\int_{-1}^1 \hat{m}_1^2 \, dx_3 &\stackrel{\eqref{eq:valm1hat}}{=} \tfrac{1}{4} (1-\ell^2)^2 + 2 \ell^2 (1-\ell^2) = \tfrac{1}{8} (2+12\ell^2 - 14\ell^4),\\
\int_{-1}^1 \hat{m}_3^2 \, dx_3 &\stackrel{\eqref{eq:valm3hat}}{=} \tfrac{1}{4} \ell^2 (1-\ell^2)^2 + 2 (1-\ell^2) (1-2\ell^2)^2\\
&=\tfrac{1}{8}(16-78\ell^2+124\ell^4-62\ell^6),
\end{align*}
and
\begin{align}
\int_{-1}^1 \bigl\lvert \partial_1  \hat{m}_1 \bigr\rvert^2 \, dx_3 &\stackrel{\eqref{eq:valp3m3hat}}{=} \tfrac{\pi^2}{4} \ell^2 (1-\ell^2)^2 + \tfrac{\pi^2}{2} (1-\ell^2) (1-2\ell^2)^2\notag\\
&=\tfrac{\pi^2}{16}(1-\ell^2)(8-28\ell^2+28\ell^4),\label{eq:intp1m1}\\
\int_{-1}^1 \bigl\lvert \partial_3  \hat{m}_1 \bigr\rvert^2 \, dx_3 &\stackrel{\eqref{eq:valp3m1hat}}{=} \tfrac{\pi^2}{4} (1-\ell^2)^2 + \tfrac{\pi^2}{2} \ell^2 (1-\ell^2) = \tfrac{\pi^2}{16} (1-\ell^2) (4+4\ell^2),\notag\\
\int_{-1}^1 \bigl\lvert \partial_1  \hat{m}_3 \bigr\rvert^2 \, dx_3 &\stackrel{\eqref{eq:valp1m3hat}}{=} \tfrac{\pi^2}{16} (1-\ell^2)^2 (1-3\ell^2)^2 + \tfrac{\pi^2}{2} \ell^2 (1-\ell^2) (5-6\ell^2)^2\notag\\
&= \tfrac{\pi^2}{16}(1-\ell^2)(1+193\ell^2-465\ell^4+279\ell^6),\notag\\
\int_{-1}^1 \bigl\lvert \partial_3  \hat{m}_3 \bigr\rvert^2 \, dx_3 &\stackrel{\eqref{eq:intp1m1}}{=} \tfrac{\pi^2}{16} (1-\ell^2) (8-28\ell^2+28\ell^4).\notag
\end{align}

Using this and \eqref{eq:lagrangemult} in \eqref{eq:defe1}, and exploiting the relation $\ell' = \tfrac{\pi^2}{2}(1-\ell^2)$, we obtain
\begin{align*}
E_1 &= \tfrac{\pi}{8} \int_{\R} \bigl( 3 + 207\ell^2 - 519\ell^4 + 341\ell^6 \bigr) \ell' \, dx_1.
\end{align*}
By the change of variables $s=\ell(x_1)$ we arrive at
\begin{align*}
E_1 &= \tfrac{\pi}{8} \int_{-1}^1 \! \bigl( 3 + 207 s^2 - 519 s^4 + 341 s^6 \bigr) \, ds = \tfrac{148}{35} \pi.
\end{align*}

%% file: lower.tex
In this section, we prove asymptotic lower bounds for $E_\text{asym}(\theta)$ as $\theta\ll 1$. Together with the upper bound in Section~\ref{sec:upper}, they combine to Theorems~\ref{thm:mine0}~and~\ref{thm:mine1}.

Starting point is the following Lemma, which shows that the energy of any magnetization configuration $m_\theta\in X_0\cap X^\theta$ has the structure displayed in \eqref{eq:expansioneasym}:

\medskip
\begin{lem}\label{lem:calc_average_second}
Let $m_\theta \in X_0\cap X^\theta$ be admissible in the definition of $E_\text{asym}(\theta)$, and let $m^*_2\in \Xs$ satisfy the equation \eqref{eq:ele0} (e.g., $m^*_2$ could be one of the minimizers of $E_0$). Let $\hat{m}_\theta$ be such that
\begin{align}\label{eq:defhatmtheta}
  m_\theta = \left( \begin{smallmatrix} \cos\theta\\ \sin\theta \, m^*_2\\ 0 \end{smallmatrix} \right) + \sin^2\theta\,\hat{m}_\theta.
\end{align}
Then
\begin{align}\label{eq:avorth}
2\int_{-1}^1 m^*_2\, \hat{m}_{2,\theta}\, dx_3 = -\sin\theta \int_{-1}^1 \! \lvert \hat{m}_{\theta} \rvert^2 \, dx_3 \quad \text{for every }\, x_1\in \R
\end{align}
and
\begin{align}\label{eq:rewrittenenergy}
\int_\Omega \lvert \nabla m_\theta \rvert^2 dx = \sin^2\theta \int_\Omega \lvert \nabla m^*_2 \rvert^2 dx + \sin^4\theta \, B(\hat{m}_\theta,\hat{m}_\theta),
\end{align}
where
\begin{align}
B(\hat{m}_\theta,\hat{m}_\theta) := \int_\Omega \lvert \nabla \hat{m}_\theta \rvert^2 - \mu \lvert \hat{m}_\theta \rvert^2 dx = \sum_{i=1}^3 B(\hat{m}_{i,\theta},\hat{m}_{i,\theta})\label{def:B}.
\end{align}
\end{lem}
\begin{proof}
From $\lvert m_\theta \rvert^2 = 1$ and \eqref{eq:defhatmtheta} we obtain
\begin{align*}
1 = \cos^2\theta + 2\cos\theta \, \sin^2\theta\, \hat{m}_{1,\theta} + \sin^4\theta \,\hat{m}_{1,\theta}^2 +  m_{2,\theta}^2 + \sin^4\theta \, \hat{m}_{3,\theta}^2.
\end{align*}
Integrating over $(-1,1)$ in $x_3$ and using $\dashint_{-1}^1 \hat{m}_{1,\theta}\,dx_3 = 0$ (recall $\bar{m}_{1,\theta}\equiv \cos\theta$) yields
\begin{align*}
\sin^2\theta = 1-\cos^2\theta = \sin^2\theta\dashint_{-1}^1 \underbrace{(m_2^*+\sin \theta \, \hat{m}_{2,\theta})^2}_{\smash[b]{=m_{2,\theta}^2/\sin^2\theta}} dx_3 + \sin^4\theta \dashint_{-1}^1  \lvert \hat{m}_\theta' \rvert^2 \, dx_3.
\end{align*}
Dividing by $\sin^2\theta$, we obtain:
\begin{align*}
1 = \dashint_{-1}^1 \Bigl( (m^*_2)^2 + 2\sin\theta\, m^*_2\, \hat{m}_{2,\theta} + \sin^2\theta\,\lvert \hat{m}_\theta \rvert^2 \Bigr) \, dx_3.
\end{align*}
Now, using $\dashint_{-1}^1 (m_2^*)^2\,  dx_3 \equiv 1$, we conclude \eqref{eq:avorth}.

For \eqref{eq:rewrittenenergy}, we use the Euler-Lagrange equation \eqref{eq:ele0} associated to $m_2^*$ together with \eqref{eq:avorth}:
\begin{align*}
 \MoveEqLeft \sin^{-4}\theta \Bigl( \int_\Omega \lvert \nabla m_\theta \rvert^2 dx - \sin^2\theta \int_\Omega \lvert \nabla m^*_2 \rvert^2 dx \Bigr)\\
 &\stackrel{\eqref{eq:defhatmtheta}}{=} \int_\Omega \Bigl( \lvert \nabla \hat{m}_{\theta} \rvert^2 + 2\sin^{-1}\theta \, \nabla m^*_2 \cdot \nabla \hat{m}_{2,\theta} \Bigr)\, dx\\
  &\stackrel{\eqref{eq:ele0}}{=} \int_\Omega \Bigl(\lvert \nabla \hat{m}_\theta \rvert^2 + \, 2\sin^{-1}\theta \,\mu(x_1) m^*_2 \, \hat{m}_{2,\theta} \Bigr)\, dx\\
  &\stackrel{\eqref{eq:avorth}}{=} \int_\Omega \Bigl(\lvert \nabla \hat{m}_\theta \rvert^2 - \mu(x_1) \, \lvert \hat{m}_\theta \rvert^2\Bigr)\,  dx.\qedhere
\end{align*}
\end{proof}
Hence, for a family of minimizers $\{m_\theta\}_\theta$ of $\EA$ it remains to find a suitable $\mds$ -- a minimizer of \eqref{eq:defe0} for an appropriate choice of $\sigma\in\{\pm 1\}$ -- such that the corresponding term $B(\hat{m}_\theta,\hat{m}_\theta)$ (see \eqref{eq:defhatmtheta} and \eqref{eq:rewrittenenergy}) can be controlled. 
We are going to prove that $B(\hat{m}_{2,\theta},\hat{m}_{2,\theta})$ is negligible for $\theta\tod 0$ and $B(\hat{m}_\theta',\hat{m}_\theta') \to E_1$.

\subsection{Outline of the proof}\label{sec:outlineprooflb}
While obtaining an asymptotic lower bound on $E_\text{asym}(\theta)$ in terms of $E_0$ is almost straightforward (cf. Lemma~\ref{lem:compleadingm1} below), using the concentration-compactness type result \cite[Lemma~1]{dioreducedmodel13}, it requires more work to establish rigorously that the coefficient $E_1$ in \eqref{eq:expansioneasym} is given by \eqref{eq:defe1}.

Main problem here is that -- a priori -- the term $B(\hat{m}_\theta,\hat{m}_\theta)$ in \eqref{eq:rewrittenenergy} is lacking a sign, while the explicit construction in Section~\ref{sec:upper} below (see Proposition~\ref{prop:upperbd}) just provides an upper bound (note that, due to $\mu(0)=\tfrac{\pi^2}{2}$, even with optimal constants, Poincar\'e's and Hardy's inequalities yield a sign for $B$ only away from the origin, where $\mu$ is small).

We overcome this difficulty in the following way: Exploiting the relation 
\begin{align*}
m_{1,\theta}\approx\sqrt{1-m_{2,\theta}^2-m_{3,\theta}^2} \approx 1 - \sin^2\theta \tfrac{(m_2^*+\sin\theta\,\hat{m}_{2,\theta})^2 + (\sin\theta\,\hat{m}_{3,\theta})^2}{2} \quad\text{for }\theta\ll1,
\end{align*}
leading-order control over $m_{2,\theta}$ and $m_{3,\theta}$ in fact suffices to prove that $\{\hat{m}_{1,\theta}\}_\theta$ and $\{\partial_1 \hat{m}_{1,\theta}\}_\theta$ are bounded in $L^2_\text{loc}(\Omega)$ and $L^1(\Omega,\mu dx)$, respectively. Using the stray-field constraint $\partial_1 \hat{m}_{1,\theta}=-\partial_3 \hat{m}_{3,\theta}$ and a suitable interpolation inequality, $L^2(\Omega,\mu dx)$-control can be transferred to $\{\hat{m}_{3,\theta}\}_\theta$. Due to the exponential tails of the density $\mu$, local control of $\hat{m}_{1,\theta}$ suffices to render $B(\hat{m}_\theta',\hat{m}_\theta')$ harmless. In fact, $B(\hat{m}_\theta',\hat{m}_\theta')$ provides $H^1$-control of $\hat{m}_\theta'$.

For $B(\hat{m}_{2,\theta},\hat{m}_{2,\theta})$, we use that $B$ is the Hessian of the problem \eqref{eq:defe0} defining $E_0$ (cf. Section~\ref{sec:hesse0}) and satisfies a spectral gap inequality. Thus, it provides some control over tangent vectors to the ``manifold'' $\Xs$ at $m_2^*$. Unfortunately, $\hat{m}_{2,\theta}$ is not exactly tangential, but, by \eqref{eq:avorth}, it is at least formally an approximation to a tangent vector. Using part of the control of $\hat{m}_\theta'$ coming from $B(\hat{m}_{\theta}',\hat{m}_{\theta}')$, this suffices to prove that asymptotically, $B(\hat{m}_{2,\theta},\hat{m}_{2,\theta})$ still controls $\hat{m}_{2,\theta}$ in $\dot{H}^1(\Omega)$, and in particular is non-negative. This finally yields the lower bound for the expansion of $E_\text{asym}(\theta)$ at order $\sin^4\theta$.

The details of this procedure will be given in the following Section~\ref{sec:compactnessasym}, Proposition~\ref{prop:compactnessasym} summarizing the compactness results that we obtain for sequences of magnetization configurations $m_\theta\in X_0\cap X^\theta$ whose energy satisfies the upper bound of Proposition~\ref{prop:upperbd}.

\subsection{Compactness and lower bounds for minimizers of $E_\text{asym}(\theta)$}\label{sec:compactnessasym}
Our main result is the following:
\medskip
\begin{prop}\label{prop:compactnessasym}
For $0<\theta\ll 1$, let $m_\theta\in X_0\cap X^\theta$ be admissible in the definition of $E_\text{asym}(\theta)$ and satisfy the bound
\begin{align}\label{eq:secondorderupperbound}
\int_\Omega \lvert \nabla m_\theta \rvert^2 dx \leq E_0 \sin^2\theta + C\sin^4\theta
\end{align}
for some fixed positive constant $C>0$, where $E_0=4\pi$ is given in \eqref{eq:defe0}. Then, up to a suitable translation in $x_1$ and a subsequence, we have $m_{2,\theta}=\sin\theta \, m_2^* +\so(\sin\theta)$ in $\dot{H}^1(\Omega)$, for an $m_2^*$ as in \eqref{eq:defm2star}. Let $\hat{m}_\theta$ be as in \eqref{eq:defhatmtheta}.

Then, up to another subsequence, for $\mu=\mu(x_1)$ as in \eqref{eq:defmu},
\begin{align*}
\hat{m}_\theta' \wto \hat{m}'\quad \text{weakly in } H^1(\Omega) \text{ and strongly in } L^2(\Omega,\mu dx)\text{ as } \theta\tod 0,
\end{align*}
where $\hat{m}_1$ and $\hat{m}_3$ are given in \eqref{eq:defhatm}.

Moreover, we have the lower bounds
\begin{align}\label{eq:secondorderliminf}
\begin{aligned}
\MoveEqLeft\liminf_{\theta\tod 0} \;\sin^{-4}\theta \Bigl(\int_\Omega \lvert \nabla m_\theta \rvert^2 dx - \sin^2\theta \overbrace{\int_\Omega \lvert \nabla m_2^* \rvert^2 dx}^{\smash[t]{=E_0 = 4\pi\text{ by Prop.~\ref{prop:minofe0}}}} \Bigr)\\
&\geq \int_\Omega \lvert \nabla \hat{m}' \rvert^2 - \mu(x_1) \lvert \hat{m}' \rvert^2 \, dx = B(\hat{m}',\hat{m}') = E_1.
\end{aligned}
\end{align}
and, for some $\varepsilon>0$,
\begin{align}\label{eq:secondorderlimsup}
\begin{aligned}
\MoveEqLeft\limsup_{\theta\tod 0} \; \sin^{-4}\theta \Bigl( \int_\Omega \lvert \nabla m_\theta \rvert^2 dx - E_0\, \sin^2\theta - E_1 \,\sin^4\theta \Bigr)\\
&\geq \varepsilon \limsup_{\theta \tod 0} \int_\Omega \lvert \nabla \hat{m}_{2,\theta} \rvert^2 dx.
\end{aligned}
\end{align}
\end{prop}

\begin{cor}\label{cor:compactnessasym}
Provided the upper bound in Proposition~\ref{prop:compactnessasym} holds in the form
\begin{align}\label{eq:strongerupperbound}
\int_\Omega \lvert \nabla m_\theta \rvert^2 dx \leq E_0 \sin^2\theta + E_1\sin^4\theta + \so(\sin^4\theta),
\end{align}
we have strong convergence $\hat{m}_\theta \to ( \hat{m}_1,0,\hat{m}_3)$ in $\dot{H}^1(\Omega) \cap L^2(\Omega,\mu dx)$ as $\theta\tod 0$.
\end{cor}
\begin{proof}[Proof of Corollary~\ref{cor:compactnessasym}]
By \eqref{eq:secondorderlimsup}, Lemma~\ref{lem:calc_average_second} and the stronger upper bound \eqref{eq:strongerupperbound}, we have
\begin{align*}
0 \stackrel{\eqref{eq:strongerupperbound}}{=} \limsup_{\theta\tod 0} \Bigl( \bigl(B(\hat{m}_\theta',\hat{m}_\theta') -E_1\bigr) + B(\hat{m}_{2,\theta},\hat{m}_{2,\theta}) \Bigr) \stackrel{\eqref{eq:secondorderlimsup}}{\geq} \varepsilon \limsup_{\theta\tod 0}\int_\Omega \lvert \nabla \hat{m}_{2,\theta} \rvert^2 dx.
\end{align*}
Therefore, $\hat{m}_{2,\theta}\to 0$ in $\dot{H}^1(\Omega)$ and thus -- by Hardy's inequality \eqref{eq:hardy} -- in $L^2(\Omega,\mu dx)$ as $\theta \tod 0$. In particular, $B(\hat{m}_{2,\theta},\hat{m}_{2,\theta}) \to 0$. By \eqref{eq:secondorderliminf}, this yields
\begin{align*}
0 = \limsup_{\theta\tod 0} \bigl(B(\hat{m}_\theta',\hat{m}_\theta') -E_1\bigr) \geq \liminf_{\theta\tod 0} \bigl(B(\hat{m}_\theta',\hat{m}_\theta') -E_1\bigr) \stackrel{\eqref{eq:secondorderliminf}}{\geq} 0,
\end{align*}
i.e. $B(\hat{m}_\theta',\hat{m}_\theta') \to B(\hat{m}',\hat{m}')$ as $\theta\tod 0$. In view of the strong convergence of $\hat{m}_\theta'$ in $L^2(\Omega,\mu dx)$, this implies convergence $\int_\Omega \lvert \nabla \hat{m}_\theta' \rvert^2 dx \to \int_\Omega \lvert \nabla \hat{m}' \rvert^2 dx$. Therefore, $\hat{m}_\theta' \to \hat{m}'$ strongly in $\dot{H}^1(\Omega)$ as $\theta\tod 0$.
\end{proof}

The proof of Proposition~\ref{prop:compactnessasym} consists of several steps. In the first one, Lemma~\ref{lem:compleadingm1}, we prove a compactness result for sequences of magnetization configurations $m_\theta\in X_0\cap X^\theta$ with exchange energy of order $\sin^2\theta$ in the limit $\theta\tod 0$. 
\medskip
\begin{lem}\label{lem:compleadingm1}
For $0<\theta \ll 1$, let $m_\theta \in X_0\cap X^\theta$ be admissible in the definition of $E_\text{asym}(\theta)$ and satisfy the bound
\begin{align}\label{eq:apriorileading}
  \int_{\Omega} \lvert \nabla m_\theta \rvert^2 dx \leq C\sin^2\theta
\end{align}
for some fixed positive constant $C>0$. Define $m_{2,\theta}^*$ and $\hat{m}'_\theta$ by
\begin{align}\label{eq:defrescm}
  m_\theta = \left( \begin{smallmatrix} \cos\theta\\ \sin\theta \, m^*_{2,\theta}\\ 0 \end{smallmatrix} \right) + \sin^2\theta \, \left( \begin{smallmatrix} \hat{m}_{1,\theta}\\ 0\\ \hat{m}_{3,\theta} \end{smallmatrix} \right).
\end{align}
Then, up to translations in $x_1$ and for a subsequence in $\theta$, there exists $f \in \Xs$, i.e. admissible in $E_0$, such that for $\theta\tod 0$
\begin{itemize}
\item $m_{2,\theta}^* \wto f$ in $\dot{H}^1(\Omega)$,
\item $\hat{m}_{1,\theta} \to \frac{1-f^2}{2} =: \hat{m}_1$ in $L^p_\text{loc}(\Omega)$ for any $p\geq 1$,
\item $\sin\theta \, \hat{m}_{3,\theta} \wto 0$ in $H^1(\Omega)$.
\end{itemize}
\end{lem}
Here, the main issue consists in proving that there exists a (weak) limit configuration $f\in \Xs$ of the sequence $\{m_{2,\theta}^*\}_\theta$. While a change of sign in the second component can be established easily by applying the concentration-compactness result \cite[Lemma~1]{dioreducedmodel13}, one has to take into account also the convergence of the two remaining components $m_\theta'$ in order to obtain the correct numerical value $1$ in $\dashint_{-1}^1 f^2 \, dx_3=1$ and $f(\pm\infty,\cdot)=\pm 1$. In particular, this part deals with controlling $\hat{m}_{1,\theta}$ in $L^2_\text{loc}$ (see Section~\ref{sec:outlineprooflb} for the relevance of this).

In the next step, Lemma~\ref{lem:compsecondleadingm3}, we identify the (weak) limit $f$ as one of the minimizers $m_2^*$ of the variational problem defining $E_0$, provided the second-order upper bound \eqref{eq:secondorderupperbound} holds. Moreover, we improve the convergence of $m_\theta$ and derive the bounds on $\partial_1 \hat{m}_{1,\theta}$ in $L^1(\Omega,\mu dx)$ (again, cf. Section~\ref{sec:outlineprooflb}).

\medskip
\begin{lem}\label{lem:compsecondleadingm3}
For $0<\theta \ll 1$, let $m_\theta \in X_0\cap X^\theta$ be admissible in the definition of $E_\text{asym}(\theta)$ and satisfy the more restrictive bound
\begin{align}\label{eq:apriorisecondleading}
  \int_{\Omega} \lvert \nabla m_\theta \rvert^2 dx \leq E_0 \sin^2\theta + C \sin^4\theta
\end{align}
for some fixed positive constant $C>0$, and $E_0=4\pi$ as in \eqref{eq:defe0}. Then, the limit configuration $f\in \Xs$ in Lemma~\ref{lem:compleadingm1} is one of the minimizers $m_2^*$ of the variational problem defining $E_0$; moreover, adopting the notation in Lemma~\ref{lem:compleadingm1}, up to translations in $x_1$ and for a subsequence in $\theta$, we have in the limit $\theta\tod 0$:
\begin{itemize}
\item $m^*_{2,\theta} \to m^*_2$ in $\dot{H}^1(\Omega)$,
\item $\{\partial_1 \hat{m}_{1,\theta}\}_\theta$ is bounded in $L^1(\Omega,\mu dx)$,
\item $\sin\theta \,\hat{m}'_\theta \to 0$ in $H^1(\Omega)$.
\end{itemize}
\end{lem}
Note that under the more restrictive assumption \eqref{eq:apriorisecondleading} (which, in particular, is satisfied by minimizers $m_\theta$ of $E_\text{asym}(\theta)$, cf. Proposition \ref{prop:upperbd}), not only the weak limit $f$ can be identified, but also the convergence of the second component becomes strong in the $\dot{H}^1(\Omega)$-topology. 

This information then suffices to prove the asymptotic spectral gap inequality for the approximate tangent vectors $\hat{m}_{2,\theta}$:

\medskip
\begin{lem}\label{lem:hessest}
Let $m_\theta\in X_0\cap X^\theta$ satisfy the assumptions of Lemma~\ref{lem:compsecondleadingm3}. Define $\hat{m}_\theta$ as in \eqref{eq:defhatmtheta}, for $m_2^*$ given by Lemma~\ref{lem:compsecondleadingm3}.

Then, there exists a constant $\varepsilon>0$ such that in the limit $\theta\tod 0$ we have
\begin{align}\label{eq:hessest}
B(\hat{m}_{2,\theta},\hat{m}_{2,\theta}) \geq \varepsilon \int_{\Omega} \lvert \nabla \hat{m}_{2,\theta} \rvert^2 + \hat{m}_{2,\theta}^2 \, \mu \, dx - \so(1)  \int_\Omega \lvert \nabla \hat{m}'_\theta\rvert^2 dx,
\end{align}
where the Hessian $B$ has been defined in \eqref{def:B}.
\end{lem}
Combining all steps, we finally obtain Proposition~\ref{prop:compactnessasym}.

We will now prove Lemmas~\ref{lem:compleadingm1}~and~\ref{lem:compsecondleadingm3}.
\begin{proof}[Proof of Lemma~\ref{lem:compleadingm1}]
Define $m^*_{3,\theta} := \sin\theta \, \hat{m}_{3,\theta} = \tfrac{m_{3,\theta}}{\sin\theta}$.

\textbf{Step 1:} \textit{There exist $f\in \dot{H}^1(\Omega)$ and $m_3^* \in {H}^1(\Omega)$ such that $\bar{f}(0)=0$ and 
\begin{align}\label{eq:signf}
  \limsup_{x_1\to-\infty} \bar{f}(x_1) \leq 0 \quad \text{as well as} \quad \liminf_{x_1\to\infty} \bar{f}(x_1) \geq 0.
\end{align}
Moreover, up to a subsequence and translations in $x_1$-direction, as $\theta \tod 0$:
\begin{gather}\label{eq:convm2m3}
\begin{aligned}
  m^*_{2,\theta} &\wto f \quad\quad \text{in } \dot{H}^1(\Omega),\\
  \bar{m}^*_{2,\theta} &\to \bar{f} \qquad\text{locally uniformly,}\\
  m^*_{3,\theta} &\wto m_3^* \quad\; \text{in } H^1(\Omega).
\end{aligned}
\end{gather}
}

Indeed, by \eqref{eq:apriorileading}, the families $\{m^*_{2,\theta}\}_\theta$ and $\{m^*_{3,\theta}\}_\theta$ are bounded in $\dot{H}^1(\Omega)$. 
Recall that for $m_\theta\in X_0\cap\Xt$ we have $\bar{m}_{1,\theta} = \cos\theta$ (cf. \eqref{deftheta}). Moreover, $m_{3,\theta}=0$ on $\partial\Omega$.
Thus, the Poincar\'e inequalities \eqref{eq:poincm1} and \eqref{eq:poincm3} imply
\[ \int_\Omega \lvert m_{1,\theta}-\cos\theta \rvert^2 + m_{3,\theta}^2 \, dx \leq C \int_\Omega \lvert \partial_3 m_\theta'\rvert^2 dx. \]
Similar estimates hold for the rescaled variants $\hat{m}_{1,\theta}$ and $\hat{m}_{3,\theta}$ given in \eqref{eq:defhatm}. For the second component $m_{2,\theta}$, due to $\bar{m}_{2,\theta}(0)=0$, Hardy's inequality \eqref{eq:hardy} leads to a control of the $L^2$-norm of $m_{2, \theta}$ on bounded subsets $\Omega_k=(-k,k)\times (-1, 1) \subset \Omega$ for every $k>0$: 
\begin{align*}
 \int_{\Omega_k} m_{2, \theta}^2 \,  dx=\int_{\Omega_k}\bigl|m_{2, \theta}- \bar{m}_{2,\theta}(0)\bigr|^2   dx \stackrel{\eqref{eq:hardy}}{\leq} C_k \int_{\Omega_k} \lvert \nabla m_{2, \theta} \rvert^2 dx.
\end{align*}
Therefore, we deduce that $\{m^*_{3,\theta}\}_\theta$ is bounded in $H^1(\Omega)$ while $\{m^*_{2,\theta}\}_\theta$ is bounded in $\dot{H}^1(\Omega)\cap H_\loc^1(\Omega)$. 

Now we will use a concentration-compactness result to ensure that $x_1$-translated configurations $\{m^*_{2,\theta}\}_\theta$ do satisfy the constraint \eqref{eq:signf} in the limit $\theta\downarrow 0$.
For that, let
\[ u_\theta=\bar{m}_{2,\theta}^*=\tfrac{\bar{m}_{2,\theta}}{\sin \theta}. \]
Since $m_{2,\theta}(\pm\infty, \cdot)=\pm \sin \theta$ and $\theta\in (0, \pi)$, we have that $u_\theta\mp 1\in H^1(\R_{\pm})$:
Indeed
\[ \int_{\R_+}|u_\theta-1|^2\, dx_1=\tfrac{1}{\sin^2 \theta} \int_{\R_+}|\bar{m}_{2,\theta}-\sin \theta|^2\, dx_1\leq \tfrac{1}{2\sin^2 \theta} \int_{\Omega_+} |{m}_{2,\theta}-\sin \theta|^2\, dx \stackrel{\eqref{convent}}{<}\infty \]
(similarly on $\R_-$).
Moreover, $\{u_\theta\}_\theta$ is uniformly bounded in $\dot{H}^1(\R)$ and satisfies 
\[ \limsup_{x_1\to -\infty} u_\theta(x_1)=-1<0 \text{ and } \liminf_{x_1\to\infty} u_\theta(x_1)=1 > 0 \, \, \text{for every } \theta\downarrow 0. \]
By Lemma~1 in \cite{dioreducedmodel13}, up to a subfamily in $\theta$, we obtain an admissible limit $u\in \dot{H}^1(\R)$ and zeros $x_{1,\theta}$ of $\bar{m}_{2,\theta}$ for $\theta\downarrow 0$ such that 
\[ \text{$\bar{m}_{2,\theta}^*(\cdot+x_{1,\theta})\to u$ locally uniformly in $\R$ and $u$ satisfies \eqref{eq:signf}}. \]
In particular, $u(0)=0$. 
Since $\{m^*_{3,\theta}(\cdot+x_{1,\theta}, \cdot)\}_\theta$ is bounded in $H^1(\Omega)$ and $\{m^*_{2,\theta}(\cdot+x_{1,\theta}, \cdot)\}_\theta$ is bounded in $\dot{H}^1(\Omega)\cap H_\loc^1(\Omega)$, there exist $f\in \dot{H}^1(\Omega, \R)$ and $\mts \in {H}^1(\Omega, \R)$ such that  
\begin{align*}
  m^*_{2,\theta}(\cdot+x_{1,\theta}, \cdot)&\wto f \quad \text{ in } \dot{H}^1(\Omega)\cap H_\loc^1(\Omega) \text{ as }\theta\tod 0,\\
  m^*_{3,\theta}(\cdot+x_{1,\theta}, \cdot) &\wto \mts \quad \text{ in } H^1(\Omega) \text{ as }\theta \tod 0.
\end{align*}
By Rellich's theorem, we know that $m^*_{2,\theta}(\cdot+x_{1,\theta}, \cdot) \to f$ strongly in $L^2_\loc(\Omega)$, which implies in particular that we can identify the $x_3$-average of the limit $\bar{f}=u$.

In order to simplify notation, we may w.l.o.g. assume that these hold without translation in $x_1$-direction.

\medskip
Before we can show that $f$ is admissible in $\Xs$ (see Steps 4 and 5) and $m_3^* = 0$ (see Step 3), we need to identify the limit of $\{\hat{m}_{1,\theta}\}_\theta$ as $\theta \tod 0$.

\textbf{Step 2:} \textit{For every $1\leq p < \infty$ we have $\hat{m}_{1,\theta} \to \tfrac{1-{f}^2-{(m_3^*)}^2}{2}$ in $L^p_\text{loc}(\Omega)$.}

Let $\Omega_k := [-k,k]\times[-1,1]$. By Rellich's embedding theorem, Step 1 yields strong convergence $m^*_{2,\theta} \to f$ and $m^*_{3,\theta}\to m_3^*$  in $L^p(\Omega_k)$ as $\theta\tod 0$. Set
\begin{align*}
h_\theta = (m^*_{2,\theta})^2+(m^*_{3,\theta})^2.
\end{align*}
On one hand, one has that 
\begin{align*}
h_\theta\to f^2+(m_3^*)^2=:h \quad\text{in $L^p(\Omega_k)$ for every $p\geq 1$}.
\end{align*}
On the other hand, denoting
\begin{align*}
M_{\rho, \theta}:=\{m_{1,\theta}\leq \rho\} \quad\text{for some }\rho\in[-1,1]\text{ and }\theta>0,
\end{align*}
the function $\hat{m}_{1,\theta}$ can be expressed in terms of $h_\theta$ as follows:
\begin{align*}
\hat{m}_{1,\theta} = \tfrac{m_{1,\theta}-\cos\theta}{\sin^2\theta} = \tfrac{\sqrt{1-\sin^2\theta \, h_\theta}\, \, -1}{\sin^2\theta} + \tfrac{1-\cos\theta}{\sin^2\theta}=:F_\theta(h_\theta) \quad \text{on } \Omega\setminus M_{0, \theta}.
\end{align*}
It is easy to check that $F_\theta(t)\to F(t)$ as $\theta\tod 0$ (assuming $\sin^2\theta \leq t^{-1}$ to make $F_\theta(t)$ well-defined) with $F(t)=\tfrac{1-t}{2}$.

In fact, one computes
\begin{align*}
\lvert F_\theta(t)-F(t) \rvert \leq \bigl\lvert \tfrac{1}{2}-\tfrac{1-\cos\theta}{\sin^2\theta} \bigr\rvert + \lvert t \rvert \bigl\lvert \tfrac{1}{2}-\tfrac{1}{1+\sqrt{1-t\sin^2\theta}} \bigr\rvert \leq C\sin^2\theta \, (1+t^2)
\end{align*}
as $\theta\tod 0$. Therefore, 
\begin{align*}
  \MoveEqLeft\int_{\Omega_k \setminus M_{0, \theta}} \Bigl\lvert \hat{m}_{1,\theta} - \tfrac{1-(m^*_{2,\theta})^2-(m^*_{3,\theta})^2}{2} \Bigr\rvert^p dx = \int_{\Omega_k\setminus M_{0, \theta}} \lvert F_\theta(h_\theta)-F(h_\theta) \rvert^p\, dx\\
  &\leq C^p\sin^{2p}\theta \int_{\Omega_k} (1+h_\theta^2)^p\, dx\to 0 \quad \text{as }\theta\tod 0.
\end{align*}
For the estimate on $M_{0, \theta}$, we will prove more generally that for arbitrary $p\geq 1$, $0\leq \rho < 1$ and $\theta\in(0,\tfrac{\pi}{2})$ such that $\tfrac{1+\rho}{2} < \cos\theta < 1$
\begin{align}\label{eq:smallmeasure}
  \lm^2(M_{\rho, \theta}) \leq C_{\rho,p} \sin^p\theta,
\end{align}
where $C_{\rho,p}>0$ is a constant depending only on $\rho$ and $p$.

Indeed, if $\frac{1+\rho}{2}<\cos\theta < 1$, we have
\begin{align}\label{eq:estonM}
1 \leq \tfrac{m_{1,\theta} - \cos\theta}{\rho-\cos\theta} \quad\text{on } M_{\rho,\theta}.
\end{align}
By Poincar\'e's inequality we know that $m_{1,\theta} -\cos \theta\in H^1(\Omega)$, so that for $p\geq 2$, Sobolev's embedding theorem $H^1(\Omega)\subset L^p(\Omega)$ yields
\begin{align*}
\bigl( \lm^2(M_{\rho,\theta}) \bigr)^\frac{1}{p} &= \Bigl(\int_{M_{\rho,\theta}} 1 \, dx \Bigr)^\frac{1}{p} \stackrel{\eqref{eq:estonM}}{\leq} \Bigl( \int_{M_{\rho,\theta}} \bigl\lvert \tfrac{m_{1,\theta}  - \cos\theta}{\rho - \cos\theta} \bigr\rvert^p dx \Bigr)^\frac{1}{p}\\
&\leq \tfrac{C_{p}}{\cos\theta-\rho} \Bigl( \int_{\Omega} \lvert \nabla m_{1,\theta}  \rvert^2 dx \Bigr)^\frac{1}{2} \leq C_{\rho,p} \sin\theta,
\end{align*}
as long as $\frac{1+\rho}{2}<\cos\theta<1$, e.g. for $\theta\in(0,\tfrac{\pi}{3})$, if $\rho=0$. Since we have now proven that \eqref{eq:smallmeasure} holds for $p\geq 2$, it immediately follows that \eqref{eq:smallmeasure} holds also for $p\in [1, 2)$ (because $\sin^2\theta\leq \sin^p\theta$ for every $\theta\in(0,\tfrac{\pi}{2})$). 

Therefore, applying \eqref{eq:smallmeasure} for $2p+1$ instead of $p$, we find, using $\lvert \hat{m}_{1,\theta} \rvert, \lvert h_\theta\rvert\leq \tfrac{C}{\sin^2 \theta}$:
\begin{align*}
\int_{\Omega_k \cap M_{0, \theta}} \bigl\lvert \hat{m}_{1,\theta} - \tfrac{1-h_\theta}{2} \bigr\rvert^p dx\leq \tfrac{C_p}{\sin^{2p} \theta} \lm^2(M_{0,\theta}) \stackrel{\eqref{eq:smallmeasure}}{\to} 0 \quad \text{as }\theta\tod 0.
\end{align*}
So far we have obtained $\int_{\Omega_k} \lvert F_\theta(h_\theta) - F(h_\theta) \rvert^p dx \to 0$ as $\theta \tod 0$. Due to
\begin{align*}
\int_{\Omega_k} \lvert F(h_\theta)-F(h) \rvert^p\, dx\leq \tfrac{1}{2^p} \int_{\Omega_k} \lvert h_\theta-h \rvert^p\, dx\to 0,
\end{align*}
we conclude $\int_{\Omega_k} \lvert F_\theta(h_\theta) - F(h) \rvert^p dx \to 0$ in the limit $\theta\tod 0$.

\textbf{Step 3:} \textit{We have $m_3^*=0$. In particular, by Step 2, this yields $\hat{m}_{1,\theta} \to \tfrac{1-{f}^2}{2}$ in $L^p_\text{loc}(\Omega)$ for every $1\leq p < \infty$.}

Indeed, we will use $\nabla\cdot m_\theta' =0$ in $\Omega$ in order to show that $m^*_{3,\theta}$ vanishes in the limit $\theta\tod 0$. For that, observe that in terms of $\hat{m}_{1,\theta}$ and $m^*_{3,\theta}$, the divergence constraint turns into
\begin{gather}\label{eq:mag_asym_lvl1}
\left\{
\begin{aligned}
\sin\theta \, \partial_1  \hat{m}_{1,\theta} + \partial_3  m^*_{3,\theta} &= 0 \quad \text{in }\Omega,\\
m_{3,\theta}^* &= 0 \quad\text{on }\partial\Omega.
\end{aligned}\right.
\end{gather}
By Step 1 we have $m^*_{3,\theta} \wto m_3^*$ weakly in $H^1(\Omega)$, hence weakly in $H^\frac{1}{2}(\partial \Omega)$ by weak continuity of the trace operator. Therefore, $m_3^* = 0$ on $\partial\Omega$. After testing \eqref{eq:mag_asym_lvl1} with functions in $C^\infty_0(\Omega)$ we may therefore pass to the limit $\theta\tod 0$ and use Steps 1 and 2 to find 
\begin{align*}
\partial_3  m_3^* = 0 \text{ in }\mathcal{D}'(\Omega).
\end{align*}
Since $m_3^* \in H^1(\Omega)$, we conclude $m_3^*=0$ in $\Omega$.

\textbf{Step 4:} \textit{We have $\dashint_{-1}^1 \! (m^*_{2,\theta})^2 \, dx_3 \to 1$ a.e. in $x_1$. As a consequence, $\dashint_{-1}^1{f}^2 \, dx_3 \equiv 1$.}

Indeed, by the compact embedding $H_\text{loc}^1(\Omega) \to L^2(\{x_1\}\times[-1,1])$ for any fixed $x_1\in \R$, Step~1 implies that $\dashint_{-1}^1 (m^*_{2,\theta})^2 \, dx_3 \to \dashint_{-1}^1 f^2 \, dx_3$ a.e. in $x_1$. But Step~3 yields $0=\int_{-1}^1 \hat{m}_{1,\theta} \, dx_3 \to \dashint_{-1}^1 \bigl( 1 - {f}^2 \bigr) \, dx_3$ in $L^1_\text{loc}(\R)$ and thus a.e. in $x_1$. Since the function $x_1\mapsto \dashint_{-1}^1 {f}^2 \, dx_3$ is of class $W^{1,1}_\text{loc}(\R)$ and therefore continuous, we conclude $\dashint_{-1}^1 f^2 \, dx_3 = 1$ for all $x_1 \in \R$.

\textbf{Step 5:} \textit{We have $f(\pm\infty,\cdot)=\pm1$, which concludes the proof of $f\in \Xs$.}

We start by checking that $\lvert f \rvert(\pm\infty,\cdot)=1$. Indeed, since $\lvert f \rvert\in \dot{H}^1(\Omega)$, the function $x_3 \mapsto \lvert f \rvert(x_1,x_3)$ is of class $H^1\bigl((-1,1)\bigr)$ for almost every $x_1 \in \R$, hence continuous on almost every line $\{x_1\}\times(-1,1)$. Using Step~4, we deduce that for a.e. $x_1\in \R$, there exists an $x_3\in (-1,1)$ such that $\lvert f \rvert(x_1, x_3)=1$. Therefore, from Poincar\'e's inequality we deduce that
\begin{align*}
\int_{-1}^1 \bigl\lvert \lvert f \rvert(x_1, s)-1\bigr\rvert^2\, ds=\int_{-1}^1 \bigl\lvert \lvert f \rvert(x_1, s) - \lvert f \rvert(x_1, x_3)\bigr\rvert^2\, ds\leq C \int_{-1}^1 \lvert\partial_3  f \rvert^2(x_1, s)\, ds.
\end{align*}
Integrating in $x_1$, one obtains
\begin{align*}
\int_{\Omega} \bigl\lvert \lvert f \rvert - 1 \bigr\rvert^2 dx\leq C\int_{\Omega} \lvert \partial_3  f \rvert^2 \, dx,
\end{align*}
i.e. $\lvert f \rvert(\pm\infty,\cdot)=1$.

In particular,
\begin{align}\nonumber
  \int_{\R} \bigl\lvert \lvert \bar{f} \rvert - 1 \bigr\rvert^2 dx_1&=\tfrac{1}{2}  \int_{\Omega} \bigl\lvert \lvert \bar{f} \rvert - 1\bigr\rvert^2 dx\\
&\leq \int_\Omega \lvert f - \bar{f} \rvert^2 + \bigl\lvert \lvert f \rvert - 1 \bigr\rvert^2 \, dx \leq C\int_{\Omega} |\partial_{3} f |^2 \, dx <\infty,\label{eq:absm2bl2}
\end{align}
where we used $\bigl||f| - |\bar{f}|\bigr|\leq |f - \bar{f}|$ and the Poincar\'e-Wirtinger inequality \eqref{eq:poincm1}.

Due to $\||\bar{f}  |\|_{\dot{H}^1(\R)}=\|\bar{f}  \|_{\dot{H}^1(\R)}\leq \frac 1{\sqrt{2}} \|f \|_{\dot{H}^1(\Omega)}<\infty$, we have $|\bar{f}  |-1\in {H}^1(\R)$. Therefore,
\beq
\label{34}
\lvert \bar{f}(x_1)\rvert \to 1 \quad \textrm{ as } \quad \lvert x_1 \rvert \to \infty.\eeq
In order to conclude, we proceed as in \cite[Lemma~2]{dioreducedmodel13}: From \eqref{eq:signf} and \eqref{34} we deduce that $\lvert \bar{f} (x_1) \rvert = \bar{f} (x_1)$ and $\lvert \bar{f} (-x_1) \rvert = -\bar{f} (-x_1)$ if $x_1$ is sufficiently large, such that \eqref{eq:absm2bl2} translates into
  \begin{align*}
    \int_{\R_-} \lvert \bar{f} +1 \rvert^2 dx_1 + \int_{\R_+} \lvert \bar{f} -1 \rvert^2 dx_1 < \infty.
  \end{align*}
  This finally yields
  \begin{align*}
  \MoveEqLeft\int_{\Omega_-} \lvert f +1 \rvert^2 dx + \int_{\Omega_+} \lvert f - 1 \rvert^2 dx\\
 &\leq2 \int_{\Omega} \lvert f - \bar{f}  \rvert^2 dx + 4\int_{\R_-} \lvert \bar{f} + 1 \rvert^2 dx_1 + 4\int_{\R_+} \lvert \bar{f}  - 1 \rvert^2 dx_1 < \infty.\qedhere
  \end{align*}

\end{proof}

\begin{proof}[Proof of Lemma~\ref{lem:compsecondleadingm3}]
The proof of Lemma~\ref{lem:compsecondleadingm3} is essentially independent of that of Lemma~\ref{lem:compleadingm1}, once it is known that up to translations in $x_1$-direction there is a weak limit $f\in \Xs$ of $m^*_{2,\theta}$ in $\dot{H}^1(\Omega)$ that is admissible in $E_0$. We will now step by step use the additional information \eqref{eq:apriorisecondleading} to improve the compactness result from Lemma~\ref{lem:compleadingm1}.

\textbf{Step 1:} \textit{The limit $f=:m_2^*\in \Xs$ is a minimizer of $E_0$ and $m^*_{2,\theta} \to m^*_2$ in $\dot{H}^1(\Omega)$.}

Indeed, since $f\in \Xs$ is admissible in $E_0$ by Lemma~\ref{lem:compleadingm1} (in particular, $E_0 \leq \int_{\Omega} \lvert \nabla f \rvert^2 dx$), we only need to prove the corresponding lower bound $E_0 \geq \int_\Omega \lvert \nabla f \rvert^2\,dx$. Indeed, from \eqref{eq:apriorisecondleading} and Lemma~\ref{lem:compleadingm1} it follows
\begin{align}
\MoveEqLeft E_0 \leq \int_{\Omega} \lvert \nabla f \rvert^2 dx \leq \liminf_{\theta \tod 0} \int_\Omega \lvert \nabla m^*_{2,\theta} \rvert^2 dx \leq \limsup_{\theta \tod 0} \int_\Omega \lvert \nabla m^*_{2,\theta} \rvert^2 dx\label{eq:estm2min}\\
&\stackrel{\eqref{eq:defrescm}}{\leq} \limsup_{\theta \tod 0} \, \sin^{-2}\theta \int_\Omega \lvert \nabla m_\theta \rvert^2 dx \stackrel{\eqref{eq:apriorisecondleading}}{\leq} \limsup_{\theta \tod 0} \Bigl( E_0 + C\sin^2\theta \Bigr) = E_0.\notag
\end{align}
Therefore, $f$ is one of the minimizers $m_2^*$ of $E_0$ with representation \eqref{eq:defm2star}; moreover, \eqref{eq:estm2min} implies $\int_\Omega \lvert \nabla m^*_{2,\theta} \rvert^2 dx \to \int_\Omega \lvert \nabla m^*_2 \rvert^2 dx$ as $\theta\tod 0$, which yields strong convergence $m^*_{2,\theta} \to m^*_2$ in $\dot{H}^1(\Omega)$ as $\theta \tod 0$. 

\textbf{Step 2:} \textit{We have $\sin\theta \,\hat{m}'_\theta \to 0$ in $H^1(\Omega)$ as $\theta \tod 0$.}

Indeed, to obtain more information on the convergence of $\hat{m}'_\theta$, we combine Step~1 with \eqref{eq:apriorisecondleading}:
\begin{align*}
\sin^2\theta \int_\Omega \lvert \nabla \hat{m}'_{\theta} \rvert^2 dx \qquad &\stackrel{\mathclap{\eqref{eq:defrescm}\,\&\,\text{Step 1}}}{=} \qquad \underbrace{\sin^{-2}\theta \int_\Omega \lvert \nabla m_\theta \rvert^2 dx - E_0}_{\stackrel{\eqref{eq:apriorisecondleading}}{\leq} C \sin^2\theta} + \underbrace{\int_\Omega \lvert \nabla m^*_2 \rvert^2 - \lvert \nabla m^*_{2,\theta} \rvert^2 dx}_{\stackrel{\text{Step 1}}{=} \so(1)}\\
&\to 0 \quad \text{as } \theta \tod 0.
\end{align*}
Convergence of $\sin\theta \,\hat{m}'_\theta$ in $L^2(\Omega)$ again follows from the Poincar\'e inequalities \eqref{eq:poincm1} and \eqref{eq:poincm3}.

\textbf{Step 3:} \textit{We have $\int_\Omega \lvert \partial_1  \hat{m}_{1,\theta} \rvert \, \mu(x_1) dx \leq C<\infty$ uniformly in $0<\theta\ll 1$.}

Indeed, consider $M_{\frac{1}{2},\theta} := \{ m_{1,\theta} \leq \frac{1}{2} \}$. Using \eqref{eq:smallmeasure} for $\rho=\frac{1}{2}$, i.e.
\begin{align*}
 \lm^2(M_{\frac{1}{2},\theta})  \leq C \sin^2 \theta \quad\text{for }\theta\ll 1,
\end{align*}
the Cauchy-Schwarz inequality implies
\begin{align*}
\MoveEqLeft\int_{ M_{\frac{1}{2},\theta}} \lvert \partial_1  \hat{m}_{1,\theta} \rvert \underbrace{\mu(x_1)}_{\smash[b]{\leq \frac{\pi^2}{2}}} dx\leq \tfrac{\pi^2}{2} \left(\tfrac{\lm^2( M_{\frac{1}{2},\theta})}{\sin^2\theta}\right)^\frac{1}{2} \Bigl( \int_{M_{\frac{1}{2},\theta}} \lvert \partial_1  \bigl(\sin\theta\,\hat{m}_{1,\theta} \bigr) \rvert^2 dx \Bigr)^\frac{1}{2}\\
&\stackrel{\text{Step 2}}{\longrightarrow} 0 \quad \text{as }\theta \tod 0.
\end{align*}
It thus remains to treat the set $M^C_{\frac{1}{2},\theta} = \Omega \setminus M_{\frac{1}{2},\theta}$.

On $M^C_{\frac{1}{2},\theta}$, we have $m_{1,\theta} = \sqrt{1-m^2_{2,\theta}-m^2_{3,\theta}} \geq \frac 1 2$. Hence, we can estimate the derivative of $\hat{m}_{1,\theta}$ by using the Cauchy-Schwarz inequality and $\mu\leq \tfrac{\pi^2}{2}$:
\begin{gather}\label{eq:estpml2mu}
\begin{aligned}
\MoveEqLeft 
\int_{M^C_{\frac{1}{2},\theta}} \lvert \partial_1 \hat{m}_{1,\theta} \rvert  \, \mu(x_1) dx \stackrel{\eqref{eq:defrescm}}{=} \int_{M^C_{\frac{1}{2},\theta}} \Bigl\lvert \tfrac{m_{2,\theta} \partial_1  m_{2,\theta} + m_{3,\theta} \partial_1  m_{3,\theta}}{\sin^2\theta \sqrt{1-m^2_{2,\theta}-m^2_{3,\theta}}} \Bigr\rvert  \, \mu(x_1) dx\\
&\stackrel{\eqref{eq:defrescm}}{\leq} C \int_{M^C_{\frac{1}{2},\theta}} \Bigl( \lvert m^*_{2, \theta} \rvert^2 \mu(x_1)  + \bigl\lvert \partial_1  m^*_{2,\theta} \bigr\rvert^2 \Bigr)\, dx\\
&\qquad + C \int_{ M^C_{\frac{1}{2},\theta}} \Bigl(\bigl\lvert \sin\theta\,\hat{m}_{3,\theta}\bigr\rvert^2 + \bigl\lvert \partial_1  \bigl( \sin\theta \, \hat{m}_{3,\theta} \bigr)\bigr\rvert^2 \Bigr)\, dx\\
&\stackrel{\mathclap{\eqref{eq:poincm3}\,\&\,\eqref{eq:hardy}}}{\leq} \quad C \biggl(\int_\Omega \bigl\lvert \nabla m^*_{2,\theta} \bigr\rvert^2 + \bigl\lvert \nabla (\sin\theta \, \hat{m}_{3,\theta}) \bigr\rvert^2 \, dx \biggr)\stackrel{\text{Steps 1\&2}}{\leq} C.
\end{aligned}
\end{gather}
This concludes the proof of Lemma~\ref{lem:compsecondleadingm3}.
\end{proof}

We now turn to proving the spectral gap for approximate tangent vectors, which improves Proposition~\ref{prop:spectralgap} and in combination with the previous lemmata yields Proposition~\ref{prop:compactnessasym}.

\begin{proof}[Proof of Lemma~\ref{lem:hessest}]
Denote by $C>0$ a generic, universal constant, whose value may change from line to line.

\textbf{Step 1:} \textit{The function $g(x_1):=\dashint_{-1}^1 m_2^* \, \hat{m}_{2,\theta} \, dx_3$ satisfies}
\begin{align*}
\int_\Omega g^2 \, \mu \, dx &\leq C \int_\Omega g^2 \, \mu^\frac{1}{2} \, dx \stackrel{\eqref{eq:avorth}}{\leq} C \sin^2\theta \int_\Omega \lvert \hat{m}_\theta \rvert^4 \, \mu^\frac{1}{2} \, dx\\
&\leq C \sin^2\theta \Bigl( \int_\Omega \lvert \nabla \hat{m}_\theta \rvert^2 dx \Bigr)^2 = \so(1) \int_\Omega \lvert \nabla \hat{m}_\theta \rvert^2 dx.
\end{align*}
Indeed, while the inequalities on the first line are obvious due to the boundedness of $\mu$ and Jensen's inequality, the last equality holds since Lemma~\ref{lem:compsecondleadingm3} entails $\sin^2\theta \int_\Omega \lvert \nabla \hat{m}_\theta \rvert^2 dx = \so(1)$ as $\theta\tod 0$.

It remains to argue that $\int_\Omega \lvert \hat{m}_\theta \rvert^4 \, \mu^\frac{1}{2} \, dx \leq C\left(\int_\Omega \lvert \nabla \hat{m}_\theta \rvert^2 dx\right)^2$ holds for a constant $C>0$ independent of $\theta$: Indeed, by continuity of the Sobolev embedding $H^1\hookrightarrow L^4$, applied to $\hat{m}_\theta\mu^\frac{1}{8}$, we have
\begin{align*}
\MoveEqLeft \int_\Omega \lvert \hat{m}_\theta \rvert^4 \, \mu^\frac{1}{2} \, dx \leq C \Bigl( \int_{\Omega} \lvert \nabla (\hat{m}_\theta \mu^\frac{1}{8}) \rvert^2 + \lvert \hat{m}_\theta \rvert^2 \, \mu^\frac{1}{4} \, dx \Bigr)^2.
\end{align*}
Since $\mu$ decays exponentially, we have $\lvert \tfrac{d}{dx_1}\mu^\frac{1}{8}\rvert \lesssim \mu^{-\frac{7}{8}} \lvert \tfrac{d}{dx_1}\mu \rvert \lesssim \mu^\frac{1}{8}$. Hence, noting $\mu^\frac{1}{4} \lesssim (1+x_1^2)^{-1} \lesssim 1$, the Poincar\'e inequalities \eqref{eq:poincm1} and \eqref{eq:poincm3} (for $\hat{m}_\theta'$) and Hardy's inequality \eqref{eq:hardy} (for $\hat{m}_{2,\theta}$) yield
\begin{align*}
\MoveEqLeft \int_\Omega \lvert \hat{m}_\theta \rvert^4 \, \mu^\frac{1}{2} \, dx \leq C \Bigl( \int_{\Omega} \lvert \nabla \hat{m}_\theta \rvert^2 + \lvert \hat{m}_\theta \rvert^2 \, \mu^\frac{1}{4} \, dx \Bigr)^2 \leq C \Bigl( \int_\Omega \lvert \nabla \hat{m}_\theta \rvert^2 dx \Bigr)^2.
\end{align*}

\textbf{Step 2:} \textit{Projecting $\hat{m}_{2,\theta}$ onto the ``tangent space'' $\{f \;\big\vert\; \int_{-1}^1 m_2^*\, f \, dx_3 = 0\}$ yields
\begin{align*}
B(\hat{m}_{2,\theta},\hat{m}_{2,\theta}) \geq \tfrac{1}{5}\int_\Omega \hat{m}_{2,\theta}^2 \, \mu \, dx - \so(1)\int_\Omega \lvert \nabla \hat{m}_\theta \rvert^2 dx.
\end{align*}
}

Indeed, we define the projection $\psi$ of $\hat{m}_{2,\theta}$ onto $\{f \;\big\vert\; \int_{-1}^1 m_2^*\, f \, dx_3 = 0\}$ via
\begin{align*}
\psi(x):=\hat{m}_{2,\theta}(x)-g(x_1) m_2^*(x) \quad \text{for all }x\in\Omega,
\end{align*}
where $g=g(x_1)$ has been introduced in Step~1. Then one computes
\begin{align}\label{eq:decompB}
B(\hat{m}_{2,\theta},\hat{m}_{2,\theta}) = B({\psi},{\psi}) + 2B(\psi, g m_2^*) + B(g m_2^*, g m_2^*).
\end{align}
We estimate each term separately.

For the third term, using $\dashint_{-1}^1 (m_2^*)^2 dx_3 \equiv 1$ and $g^2\mu=\bigl(g^2\mu\bigr)(x_1)$, we trivially have
\begin{align}\label{eq:estimateB3}
B(gm_2^*,gm_2^*) = \int_\Omega \lvert \nabla (gm_2^*) \rvert^2 dx - \int_\Omega (gm_2^*)^2 \, \mu \, dx \geq -\int_\Omega g^2 \, \mu \, dx.
\end{align}
For the first term in \eqref{eq:decompB}, we observe that $\psi\in \dot{H}^1(\Omega)$ satisfies the assumptions of Proposition~\ref{prop:spectralgap}, i.e. $\bar{\psi}(0)=0$, and $\int_{-1}^1 m_2^* \, \psi \, dx_3 \equiv 0$. Hence, the spectral gap of $B$ yields
\begin{align}\label{eq:estimateB1}
B({\psi},{\psi}) &\geq \tfrac{1}{5} \int_\Omega \lvert \nabla \psi \rvert^2 + \psi^2 \, \mu \, dx = \tfrac{1}{5} \int_\Omega \lvert \nabla \psi \rvert^2 + \hat{m}_{2,\theta}^2 \, \mu \, dx - \tfrac{1}{5} \int_\Omega g^2 \, \mu \, dx
\end{align}
Note that in the last step we used $\mu=\mu(x_1)$ and that by definition we have $\int_{-1}^1 \psi^2 \, dx_3 \equiv \int_{-1}^1 \hat{m}_{2,\theta}^2-g^2 \, dx_3$.

Finally, we rewrite the second term in \eqref{eq:decompB}:
\begin{align*}
\MoveEqLeft B(\psi, g m_2^*) =\int_\Omega \nabla \psi\cdot \nabla (g m_2^*) - g m_2^* \psi \, \mu \, dx\\
&=\int_\Omega \nabla (\psi g)\cdot \nabla  m_2^* - (\psi g) m_2^*\, \mu + m_2^* \nabla \psi \cdot \nabla g - \psi \nabla g \cdot\nabla m_2^* \, dx\\
&\stackrel{\mathclap{\eqref{eq:ele0}}}{=}\,\int_{\R} \tfrac{d}{dx_1} g\int_{-1}^1  m_2^* \partial_1  \psi -\psi \partial_1  m_2^* \, dx_3 \, dx_1\\
&=-2\int_{\R} \tfrac{d}{dx_1} g \int_{-1}^1 \psi \partial_1  m_2^* \, dx_3 \, dx_1 = 2\int_{\Omega} g \partial_1 (\psi \partial_1  m_2^*) \, dx,
\end{align*}
where we used the fact that
\begin{align*}
0 = \tfrac{d}{dx_1} \int_{-1}^1\psi m_2^*\, dx_3 = \int_{-1}^1  m_2^* \partial_1  \psi +\psi \partial_1  m_2^* \, dx_3
\end{align*}
and, in the last step, integration by parts ($g(\pm \infty)=0$ since $\bar{\hat{m}}_{2,\theta}$ and thus also $g \in H^1(\R)$). By \eqref{eq:defm2star}, we have $\lvert \partial_1  m_2^* \rvert, \lvert \partial^2_1 m_2^* \rvert \leq C\mu^\frac{1}{2}$ in $\Omega$ so that Young's inequality and Hardy's inequality \eqref{eq:hardy} yield
\begin{align}\label{eq:estimateB2}
2 \lvert B(\psi,gm_2^*) \rvert \leq C \int_\Omega \lvert g \rvert ( \lvert \partial_1 \psi \rvert + \lvert \psi \rvert ) \mu^\frac{1}{2} \, dx \leq \tfrac{1}{5} \int_\Omega \lvert \nabla \psi \rvert^2 dx + C \int_\Omega g^2 \mu^\frac{1}{2} dx.
\end{align}
Thus, in view of Step~1, using \eqref{eq:estimateB3}-\eqref{eq:estimateB2} in \eqref{eq:decompB} concludes Step~2.

\textbf{Step 3:} \textit{Conclusion.}

Trivially, we have
\begin{align*}
\MoveEqLeft \int_\Omega \lvert \nabla \hat{m}_{2,\theta} \rvert^2  + \hat{m}_{2,\theta}^2 \, \mu \, dx = B(\hat{m}_{2,\theta},\hat{m}_{2,\theta}) + 2\int_\Omega \hat{m}_{2,\theta}^2 \, \mu \,dx\\
&\stackrel{\text{Step~2}}{\leq} 11 B(\hat{m}_{2,\theta},\hat{m}_{2,\theta}) + \so(1) \int_\Omega \lvert \nabla \hat{m}_\theta \rvert^2 dx.
\end{align*}
This yields \eqref{eq:hessest}.
\end{proof}

\begin{proof}[Proof of Proposition~\ref{prop:compactnessasym}]
By assumption, \eqref{eq:apriorisecondleading} is satisfied so that we may apply Lemma~\ref{lem:compleadingm1} and Lemma~\ref{lem:compsecondleadingm3} to the sequence $\{m_\theta\}_\theta$. This already yields $m_{2,\theta}=\sin\theta\,m_2^* + \so(\sin\theta)$ in $\dot{H}^1(\Omega)$ as $\theta\tod 0$ for a minimizer $m_2^*$ of \eqref{eq:defe0} as in \eqref{eq:defm2star}.

In the remainder of the proof, we first explain why $B(\hat{m}_\theta',\hat{m}_\theta')$ is harmless and apply Lemma~\ref{lem:hessest} to control $\hat{m}_\theta$ in $\dot{H}^1(\Omega)$. Then, simple lower-semicontinuity arguments yield the bounds \eqref{eq:secondorderliminf} and \eqref{eq:secondorderlimsup}.

\textbf{Step 1:} \textit{For every $\delta>0$, there exists a constant $C_\delta<\infty$, independent of $\theta\tod 0$, such that
\begin{align}\label{eq:mhatl2dmusmall}
\int_\Omega \lvert \hat{m}_\theta' \rvert^2 \, \mu \, dx \leq \delta \int_\Omega \lvert \nabla \hat{m}_\theta' \rvert^2 dx + C_\delta\quad\text{as }\theta\tod 0.
\end{align}
In particular, the sequence $\{\hat{m}_\theta\}_\theta$ is bounded uniformly in $\dot{H}^1(\Omega)$ and therefore -- up to a subsequence -- converges weakly in $\dot{H}^1(\Omega)$ as well as strongly in $L^2(\Omega,\mu dx)$. The weak limit of $\hat{m}_\theta'$ is given by $\hat{m}'$ as defined in \eqref{eq:defhatm}.
}

We start by bounding $\hat{m}_{1,\theta}$ in $L^2(\Omega,\mu dx)$: Indeed, for $L\gg 1$ such that we have $\sup_{\lvert x_1 \rvert > L}\mu\leq \delta \tfrac{\pi^2}{4}$, Poincar\'e's inequality in $x_3$ yields
\begin{align*}
\int_{\{\lvert x_1 \rvert > L\}} \hat{m}_{1,\theta}^2 \, \mu\, dx \leq \delta \int_{\Omega\setminus \Omega_L} \lvert \nabla \hat{m}_{1,\theta} \rvert^2 dx.
\end{align*}
Thus, the first half of \eqref{eq:mhatl2dmusmall} follows after recalling that by Lemma~\ref{lem:compleadingm1} the sequence $\{\hat{m}_{1,\theta}\}_\theta$ is bounded uniformly in $L^2_\text{loc}(\Omega)$.

For $\hat{m}_{3,\theta}$, we can argue via the interpolation inequality\footnote{%
W.l.o.g. $u:=\hat{m}_{3,\theta}\in H^1_0(\Omega)$ is compactly supported and smooth. Then we have
\begin{align*}
u^2(x_1,x_3) \mu(x_1) &= \lvert u(x_1,x_3)-u(x_1,-1) \rvert \mu(x_1) \; \lvert u^2(x_1,x_3)-u^2(-\infty,x_3) \rvert^\frac{1}{2}\\
&\leq \int_{-1}^1 \lvert \partial_3 u(x_1,y_3) \rvert \, \mu(x_1) \, dy_3 \underbrace{\left(\int_\R \lvert \partial_1 u^2(y_1,x_3) \rvert \, dy_1\right)^\frac{1}{2}}_{\lesssim \left( \int_\R \lvert \partial_1 u \rvert^2 dy_1 \int_\R u^2 dy_1\right)^\frac{1}{4}}
\end{align*}
Integrating over $\Omega$ and applying Cauchy-Schwarz' and Poincar\'e's inequality yields the result.
}
\begin{align*}
\int_{\Omega} \hat{m}_{3,\theta}^2 \, \mu \, dx &\leq C \int_\Omega \lvert \underbrace{\partial_3 \hat{m}_{3,\theta}}_{\mathclap{=-\partial_1 \hat{m}_{1,\theta}}} \rvert \, \mu \,dx \Bigl( \int_{\Omega} \lvert \nabla \hat{m}_{3,\theta} \rvert^2 dx\Bigr)^\frac{1}{2}.
\end{align*}
Since by Lemma~\ref{lem:compsecondleadingm3} the sequence $\{\partial_1\hat{m}_{1,\theta}\}_\theta$ is bounded uniformly in $L^1(\Omega,\mu dx)$, Young's inequality yields the second half of \eqref{eq:mhatl2dmusmall}.

Regarding boundedness of $\{\hat{m}_\theta\}_\theta$, we have
\begin{align*}
\infty &> C\stackrel{\eqref{eq:secondorderupperbound}\,\&\,\text{Lem.~\ref{lem:calc_average_second}}}{\geq} B(\hat{m}_\theta,\hat{m}_\theta) = \int_\Omega \lvert \nabla \hat{m}_\theta' \rvert^2 dx - \int_\Omega \lvert \hat{m}_\theta' \rvert^2 \, \mu \, dx + B(\hat{m}_{2,\theta},\hat{m}_{2,\theta})\\
&\geq \bigl(1-\delta-\so(1)\bigr) \int_\Omega \lvert \nabla \hat{m}_\theta' \rvert^2 dx - C_\delta + \varepsilon \int_\Omega \lvert \nabla \hat{m}_{2,\theta} \rvert^2 dx,
\end{align*}
where we applied \eqref{eq:mhatl2dmusmall} and Lemma~\ref{lem:hessest}.

Thus, for $\delta$ and $\theta$ sufficiently small, $\{\hat{m}_\theta\}_\theta$ is bounded uniformly in $\dot{H}^1(\Omega)$. In particular, we may extract a subsequence for which $\nabla \hat{m}_\theta$ converges weakly in $L^2(\Omega)$. Strong convergence of $\{\hat{m}_\theta\}_\theta$ in $L^2(\Omega,\mu dx)$ follows from Rellich's theorem on sufficiently large but compact subsets and Poincar\'e's/Hardy's inequality on the (negligible) tails of the density $\mu$.

Comparing the weak limit of $\{ \hat{m}_{1,\theta}\}_\theta$ to the one obtained in Lemma~\ref{lem:compleadingm1}~and~\ref{lem:compsecondleadingm3}, and passing to the limit in the equation $\nabla \cdot \hat{m}_\theta' = 0$ in $\mathcal{D}'(\R^2)$, identifies $\hat{m}'$ as the function defined in \eqref{eq:defhatm} and thereby concludes Step~1.

\textbf{Step 2:} \textit{The lower bounds \eqref{eq:secondorderliminf} and \eqref{eq:secondorderlimsup} hold along the subsequence obtained in Step~1.}

Indeed, Lemma~\ref{lem:calc_average_second} shows
\begin{align*}
\sin^{-4}\theta \Bigl( \int_\Omega \lvert \nabla m_\theta \rvert^2 dx - \sin^2\theta \int_\Omega \lvert \nabla m_2^* \rvert^2 dx \Bigr) = B(\hat{m}_\theta',\hat{m}_\theta') + B(\hat{m}_{2,\theta},\hat{m}_{2,\theta}).
\end{align*}
By Step~1, we have 
\begin{align*}
\limsup_{\theta\tod 0} B(\hat{m}_\theta',\hat{m}_\theta') \geq \liminf_{\theta\tod 0} B(\hat{m}_\theta',\hat{m}_\theta') \geq B(\hat{m}',\hat{m}') = E_1.
\end{align*}
Using boundedness of $\{\hat{m}_\theta'\}_\theta$ in $\dot{H}^1(\Omega)$, Lemma~\ref{lem:hessest} yields
\begin{align*}
B(\hat{m}_{2,\theta},\hat{m}_{2,\theta}) \geq \varepsilon \int_\Omega \lvert \nabla \hat{m}_{2,\theta} \rvert^2 dx - \so(1) \quad\text{as }\theta\tod 0.
\end{align*}
Hence, both \eqref{eq:secondorderliminf} and \eqref{eq:secondorderlimsup} easily follow.
\end{proof}

%% file: upper.tex
We prove the upper bound for $\EA(\theta)$ that we relied on in the previous section:
\medskip
\begin{prop}\label{prop:upperbd}
For every $0<\theta\ll 1$ there exists a smooth function $m_\theta \in X_0\cap \Xt$ that is admissible in $\EA(\theta)$ and satisfies the estimate
\begin{align}\label{eq:upperbd}
\EA(\theta) \leq \int_\Omega \lvert\nabla m_\theta \rvert^2 dx \leq E_0 \sin^2\theta + E_1 \sin^4\theta + \co(\theta^6),
\end{align}
where $E_0$ and $E_1$ have been defined in \eqref{eq:defe0} and \eqref{eq:defe1}, respectively.
\end{prop}

\subsection{Strategy of the proof}
For minimizers of $\EA(\theta)$, $\theta\ll 1$, we have already identified the leading-order behavior of the second component as $\sin \theta \, m^*_2$ (see \eqref{eq:defm2star}). Therefore, we will begin the construction by making the ansatz $m_{2,\theta} := \sin\theta\, m^*_2$ and searching for configurations $m'_\theta$ that satisfy
\begin{align*}
\lvert m'_\theta \rvert^2 = 1-\sin^2\theta \,(m^*_2)^2 \quad \text{and} \quad m'_\theta(\pm\infty,\cdot)= (\cos \theta, 0).
\end{align*}
Since also \eqref{flux-clos} needs to be satisfied, the construction of $m'_\theta$ is equivalent to finding a ``stream function'' $\psi_\theta \colon \Omega \to \R$ such that $m'_\theta = \nabla^\perp \psi_\theta$ in $\Omega$ (hence $\nabla\cdot m'=0$ in $\Omega$) and $\psi_\theta$ is constant on each component of $\{x_3=\pm 1\}$ (leading to $m_3=0$ on $\partial \Omega$). The limit condition $\cos \theta=m_{1, \theta}=-\partial_3 \psi_\theta$ at $x_1=\pm\infty$ fixes the difference of those constants: Indeed, setting $\psi_\theta(\cdot, -1):=0$ on $\R$, we deduce that $\psi_\theta(\cdot, 1)=\int_{-1}^1 \partial_3 \psi_\theta(\cdot, s)\, ds=-2\cos \theta$. Therefore, $\psi_\theta$ has to satisfy:
\begin{align}
\lvert \nabla \psi_\theta \rvert^2 &= 1-\sin^2\theta \, (m_2^*)^2 \quad \text{in } \Omega,\label{eq:eik}\\
\psi_\theta &= \mathrlap{0}\phantom{1-\sin^2\theta \, (m_2^*)^2}\quad \text{on }\{x_3=-1\},\label{eq:bcb}\\
\psi_\theta &= \mathrlap{-2\cos\theta}\phantom{1-\sin^2\theta \, (m_2^*)^2}\quad \text{on }\{x_3=+1\}.\label{eq:bct}
\end{align}
However, there is no solution of the ``perturbed'' eikonal equation \eqref{eq:eik} with \eqref{eq:bcb} \& \eqref{eq:bct}.
Indeed, integrating \eqref{eq:eik} on $\{x_1\}\times (-1,1)$ for every $x_1\in \R$, we deduce:
\begin{align*}
\cos^2 \theta&=1-\sin^2\theta \dashint_{-1}^1 (m^*_2)^2(x_1, s)\, ds \stackrel{\eqref{eq:eik}}{=}\dashint_{-1}^1 |\nabla \psi_\theta|^2(x_1, s)\, ds\\
&\geq \dashint_{-1}^1 |\partial_3  \psi_\theta|^2(x_1, s)\, ds\geq \bigg(\dashint_{-1}^1 \partial_3  \psi_\theta(x_1, s)\, ds\bigg)^2=\cos^2 \theta.
\end{align*}
It means that $\partial_1  \psi_\theta\equiv 0$ and $\partial_3  \psi_\theta\equiv -\cos \theta$ in $\Omega$, i.e., 
\begin{align*} 
\psi_\theta(x)=-\cos \theta\,(x_3+1) \quad \text{in } \Omega,
\end{align*}
which is a contradiction to \eqref{eq:eik}.

Therefore, we will solve \eqref{eq:eik} imposing only one boundary condition at a time. Denote by
\begin{align*}
\psi_\theta^b \text{ and } \psi_\theta^t \text{ the solutions of \eqref{eq:eik} \& \eqref{eq:bcb} and \eqref{eq:eik} \& \eqref{eq:bct}, respectively}.
\end{align*}
We will then try to blend the two solutions $\psi_\theta^b$ and $\psi_\theta^t$ into one. This leads to a new difficulty: The blended function will not solve \eqref{eq:eik} anymore, but an equation of the form $ \lvert \nabla \psi \rvert^2 = 1-\sin^2\theta \, v^2$ with $v^2=(m^*_2)^2+\co(\theta^2)$. Yet the function $v^2$ will in general not have a curve $\gamma$ of zeros that connects $\{x_3=1\}$ with $\{x_3=-1\}$; note that the existence of such a curve is necessary to define $m_{2,\theta}:=\pm \sin \theta \, \lvert v \rvert$ with the suitable sign to the left and right of $\gamma$ so that \eqref{eq:bcinfty} holds. To fix this problem, we solve \eqref{eq:eik} a third time for suitably chosen boundary data on the curve of zeros of $m^*_2$, calling this solution $\psi^m_\theta$, and blend all three solutions into one, according to Figure~\ref{fig:combsolutions}.
\begin{figure}
\centering
\begin{pspicture}(-5,-2)(5,2)
  \psline(-5,2)(5,2)
  \psline(-5,-2)(5,-2)
  \parametricplot[plotstyle=line,plotpoints=51]{-1}{1}{90 t mul sin dup dup mul 2 mul 1 add sqrt div 2 sqrt neg mul 1 add dup neg 2 add div ln 3.1415 div 2 mul 2 t mul} 
  \rput(0,2){%
    \psline(0,-0.1)(0,0.1)
    \rput[b](0,0.15){\psscalebox{0.6}{$0$}}
  }
  \rput(0.75,2){%
    \psline(0,-0.1)(0,0.1)
    \rput[b](0,0.15){\psscalebox{0.6}{$\frac{3}{8}$}}
  }
  \rput(2,2){%
    \psline(0,-0.1)(0,0.1)
    \rput[b](0,0.15){\psscalebox{0.6}{$1$}}
  }
  \rput(4,2){%
    \psline(0,-0.1)(0,0.1)
    \rput[b](0,0.15){\psscalebox{0.6}{$2$}}
  }
  \rput(-0.75,2){%
    \psline(0,-0.1)(0,0.1)
    \rput[b](0,0.15){\psscalebox{0.6}{$-\frac{3}{8}$}}
  }
  \rput(-2,2){%
    \psline(0,-0.1)(0,0.1)
    \rput[b](0,0.15){\psscalebox{0.6}{$-1$}}
  }
  \rput(-4,2){%
    \psline(0,-0.1)(0,0.1)
    \rput[b](0,0.15){\psscalebox{0.6}{$-2$}}
  }
  \psline[linestyle=dotted](0,2)(0,-2)
  \rput[l](5.1,2){\psscalebox{0.6}{$1$}}
  \rput[l](5.1,1.5){\psscalebox{0.6}{$\frac{3}{4}$}}
  \rput[l](5.1,1){\psscalebox{0.6}{$\frac{1}{2}$}}
  \rput[r](-5.1,-1){\psscalebox{0.6}{$-\frac{1}{2}$}}
  \rput[r](-5.1,-1.5){\psscalebox{0.6}{$-\frac{3}{4}$}}
  \rput[l](5.1,-2){\psscalebox{0.6}{$-1$}}
  \psline(5,1.5)(-0.75,1.5)(-0.75,1)(5,1)
  \psline(-5,-1.5)(0.75,-1.5)(0.75,-1)(-5,-1)
  \psline(-0.75,1)(-0.75,-1)
  \psline(0.75,1)(0.75,-1)
  \psset{linewidth=0.005}
  \rput(-5,-2){\psline(0,0)(0.5,0.5)}
  \rput(-4.5,-2){\psline(0,0)(0.5,0.5)}
  \rput(-4,-2){\psline(0,0)(0.5,0.5)}
  \rput(-3.5,-2){\psline(0,0)(0.5,0.5)}
  \rput(-3,-2){\psline(0,0)(0.5,0.5)}
  \rput(-2.5,-2){\psline(0,0)(0.5,0.5)}
  \rput(-2,-2){\psline(0,0)(0.5,0.5)}
  \rput(-1.5,-2){\psline(0,0)(0.5,0.5)}
  \rput(-1,-2){\psline(0,0)(0.5,0.5)}
  \rput(-0.5,-2){\psline(0,0)(0.5,0.5)}
  \rput(0,-2){\psline(0,0)(0.5,0.5)}
  \rput(0.5,-2){\psline(0,0)(3,3)}
  \rput(1,-2){\psline(0,0)(3,3)}
  \rput(1.5,-2){\psline(0,0)(3,3)}
  \rput(2,-2){\psline(0,0)(3,3)}
  \rput(2.5,-2){\psline(0,0)(2.5,2.5)}
  \rput(3,-2){\psline(0,0)(2,2)}
  \rput(3.5,-2){\psline(0,0)(1.5,1.5)}
  \rput(4,-2){\psline(0,0)(1,1)}
  \rput(4.5,-2){\psline(0,0)(0.5,0.5)}
  \psline(0.75,-1.25)(3,1)
  \psline(0.75,-0.75)(2.5,1)
  \psline(0.75,-0.25)(2,1)
  \psline(0.75,0.25)(1.5,1)
  \psline(0.75,0.75)(1,1)
  \rput(4.5,1.5){\psline(0,0.5)(0.5,0)}
  \rput(4,1.5){\psline(0,0.5)(0.5,0)}
  \rput(3.5,1.5){\psline(0,0.5)(0.5,0)}
  \rput(3,1.5){\psline(0,0.5)(0.5,0)}
  \rput(2.5,1.5){\psline(0,0.5)(0.5,0)}
  \rput(2,1.5){\psline(0,0.5)(0.5,0)}
  \rput(1.5,1.5){\psline(0,0.5)(0.5,0)}
  \rput(1,1.5){\psline(0,0.5)(0.5,0)}
  \rput(0.5,1.5){\psline(0,0.5)(0.5,0)}
  \rput(0,1.5){\psline(0,0.5)(0.5,0)}
  \rput(-0.5,1.5){\psline(0,0.5)(0.5,0)}
  \rput(-1,1.5){\psline(0,0.5)(0.5,0)}
  \rput(-5,-1){\psline(0,3)(3,0)}
  \rput(-4.5,-1){\psline(0,3)(3,0)}
  \rput(-4,-1){\psline(0,3)(3,0)}
  \rput(-3.5,-1){\psline(0,3)(2.75,0.25)}
  \rput(-3,-1){\psline(0,3)(2.25,0.75)}
  \rput(-2.5,-1){\psline(0,3)(1.75,1.25)}
  \rput(-2,-1){\psline(0,3)(1.25,1.75)}
  \rput(-1.5,-1){\psline(0,3)(0.75,2.25)}
  \psline(-5,1.5)(-2.5,-1)
  \psline(-5,1)(-3,-1)
  \psline(-5,0.5)(-3.5,-1)
  \psline(-5,0)(-4,-1)
  \psline(-5,-0.5)(-4.5,-1)
  \psline(-0.75,-0.75)(-0.5,-1)
  \psline(0.5,-1)(0.75,-0.75)
  \pscurve(-0.75,-0.25)(-0.5,-0.5)(0,-1)(0.5,-0.5)(0.75,-0.25)
  \pscurve(-0.75,0.25)(-0.5,0)(0,-0.5)(0.5,0)(0.75,0.25)
  \pscurve(-0.75,0.75)(-0.5,0.5)(0,0)(0.5,0.5)(0.75,0.75)
  \pscurve(-0.5,1)(0,0.5)(0.5,1)
  \rput(-5,-1.5){\psbezier(0,0)(0.2,0.2)(-0.2,0.3)(0,0.5)}
  \rput(-4.5,-1.5){\psbezier(0,0)(0.2,0.2)(-0.2,0.3)(0,0.5)}
  \rput(-4,-1.5){\psbezier(0,0)(0.2,0.2)(-0.2,0.3)(0,0.5)}
  \rput(-3.5,-1.5){\psbezier(0,0)(0.2,0.2)(-0.2,0.3)(0,0.5)}
  \rput(-3.0,-1.5){\psbezier(0,0)(0.2,0.2)(-0.2,0.3)(0,0.5)}
  \rput(-2.5,-1.5){\psbezier(0,0)(0.2,0.2)(-0.2,0.3)(0,0.5)}
  \rput(-2.0,-1.5){\psbezier(0,0)(0.2,0.2)(-0.2,0.3)(0,0.5)}
  \rput(-1.5,-1.5){\psbezier(0,0)(0.2,0.2)(-0.2,0.3)(0,0.5)}
  \rput(-1.0,-1.5){\psbezier(0,0)(0.2,0.2)(-0.2,0.3)(0,0.5)}
  \rput(-0.5,-1.5){\psbezier(0,0)(0.2,0.2)(-0.2,0.3)(0,0.5)}
  \rput(0.0,-1.5){\psbezier(0,0)(0.2,0.2)(-0.2,0.3)(0,0.5)}
  \rput(0.5,-1.5){\psbezier(0,0)(0.2,0.2)(-0.2,0.3)(0,0.5)}
  \rput(-0.5,1){\psbezier(0,0)(0.2,0.2)(-0.2,0.3)(0,0.5)}
  \rput(0.0,1){\psbezier(0,0)(0.2,0.2)(-0.2,0.3)(0,0.5)}
  \rput(0.5,1){\psbezier(0,0)(0.2,0.2)(-0.2,0.3)(0,0.5)}
  \rput(1.0,1){\psbezier(0,0)(0.2,0.2)(-0.2,0.3)(0,0.5)}
  \rput(1.5,1){\psbezier(0,0)(0.2,0.2)(-0.2,0.3)(0,0.5)}
  \rput(2.0,1){\psbezier(0,0)(0.2,0.2)(-0.2,0.3)(0,0.5)}
  \rput(2.5,1){\psbezier(0,0)(0.2,0.2)(-0.2,0.3)(0,0.5)}
  \rput(3.0,1){\psbezier(0,0)(0.2,0.2)(-0.2,0.3)(0,0.5)}
  \rput(3.5,1){\psbezier(0,0)(0.2,0.2)(-0.2,0.3)(0,0.5)}
  \rput(4.0,1){\psbezier(0,0)(0.2,0.2)(-0.2,0.3)(0,0.5)}
  \rput(4.5,1){\psbezier(0,0)(0.2,0.2)(-0.2,0.3)(0,0.5)}
  \rput(5.0,1){\psbezier(0,0)(0.2,0.2)(-0.2,0.3)(0,0.5)}
  \rput*(-3,0.5){\psscalebox{1.2}{$\Omega^t$}}
  \rput*(3,-0.5){\psscalebox{1.2}{$\Omega^b$}}
  \rput*(0,0){\psscalebox{1.2}{$\Omega^m$}}
  \psline(6,1.35)(4.5,1.25)
  \rput*(6,1.25){\psscalebox{1.2}{$\Omega^t_\text{inter}$}}
  \psline(-6,-1.15)(-4.5,-1.25)
  \rput*(-6,-1.25){\psscalebox{1.2}{$\Omega^b_\text{inter}$}}
  \rput[l](-1.3,2.70){\psscalebox{0.8}{$\gamma = \{m_2^*=0\}$}}
  \psline(-1.1,2.55)(-1.4,1.8)
\end{pspicture}
\caption{The way in which the three solutions $\psi^t_\theta$, $\psi^m_\theta$ and $\psi^b_\theta$ are combined in the case $\sigma=1$.}
\label{fig:combsolutions}
\end{figure}
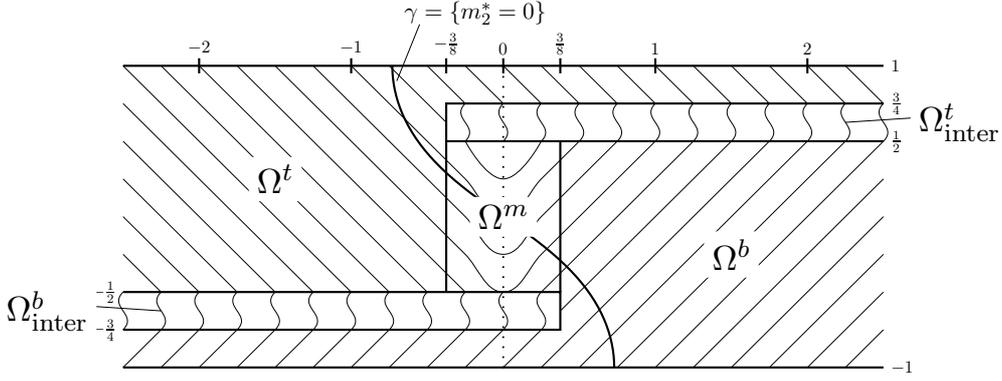

In order to prove the energy estimate \eqref{eq:upperbd}, we will show that each of the approximate solutions $\psi_\theta^t$, $\psi_\theta^b$, $\psi_\theta^m$ and therefore also the blended function $\psi_\theta$ agrees with
\begin{align}\label{eq:deftildepsi}
\tilde{\psi}_\theta(x_1,x_3) = - \int_{-1}^{x_3} \Bigl(\cos\theta + \sin^2\theta \,\tfrac{1-(m^*_2)^2}{2}(x_1,s)\Bigr) \, ds
\end{align}
up to an error of order $\co(\theta^4)$ that is exponentially decaying as $\lvert x_1 \rvert\to \infty$. We remark that it is natural to consider $\tilde{\psi}_\theta$ since by Proposition~\ref{prop:compactnessasym} we expect to have $\partial_3 \psi_\theta = -m_{1,\theta} \approx -\bigl(\cos\theta+\sin^2\theta\,\frac{1-(m_2^*)^2}{2}\bigr)$.

\subsubsection*{Method of characteristics}
Our main tool for constructing solutions to the Hamilton-Jacobi equation \eqref{eq:eik} and proving $\psi_\theta=\tilde{\psi}_\theta+\co(\theta^4)$ is the method of characteristics. Recall (see e.g. \cite{evanspde10,lions82}) that given a Hamiltonian $H_\theta\colon\R^2\times \Omega\to \R$ the characteristics $(p,x) \colon I_\text{max}\subset\R \to \R^2\times\overline{\Omega}$ are defined as solutions of the Hamiltonian system $\dot{p}=-\nabla_{x} H_\theta(p, x)$ and $\dot{x}=\nabla_{p} H_\theta(p, x)$ for some initial data $(p_0, x_0)$ at time $s=0\in I_\text{max}$ with $H_\theta(p_0, x_0)=0$ (so that $H_\theta(p(s), x(s))=0$ for all times $s$ of existence).

In the case of \eqref{eq:eik}, i.e. for
\begin{align*}
H_\theta(p, x)=\lvert p \rvert^2 - \Bigl(1-\sin^2\theta \,(m^*_2)^2(x)\Bigr),
\end{align*}
choosing $\theta>0$ small enough so that $1-\sin^2\theta \, (m^*_2)^2(x)>\frac{1}{2}$, we obtain the following characteristic equations:
\begin{gather}\label{eq:eikchar}
\begin{aligned}
\dot{x} &= 2p,\\
\dot{p} &= -\sin^2\theta\, \nabla(m^*_2)^2(x) =: F(x).
\end{aligned}
\end{gather}
Observe that also $\tilde{\psi}_\theta$ satisfies \eqref{eq:eik} up to an error of order $\co(\theta^4)$:
\begin{align}\label{alta_con}
\lvert \nabla \tilde{\psi}_\theta \rvert^2=1-\sin^2\theta \, (m_2^*)^2+\co(\theta^4).
\end{align}
Finally, we cast property \eqref{alta_con} in the form of a Hamilton-Jacobi equation for the Hamiltonian
\begin{align*}
\tilde{H}_\theta(\tilde{p}, \tilde{x})=\lvert \tilde{p} \rvert^2 - \lvert \nabla \tilde{\psi}_\theta \rvert^2.
\end{align*}
The associated characteristic equations read
\begin{align}\label{eq:eikchartilde}
\begin{aligned}
\dot{\tilde{x}} &= 2\tilde{p},\\
\dot{\tilde{p}} &= \nabla \lvert \nabla \tilde{\psi}_\theta \rvert^2(\tilde{x}) =: \tilde{F}(\tilde{x}).
\end{aligned}
\end{align}
Observe that \eqref{eq:eikchar} and \eqref{eq:eikchartilde} are the same, up to an error $\lvert F(x) - \tilde{F}(x) \rvert$ of order $\theta^4$, which decays exponentially as $\lvert x \rvert \to \infty$.

In Lemmas~\ref{lem:propeik1}~and~\ref{lem:propeik2} in the Appendix we have collected properties of the solutions of \eqref{eq:eikchar} and \eqref{eq:eikchartilde} which will be used in the sequel. In particular, we prove existence and uniqueness, a growth estimate and a stability estimate under perturbations of the initial data. Furthermore, we prove that the characteristics cover the whole domain.

\subsection{Proof of Proposition~\ref{prop:upperbd}}
To simplify notation, we will omit the index $\theta$ from here on.

\textbf{Step 1:} \textit{A simple perturbation result.}

Let $(p,x)$ and $(\tilde{p},\tilde{x})$ be solutions of \eqref{eq:eikchar} and \eqref{eq:eikchartilde} with respect to the initial conditions $(p_0,x_0)$ and $(\tilde{p}_0=\nabla\tilde{\psi}(\tilde{x}_0),\tilde{x}_0)$ at time $s=0$, respectively. Suppose that $\lvert x_0 - \tilde{x}_0 \rvert \leq 1$ and $\lvert p_0 - \tilde{p}_0\rvert \leq C \sin^2\theta<\tfrac{1}{2}$. Then, by the fundamental theorem of calculus, we obtain the following estimates:
\begin{align*}
\MoveEqLeft \lvert p(r)-\tilde{p}(r) \rvert \leq \lvert p_0 - \tilde{p}_0 \rvert + \biggl\lvert \int_0^r \underbrace{\dot{p}(s) - \dot{\tilde{p}}(s)}_{\mathclap{= F(x(s))-F(\tilde{x}(s))+F(\tilde{x}(s)) - \tilde{F}(\tilde{x}(s))}} \, ds \biggr\rvert\\
&\leq \lvert p_0 - \tilde{p}_0 \rvert + C\bigl( \lVert \nabla F \rVert_{\infty,\text{loc}} \sup_{s\in[0,r]} \lvert x(s) - \tilde{x}(s) \rvert + \lVert F-\tilde{F} \rVert_{\infty,\text{loc}}\bigr),
\end{align*}
as well as
\begin{align*}
\MoveEqLeft \lvert x(r)-\tilde{x}(r) \rvert \leq \lvert x_0 - \tilde{x}_0 \rvert + \biggl\lvert \int_0^r \underbrace{\dot{x}(s) - \dot{\tilde{x}}(s)}_{\mathclap{=2p(s)-2\tilde{p}(s)}} \, ds \biggr\rvert \leq \lvert x_0 - \tilde{x}_0 \rvert + 2\lvert r \rvert \sup_{s\in[0,r]} \lvert p(s) - \tilde{p}(s) \rvert\\
&\leq \lvert x_0 - \tilde{x}_0 \rvert + C\bigl( \lvert p_0 - \tilde{p}_0 \rvert + \lVert \nabla F \rVert_{\infty,\text{loc}} \sup_{s\in[0,r]} \lvert x(s) - \tilde{x}(s) \rvert + \lVert F-\tilde{F} \rVert_{\infty,\text{loc}} \bigr),
\end{align*}
where we denote by $\lVert \cdot \rVert_{\infty,\text{loc}}$ the $L^\infty$ norm of a function, taken over a sufficiently large but bounded box that contains both trajectories $x$ and $\tilde{x}$.

Since $\lVert \nabla F \rVert_{\infty} \leq C\sin^2\theta\to 0$ as $\theta \tod 0$, we can absorb the third term on the right hand side of the second estimate into the left hand side, so that we obtain
\begin{align*}
\sup_{r \in [0,s]}\lvert x(r)-\tilde{x}(r) \rvert \leq C \bigl( \lvert x_0-\tilde{x}_0 \rvert +  \lvert p_0 - \tilde{p}_0 \rvert +  \lVert F-\tilde{F} \rVert_{\infty,\text{loc}} \bigr),
\end{align*}
and thus also
\begin{align*}
\sup_{r \in [0,s]} \lvert p(r) - \tilde{p}(r) \rvert \leq C\bigl( \lvert p_0 - \tilde{p}_0 \rvert + \lVert F - \tilde{F} \rVert_{\infty,\text{loc}} + \sin^2\theta \lvert x_0 - \tilde{x}_0\rvert\bigr).
\end{align*}
Now additionally suppose that $x$ and $\tilde{x}$ intersect at some point $\bar{x}=x(s)=\tilde{x}(\tilde{s})\in\overline{\Omega}$, and that the initial values $x_0$ and $\tilde{x}_0$ are taken from the graph of a smooth function $f \colon [a,b] \to [-1,1]$, so that we have the estimate $\lvert x_{0,3}- \tilde{x}_{0,3} \rvert \leq \lVert f' \rVert_{\infty} \lvert x_{0,1} - \tilde{x}_{0,1}\rvert$. We prove that $x_0$ and $\tilde{x}_0$ cannot differ much and therefore the previous estimates apply.

Note that $s$ and $\tilde{s}$ need to have the same sign. W.l.o.g. we may assume $0 < s \leq \tilde{s}$. Thus, we have
\begin{align*}
\MoveEqLeft \lvert \tilde{s} - s \rvert \leq 2 \int_s^{\tilde{s}} \underbrace{-\tilde{p}_3(r)}_{\geq \frac{1}{2}} \, dr = \underbrace{\int_0^s \dot{x}_3(r) \, dr - \int_0^{\tilde{s}} \dot{\tilde{x}}_3(r) \, dr}_{= x_{0,3}-\tilde{x}_{0,3}} - 2\int_0^s p_3(r) - \tilde{p}_3(r) \, dr\\
&\leq C\bigl( \lvert x_0 - \tilde{x}_0 \rvert + \lvert p_0 - \tilde{p}_0 \rvert + \lVert F - \tilde{F} \rVert_{\infty,\text{loc}} \bigr).
\end{align*}
This yields
\begin{align*}
\MoveEqLeft \underbrace{\lvert x_{0,3} - \tilde{x}_{0,3} \rvert}_{\leq \lVert f' \rVert_\infty \lvert x_{0,1}-\tilde{x}_{0,1}\rvert} +\quad \lvert \underbrace{x_{0,1} - \tilde{x}_{0,1}}_{\mathclap{= x_{0,1} - \bar{x}_1 + \bar{x}_1 - \tilde{x}_{0,1}}} \rvert \leq (1+\lVert f' \rVert_{\infty} )\biggl\lvert \int_0^{\tilde{s}} 2\tilde{p}_1(r) \, dr - \int_0^s 2p_1(r) \, dr \biggr\rvert\\
&\leq C \sup_{r \in [0,s]} \lvert p_1(r) - \tilde{p}_1(r) \rvert + C \underbrace{\sup_{r \in [0,s]} \max\bigl( \lvert p_1(r) \rvert,\lvert \tilde{p}_1(r) \rvert \bigr)}_{\leq C \sin^2\theta} \, \lvert \tilde{s} - s \rvert\\
&\leq C \bigl( \lvert p_0 - \tilde{p}_0 \rvert + \lVert F-\tilde{F} \rVert_{\infty,\text{loc}} + \sin^2\theta \, ( \lvert \tilde{x}_{0,3} - x_{0,3} \rvert + \lvert \tilde{x}_{0,1} - x_{0,1} \rvert)\bigr),
\end{align*}
so that (again absorbing the small term $C\sin^2\theta\,(\lvert \tilde{x}_{0,3} - x_{0,3} \rvert + \lvert \tilde{x}_{0,1} - x_{0,1} \rvert )$ into the left hand side)
\begin{align*}
\lvert x_0 - \tilde{x}_0 \rvert \leq C\bigl( \lvert x_{0,1} - \tilde{x}_{0,1}\rvert + \lvert x_{0,3} - \tilde{x}_{0,3}\rvert \bigr)\leq C \bigl( \lvert p_0 - \tilde{p}_0 \rvert + \lVert F-\tilde{F} \rVert_{\infty,\text{loc}} \bigr).
\end{align*}
This improves our estimates to
\begin{align*}
\MoveEqLeft \lvert x_0 - \tilde{x}_0 \rvert + \lvert s - \tilde{s} \rvert + \sup_{0\leq r \leq s}\lvert p(r) - \tilde{p}(r) \rvert + \sup_{0\leq r \leq s} \lvert x(r) - \tilde{x}(r) \rvert\\
&\leq C \bigl( \lvert p_0 - \tilde{p}_0 \rvert + \lVert F-\tilde{F} \rVert_{\infty,\text{loc}} \bigr),
\end{align*}
and we may deduce for a solution $\psi$ of \eqref{eq:eik}, using $\psi(\bar{x})=\psi(x_0)+\int_0^s \tfrac{d}{dt}\psi(x(t)) dt$ and the equivalent for $\tilde{\psi}$,
\begin{align*}
\MoveEqLeft \lvert \psi(\bar{x}) - \tilde{\psi}(\bar{x}) \rvert \leq \lvert \psi(x_0) - \tilde{\psi}(\tilde{x}_0) \rvert + 2\biggl\lvert\int_0^s \lvert p(s)\rvert^2 ds - \int_0^{\tilde{s}} \lvert \tilde{p}(s) \rvert^2 ds \biggr\rvert\\
&\leq \lvert \psi(x_0) - \tilde{\psi}(\tilde{x}_0) \rvert + C\bigl( \lvert p_0 - \tilde{p}_0 \rvert + \lVert F-\tilde{F} \rVert_{\infty,\text{loc}} \bigr)
\end{align*}
and
\begin{align*}
\MoveEqLeft \underbrace{\lvert \nabla \psi(\bar{x}) - \nabla \tilde{\psi}(\bar{x}) \rvert}_{\mathclap{\leq \lvert p(s)-\tilde{p}(s)\rvert + \lvert\tilde{p}(s)-\tilde{p}(\tilde{s})\rvert}} \leq C\bigl( \lvert p_0 - \tilde{p}_0 \rvert + \lVert F-\tilde{F} \rVert_{\infty,\text{loc}} \bigr) + \biggl\lvert \int_{s}^{\tilde{s}} \tilde{F}(x(s)) \, ds \biggr\rvert\\
&\leq C\bigl( \lvert p_0 - \tilde{p}_0 \rvert + \lVert F-\tilde{F} \rVert_{\infty,\text{loc}} \bigr).
\end{align*}

\textbf{Step 2:} \textit{Construction of $\psi^b$ and $\psi^t$.}

For $\theta\ll 1$, we consider the following admissible initial data of the Hamilton-Jacobi equation \eqref{eq:eik}:
  \begin{align*}
  (x^t_0,z^t_0,p^t_0) &= \bigl((x^t_{0,1},1), -2\cos\theta, -\sqrt{1-\sin^2\theta \,(m^*_2)^2(x^t_{0,1},1)}\, {\bf e}_3\bigr)\in \partial\Omega\times\R\times\R^2
  \intertext{and}
  (x^b_0,z^b_0,p^b_0) &= \bigl((x^b_{0,1},-1), 0, -\sqrt{1-\sin^2\theta \,(m^*_2)^2(x^b_{0,1},-1)}\, {\bf e}_3\bigr)\in \partial\Omega\times\R\times\R^2.
  \end{align*}
 Thus, by the method of characteristics (cf. Appendix), there exist smooth functions $\psi^b$ and $\psi^t$
that satisfy $$\textrm{$\nabla \psi^b(x^b(s)) = p^b(s)$ and $\nabla \psi^t(x^t(s))=p^t(s)$, respectively,}$$ and solve \eqref{eq:eik} subject to the one-sided constraints $\psi^b = \tilde{\psi}=0$ on $\{x_3=-1\}$ and $\psi^t = \tilde{\psi}=-2\cos\theta$ on $\{x_3=1\}$. 

Observe that $\lvert p_0^t - \tilde{p}_0 \rvert \leq C\sin^2\theta \lvert x_0^t - \tilde{x}_0 \rvert$, so that Step~1 for $f \equiv 1$ immediately yields $\lvert x_0^t - \tilde{x}_0 \rvert \leq C \lVert F-\tilde{F} \rVert_{\infty,\text{loc}}$ and thus also
\begin{align*}
\lvert \psi^t(\bar{x})-\tilde{\psi}(\bar{x}) \rvert + \lvert \nabla \psi^t(\bar{x}) - \nabla \tilde{\psi}(\bar{x}) \rvert \leq C\lVert F-\tilde{F} \rVert_{\infty,\text{loc}},   \quad \forall \bar{x}\in \Omega.
\end{align*}
The same inequality holds for $\psi^b$, i.e.,
\begin{align*}
\lvert \psi^b(\bar{x})-\tilde{\psi}(\bar{x}) \rvert + \lvert \nabla \psi^b(\bar{x}) - \nabla \tilde{\psi}(\bar{x}) \rvert \leq C\lVert F-\tilde{F} \rVert_{\infty,\text{loc}},   \quad \forall\bar{x}\in \Omega.
\end{align*}

\textbf{Step 3:} \textit{Definition of $\psi^m$.}

Now, we focus on the region around the curve $\gamma$ of zeros of $m^*_2$, where we interpolate $\psi^t$ with $\psi^b$ by means of another stream function $\psi^m$ that will satisfy the same equation \eqref{eq:eik}. To define $\psi^m$, we will impose initial conditions on the curve $\gamma$ that interpolate between $\nabla\psi^t$ and $\nabla\psi^b$.

To this end, we may assume that $\gamma$ denotes an arc-length parametrization of the curve of zeros of $m^*_2$, i.e., $\lvert \dot{\gamma} \rvert = 1$. Moreover, we can choose $\gamma$ symmetric w.r.t. $\gamma(0)=0$. Denote by $a>0$ the value at which $\gamma_1(a) = -\gamma_1(-a) = \tfrac{1}{2}$. Obviously, $a>\frac{1}{2}$. We also remark that on $[-a,a]$ the curve $\gamma$ is in fact the graph of a smooth function.

Let $\alpha\colon (-\tfrac{1}{2},\tfrac{1}{2}) \to [0,1]$ be a smooth cutoff function with $\alpha(s)=1$ for $s\leq-\tfrac{1}{4}$ and $\alpha(s)=0$ for $s\geq\tfrac{1}{4}$. 
\begin{figure}
  \centering
  \begin{pspicture}(-2.25,-0.75)(2.25,2.25)
    \psset{unit=1.5}
    \psline{->}(-2.4,0)(2.4,0)
    \psline{->}(0,-0.2)(0,1.2)
    \psline{-}(-2,-0.1)(-2,0.1)
    \psline{-}(-1,-0.1)(-1,0.1)
    \psline{-}(1,-0.1)(1,0.1)
    \psline{-}(2,-0.1)(2,0.1)
    \psline{-}(-0.1,1)(0.1,1)
    \psline{-}(-2,1)(-1,1)
    \psbezier{-}(-1,1)(0,1)(0,0)(1,0)
    \psline{-}(1,0)(2,0)
    \rput[t](-2,-0.15){\psscalebox{0.7}{$-\tfrac{1}{2}$}}
    \rput[t](-1,-0.15){\psscalebox{0.7}{$-\tfrac{1}{4}$}}
    \rput[t](1,-0.15){\psscalebox{0.7}{$\tfrac{1}{4}$}}
    \rput[t](2,-0.15){\psscalebox{0.7}{$\tfrac{1}{2}$}}
    \rput[l](0.2,1){\psscalebox{0.7}{$1$}}
    \rput[b](0,1.25){\psscalebox{0.7}{$\alpha$}}
    \rput[l](2.5,0){\psscalebox{0.7}{$x_1$}}
  \end{pspicture}
  \caption{The cutoff function $\alpha$.}
\end{figure}
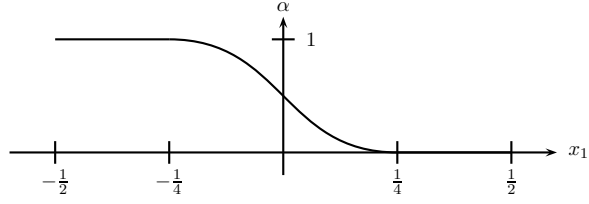
Then we define $h \colon (-\tfrac{1}{2},\tfrac{1}{2}) \to \R$ via
\begin{align*}
  \psi^m\bigl(\gamma(s)\bigr):=h(s) = \alpha(s) \psi^t\bigl(\gamma(s)\bigr) + \bigl(1-\alpha(s)\bigr) \psi^b\bigl(\gamma(s)\bigr)=\tilde{\psi}\bigl(\gamma(s)\bigr) + \co(\lVert F-\tilde{F}\rVert_{\infty,\text{loc}}).
\end{align*}
Given $s\in(-\tfrac{1}{2},\tfrac{1}{2})$ we let
\begin{align*}
  x_0^m := \gamma(s),\\
  z_0^m := h(s),
\end{align*}
while $p_0^m$ is uniquely determined by
\begin{align*}
  p_0^m \cdot \dot{\gamma}(s) &:= \dot{h}(s) = \nabla \tilde{\psi} \bigl(\gamma(s)\bigr) \cdot \dot{\gamma}(s) + \co(\lVert F-\tilde{F} \rVert_{\infty,\text{loc}}),\\
  p_0^m \cdot \dot{\gamma}^\perp(s) &:= -\sqrt{1-\lvert \dot{h}(s) \rvert^2} = \nabla \tilde{\psi} \bigl(\gamma(s)\bigr) \cdot \dot{\gamma}^\perp(s) + \co(\lVert F-\tilde{F} \rVert_{\infty,\text{loc}})
   \end{align*}
Clearly, $(x_0^m,z_0^m,p_0^m)$ are admissible initial data such that the method of characteristics yields a smooth function $\psi^m \colon [-\tfrac{3}{8},\tfrac{3}{8}] \times [-1,1]\to \R$ that coincides with $\psi^t$ in a neighbourhood of $\{-\tfrac{3}{8}\} \times [-1,1]$ and with $\psi^b$ in a neighbourhood of $\{\tfrac{3}{8}\} \times [-1,1]$, and solves \eqref{eq:eik} with the boundary conditions $(z_0^m,p_0^m)$ on $\gamma$.

In particular, using the notation from Step~1, we have $\lvert p_0^m - \tilde{p}_0 \rvert \leq C \lVert F-\tilde{F} \rVert_{\infty,\text{loc}}$, so that Step~1 yields
\begin{align*}
\lvert \psi^m(\bar{x}) - \tilde{\psi}(\bar{x}) \rvert + \lvert \nabla \psi^m(\bar{x}) - \nabla\tilde{\psi}(\bar{x}) \rvert \leq C \lVert F-\tilde{F} \rVert_{\infty,\text{loc}}\quad \forall \bar{x} \in [-\tfrac{3}{8},\tfrac{3}{8}]\times[-1,1].
\end{align*}

\textbf{Step 4:} \textit{Construction of a smooth stream function $\psi \colon \Omega \to\R$ such that
    \begin{itemize}
      \item $\psi$ satisfies \eqref{eq:bcb} and \eqref{eq:bct} on $\partial \Omega$,
      \item $\lvert \nabla \psi \rvert \leq 1$ and $\lvert \nabla \psi \rvert^2 = 1-\sin^2\theta \,(m^*_2)^2$ in a neighborhood $N\subset\Omega$ of $\gamma$,
      \item $\lvert \psi(x) - \tilde{\psi}(x) \rvert + \lvert \nabla\psi(x) - \nabla\tilde{\psi}(x) \rvert\leq C \lVert F - \tilde{F} \rVert_{\infty,\text{loc}}$ for all $x \in \Omega$.
    \end{itemize}
}

  Here we blend the functions $\psi^t$, $\psi^b$ and $\psi^m$ from Steps 2 and 3 as indicated in Figure~\ref{fig:combsolutions}.
  Let $\eta \colon [-1,1] \to [0,1]$ be a smooth cutoff function with $\eta(x_3)=1$ for $x_3 \geq \tfrac{3}{4}$ and $\eta(x_3)=0$ if $x_3\leq\tfrac{1}{2}$. We define
  \begin{align*}
    \psi(x) &=
    \begin{cases}
      \psi^t(x), &\text{if } x \in \Omega^t,\\
      \psi^m(x), &\text{if } x \in \Omega^m,\\
      \psi^b(x), &\text{if } x \in \Omega^b,\\
      \eta(x_3) \psi^t(x) + \bigl(1-\eta(x_3)\bigr) \psi^m(x), &\text{if } x \in \Omega^t_\text{inter} \text{ and } |x_1| \leq \tfrac{3}{8},\\
      \eta(x_3) \psi^t(x) + \bigl(1-\eta(x_3)\bigr) \psi^b(x), &\text{if } x \in \Omega^t_\text{inter} \text{ and } x_1 \geq \tfrac{3}{8},\\
      \eta(-x_3) \psi^b(x) + \bigl(1-\eta(-x_3)\bigr) \psi^m(x), &\text{if } x \in \Omega^b_\text{inter} \text{ and } |x_1| \leq \tfrac{3}{8},\\
      \eta(-x_3) \psi^b(x) + \bigl(1-\eta(-x_3)\bigr) \psi^t(x), &\text{if } x \in \Omega^b_\text{inter} \text{ and } x_1 \leq -\tfrac{3}{8}.\\
    \end{cases}\\
    &= \tilde{\psi}(x) + \co(\lVert F - \tilde{F} \rVert_{\infty,\text{loc}}).
  \end{align*}
  Observe that the resulting function is smooth, since $\psi^t$, $\psi^m$, $\psi^b$ and $\eta$ are, and since $\psi^m=\psi^t$ and $\psi^m=\psi^b$ in neighbourhoods of $\{-\tfrac{3}{8}\}\times[-1,1]$ and $\{\tfrac{3}{8}\}\times[-1,1]$, respectively. Furthermore, by definition, $\psi$ satisfies \eqref{eq:bcb} and \eqref{eq:bct} on $\partial \Omega$ and $\lvert \nabla \psi \rvert^2 = 1 - \sin^2\theta \, (\mds)^2 \leq 1$ on $\Omega^t \cup \Omega^b \cup \Omega^m$. In particular $\lvert \nabla \psi \rvert^2 = 1$ on $\gamma$. On $\Omega^t_\text{inter} \cup \Omega^b_\text{inter}$ we can estimate
  \begin{align*}
    \lvert \nabla \psi \rvert &\leq \eta \underbrace{\lvert \nabla \hat{\psi} \rvert}_{\mathclap{= \sqrt{1-\sin^2\theta \, (\mds)^2}}} \; + \; (1-\eta) \underbrace{\lvert \nabla \check{\psi} \rvert}_{\mathclap{= \sqrt{1-\sin^2\theta \, (\mds)^2}}} \; + \; \lvert \partial_3 \eta \rvert \underbrace{\lvert \hat{\psi} - \check{\psi} \rvert}_{\mathclap{\leq C\lVert F-\tilde{F} \rVert_{\infty,\text{loc}}}}\leq \sqrt{1-\sin^2\theta \, (\mds)^2} + C\lVert F - \tilde{F} \rVert_{\infty,\text{loc}} < 1
  \end{align*}
  for $\hat{\psi}, \check{\psi} \in \{\psi^b, \psi^t, \psi^m\}$ suitably chosen and $\theta\ll 1$, since $(\mds)^2\geq C>0$ for every $x \in \Omega^t_\text{inter}\cup\Omega^b_\text{inter}$.

{\bf Step 5:} {\it Conclusion.} 

First, observe that
\begin{gather*}
\lvert \nabla \psi \rvert^2 = 1-\sin^2\theta \, (m^*_2)^2 + 
\begin{cases}
\co(\lVert F-\tilde{F} \rVert_{\infty,\text{loc}}), &\text{on }\Omega^t_\text{inter}\cup\Omega^b_\text{inter},\\ 
0, &\text{otherwise}.
\end{cases}
\end{gather*}
We define
\begin{align*}
  m'(x) = \nabla^\perp \psi(x), \qquad m_{2}(x) = s(x) \sqrt{1-\lvert m'(x) \rvert^2},
\end{align*}
where $s(x)=-1$ if $x$ is left of $\gamma$ and $s(x)=1$ else, to obtain a magnetization that is admissible in $E_\text{asym}(\theta)$.

This yields
\begin{align*}
  s(x) \sqrt{1-\lvert m'(x) \rvert^2}= \sin\theta \, m^*_2(x)
\end{align*}
for $x \in N$ and 
\begin{align*}
  s(x) \sqrt{1-\lvert m'(x) \rvert^2}= \sin\theta \, m^*_2(x) \sqrt{1+\co\Bigl(\tfrac{\lVert F-\tilde{F} \rVert_{\infty,\text{loc}}}{\sin^2\theta}\Bigr)} = \sin\theta \, m^*_2(x) + \co\Bigl(\tfrac{\lVert F-\tilde{F} \rVert_{\infty,\text{loc}}}{\sin\theta}\Bigr)
\end{align*}
for $x\not\in N$, since $\lvert m^*_2 \rvert \geq C>0$  in $\Omega\setminus N$.

  We claim that the smooth $m=m_\theta$ defined above generates a suitable recovery family as $\theta\tod 0$. Indeed, it is easy to see that $m\in X_0\cap\Xt$, and using the expansion of $\psi$ given in Step~4 we obtain (by \eqref{eq:defhatm})
  \begin{align*}
    m = \left( \begin{smallmatrix} \cos\theta\\\sin\theta \,m^*_2\\0\end{smallmatrix}\right) + \sin^2\theta \left( \begin{smallmatrix} \hat{m}_1 + \sin^2\theta \, \phi_1\\ \sin\theta \,\phi_2\\ \hat{m}_3 + \sin^2\theta \,\phi_3\end{smallmatrix} \right)
  \end{align*}
  with $\phi=(\phi_1,\phi_2,\phi_3) \in C^\infty(\Omega)$ decaying exponentially as $\lvert x \rvert \to \infty$. Thus, by Lemma \ref{lem:calc_average_second}, we have
  \begin{align*}
     \MoveEqLeft \sin^{-4}\theta \int_{\Omega} \Bigl(\lvert \nabla m \rvert^2 -\sin^2\theta \, \lvert \nabla m^*_2 \rvert^2 \Bigr)\, dx\\
     &\stackrel{\eqref{eq:rewrittenenergy}}{=} \int_{\Omega} \Bigl(\lvert \nabla \hat{m}_1 \rvert^2 + \lvert \nabla \hat{m}_3 \rvert^2- \mu(x_1) ( \hat{m}_1^2 + \hat{m}_3^2) \Bigr)\, dx + \co(\theta^2)\\
     &\stackrel{\eqref{eq:defe1}}{=} \tilde{E}_1 + \co(\theta^2),
  \end{align*}
  which ends the proof of Proposition~\ref{prop:upperbd}.\qed

%% file: inequalities.tex
\begin{lem}
The following Poincar\'e(-Wirtinger) inequalities hold with optimal constant:
\begin{align}
\forall f\in \dot{H}^1(\Omega)\colon \int_\Omega \lvert f-\bar{f} \rvert^2 dx &\leq \tfrac{4}{\pi^2} \int_\Omega \lvert\partial_3 f \rvert^2 dx,\label{eq:poincm1}\\ 
\forall f\in H^1_0(\Omega)\colon \qquad \int_\Omega f^2 \, dx &\leq \tfrac{4}{\pi^2} \int_\Omega \lvert\partial_3 f \rvert^2 dx.\label{eq:poincm3}
\end{align}
\end{lem}
In the context of this paper, we typically apply the first inequality to ``first components'', i.e. $m_{1,\theta}$ with $\bar{m}_{1,\theta}\equiv \cos\theta$ or $\hat{m}_{1,\theta}$ with $\bar{\hat{m}}_{1,\theta} = 0$. The second inequality is applied to ``third components'', i.e. $m_{3,\theta}$ or $\hat{m}_{3,\theta}$, which satisfy $m_{3,\theta}(\pm\infty,\cdot)=0$ as well as $m_{3,\theta}=0$ on $\partial\Omega$.

The above lemma can be obtained as a corollary e.g. of \cite[Theorem~4.24]{dacorogna08}.

Combining the classical Hardy inequality \cite[pg.~3]{hardy20} with \eqref{eq:poincm1}, one obtains
\medskip
\begin{lem}
There exists a constant $0<C<\infty$, such that we have:
\begin{align}
  \int_\Omega \bigl\lvert f-\bar{f}(0) \bigr\rvert^2 \, \tfrac{1}{1+x_1^2}\,dx &\leq C \int_\Omega \lvert \nabla f \rvert^2 dx\quad\text{for any }f\in\dot{H}^1(\Omega).\label{eq:hardy}
\end{align}
\end{lem}
Note that we may replace the profile $\frac{1}{1+x_1^2}$ with any continuous function that decays at least quadratically for $\lvert x_1 \rvert \to \infty$ such as $\mu=\mu(x_1)$ as in \eqref{eq:defmu}.

%% file: eikonal.tex
The following two lemmata show that the method of characteristics can in fact be applied to construct a solution of the modified eikonal equations \eqref{eq:eik} and \eqref{alta_con} on $\Omega$. Lemma~\ref{lem:propeik1} (iii) yields that under suitable assumptions on the initial data characteristics cannot cross, while by Lemma~\ref{lem:propeik2} characteristics cover the whole domain $\Omega$.

\medskip

\begin{lem}\label{lem:propeik1}
  There exists a universal constant $0<C<\infty$ such that for $\theta\leq \tfrac{1}{C}$ and any initial datum $(x_0,p_0)\in\Omega\times\R^2$ with $\lvert p_0 + \mathbf{e}_3 \rvert \leq \tfrac{1}{2}$ there exists a unique solution $(x,p)$ of \eqref{eq:eikchar}, subject to the initial condition $(x,p)(0)=(x_0,p_0)$. The solution depends smoothly on time and initial data and satisfies:
  \begin{enumerate}
    \item Let $[T_\text{t},T_\text{b}]$ be the maximal interval of existence of $(x,p)$. We have $T_\text{t}=\sup\left\{t\leq 0 \with x_3(t)=1 \right\}$, $T_\text{b}=\inf\left\{ t\geq 0 \with x_3(t) = -1 \right\}$, and $\lvert T_\text{t} \rvert + \lvert T_\text{b} \rvert \leq C$.
    \item It holds $\lvert p(t) - p_0 \rvert \leq C\sin^2\theta$ for all $t \in [T_\text{t},T_\text{b}]$.
    \item Given $0<q<1$ there exists a constant $C(q)$ such that the solution $(\tilde{x},\tilde{p})$ of \eqref{eq:eikchar} corresponding to the initial datum $(\tilde{x},\tilde{p})(0)=(\tilde{x}_0,\tilde{p}_0)\in\Omega\times\R^2$ satisfies the estimate
      \begin{align*}
        \MoveEqLeft \tfrac{1}{C(q)} \bigl(\lvert x_0 - \tilde{x}_0 \rvert + \lvert t - \tilde{t} \rvert \bigr) \\
        &\leq \lvert x(t) - \tilde{x}(\tilde{t}) \rvert \leq C(q) \bigl(\lvert x_0 - \tilde{x}_0 \rvert + \lvert t - \tilde{t} \rvert \bigr) \qquad \forall t\in[T_\text{t},T_\text{b}],\,\tilde{t}\in[\tilde{T}_\text{t},\tilde{T}_\text{b}],
      \end{align*}
      provided $\theta \leq \tfrac{1}{C(q)}$, $\lvert p_0-\tilde{p}_0 \rvert \leq \frac{\lvert x_0 - \tilde{x}_0 \rvert}{C(q)}$, and $\lvert (x_0-\tilde{x}_0)\cdot \xi \rvert \leq q^2 \lvert x_0 - \tilde{x}_0 \rvert \lvert \xi \rvert$ for $\xi \in \{p_0,\tilde{p}_0\}$. This last inequality in fact is a lower bound on the angle between $x_0-\tilde{x}_0$ and $\xi\in\{p_0,\tilde{p}_0\}$.
  \end{enumerate}
  The same statements hold for solutions of \eqref{eq:eikchartilde}.
\end{lem}

\begin{proof}
  We denote by $0<C<\infty$ a universal, generic constant. Note that we only treat the case \eqref{eq:eikchar}, since the argument for \eqref{eq:eikchartilde} is similar. In fact, the only property that we actually need is $F = \co(\theta^2) = \tilde{F}$.

  Existence of a unique solution $(x,p)$ of \eqref{eq:eikchar} that depends smoothly on time and initial data is immediate by standard theory, see e.g. \cite[Ch. 2, Cor. 6]{arnoldode92}. 

  Integrating the second line of \eqref{eq:eikchar} obviously yields
  \begin{align}
    \lvert p(t) -p_0 \rvert &\leq C\sin^2\theta \, \lvert t \rvert\label{eq:estp},
  \end{align}
  if the solution exists on $[0,t)$. In view of the structure of \eqref{eq:eikchar}, this already rules out finite-time blow-up.

  Hence, solutions exist until $x(t)$ leaves $\bar{\Omega}$ and we may estimate for $t>0$
  \begin{align*}
    x_{0,3} - x_3(t) &\stackrel{\mathclap{\eqref{eq:eikchar}}}{=} - 2\int_0^t\! p_{0,3}\, ds- 2\int_0^t\! p_3(s)-p_{0,3} \, ds \\
    &\stackrel{\mathclap{\lvert p_0+\mathbf{e}_3\rvert \leq \frac{1}{2} }}{\geq} \quad t - Ct \sup_{0\leq s\leq t} \lvert p(s) - p_0 \rvert\\
    &\stackrel{\mathclap{\eqref{eq:estp}}}{\geq} t \bigl( 1 - C \sin^2\theta \, t \bigr).
  \end{align*}
  Thus if $\theta$ is sufficiently small, e.g. such that $C\sin^2\theta\leq \tfrac{1}{8}$, we have
  \begin{align*}
    1\geq -x_3(t) \geq \tfrac{t}{8} \bigl( 8 - t \bigr) - x_{0,3} \geq \tfrac{t}{8} \bigl( 8 - t \bigr) - 1,
  \end{align*}
  and therefore $0\leq T_\text{b}\leq 4$. A similar argument shows that $0\leq -T_\text{t} \leq 4$.

  This proves statement (i) of Lemma~\ref{lem:propeik1}. In particular, \eqref{eq:estp} improves to statement (ii).

  We now address statement (iii): Assume $\tilde{t}\geq t \geq 0$. In the remaining cases, the proof is similar (using that $\lvert t \rvert + \lvert \tilde{t} \rvert \leq 2 \lvert t-\tilde{t} \rvert$ if $t\tilde{t}\leq 0$). Successively employing \eqref{eq:eikchar}, one computes
  \begin{gather*}
    \begin{aligned}
      x(t) - \tilde{x}(\tilde{t}) &= (x_0 - \tilde{x}_0) - 2\tilde{p}_0 (\tilde{t}-t) + 2(p_0-\tilde{p}_0) \, t\\
      &\quad+2\int_0^t \int_0^s \underbrace{F\bigl(x(s)\bigr) - F\bigl(\tilde{x}(s)\bigr)}_{\lvert\cdot \rvert \leq C\sin^2\theta \, \lvert x(s) - \tilde{x}(s) \rvert} \, d\sigma\,ds\\
      &\quad-2\int_t^{\tilde{t}} \int_0^s F\bigl(\tilde{x}(s)\bigr) d\sigma \, ds.
    \end{aligned}
  \end{gather*}
  With help of (i) this yields the estimate
  \begin{gather}\label{eq:errestx}
    \begin{aligned}
      \MoveEqLeft \bigl\lvert \bigl(x(t) - \tilde{x}(\tilde{t})\bigr) - (x_0-\tilde{x}_0) + 2\tilde{p}_0 (\tilde{t}-t) \bigr\rvert\\
      &\leq C \sin^2\theta  \bigl(\lvert t-\tilde{t} \rvert + \sup_{s} \lvert x(s) - \tilde{x}(s) \rvert\bigr) + C \lvert p_0 - \tilde{p}_0 \rvert.
    \end{aligned}
  \end{gather}
  By the triangle inequality and \eqref{eq:errestx} in the case $t=\tilde{t}$ we derive 
  \begin{gather*}
    \sup_{s} \lvert x(s) - \tilde{x}(s) \rvert \leq C \bigl( \lvert x_0 - \tilde{x}_0 \rvert + \lvert p_0 - \tilde{p}_0\rvert \bigr),
  \end{gather*}
  such that \eqref{eq:errestx} turns into 
  \begin{gather}\label{eq:errestximp}
    \begin{aligned}
      \MoveEqLeft \bigl\lvert \bigl(x(t) - \tilde{x}(\tilde{t})\bigr) - (x_0-\tilde{x}_0) + 2\tilde{p}_0 (\tilde{t}-t) \bigr\rvert\\
      &\leq C \sin^2\theta  \bigl(\lvert t-\tilde{t} \rvert + \lvert x_0 - \tilde{x}_0\rvert \bigr) + C \lvert p_0 - \tilde{p}_0 \rvert.
    \end{aligned}
  \end{gather}
  The upper bound of statement (iii) now easily follows from the triangle inequality and \eqref{eq:errestximp}.

  For the lower bound note that Young's inequality and the assumptions yield:
  \begin{align*}
    \bigl\lvert (x_0 - \tilde{x}_0 ) - 2 \tilde{p}_0 (\tilde{t}-t) \bigr\rvert^2 &\geq (1-q^2) \lvert x_0 - \tilde{x}_0 \rvert^2 + (1-q^2) \lvert \tilde{t}-t \rvert^2 \smash{\underbrace{4\lvert p_0 \rvert^2}_{\geq 1}}.
  \end{align*}
  Thus by the triangle inequality and \eqref{eq:errestximp}, we can estimate
  \begin{align*}
    \tfrac{1}{2}\sqrt{1-q^2} \bigl(\lvert x_0 - \tilde{x}_0 \rvert + \lvert t- \tilde{t} \rvert \bigr) \leq \lvert x(t)-\tilde{x}(\tilde{t}) \rvert + C\sin^2\theta \bigl( \lvert t- \tilde{t} \rvert + \lvert x_0 -\tilde{x}_0 \rvert \bigr) + C \lvert p_0  - \tilde{p}_0 \rvert.
  \end{align*}
  In view of the assumptions $\theta \leq \tfrac{1}{C(q)}$ and $\lvert p_0 - \tilde{p}_0 \rvert \leq \tfrac{\lvert x_0-\tilde{x}_0 \rvert}{C(q)}$ it remains to choose $C(q)$ sufficiently large to absorb the second and third terms on the right hand side into the left hand side and conclude the lower bound of (iii) in Lemma~\ref{lem:propeik1}.
\end{proof}

\medskip
\begin{lem}\label{lem:propeik2}
  Let $\gamma \subset \bar{\Omega}$ be the graph of a smooth function $f \colon [a,b] \subset\R \to [-1,1]$, and let $p_0=p_0(x_{0,1})\in\R^2$ be a smooth function of $x_{0,1}\in[a,b]$ with $\lVert \frac{d}{dx_1} p_0 \rVert_{\infty} \leq \frac{1}{C(q)}$ for $q^2=\tfrac{7}{9}$ if $\lVert f' \rVert_\infty < 1$, and $q^2=\frac{1+2\lVert f' \rVert_\infty^2}{2+2\lVert f' \rVert_\infty^2}<1$ otherwise (cf. previous lemma). Moreover, assume $\lvert p_0 + \mathbf{e}_3 \rvert \leq \tfrac{1}{4}$ and 
  \begin{align*}
    \lvert p_{0,1} \rvert \leq \frac{\lvert p_{0,3} \rvert}{4\lVert f' \rVert_\infty} \quad \forall x_{0,1}\in[a,b].
  \end{align*}
  Denote for $\theta$ sufficiently small by $(x^a,p^a)$ and $(x^b,p^b)$ the non-intersecting characteristics solving \eqref{eq:eikchar} corresponding to the initial data $\bigl((a,f(a)),p_0\bigr)$ and $\bigl( (b,f(b) ), p_0 \bigr)$, respectively. Let $M$ be the bounded subdomain of $\Omega$ that is bounded by the curves $x^a$ and $x^b$.

  Then for each point $x_1 \in M$ there exists an $x_0\in\gamma$ such that the characteristic $(x,p)$ corresponding to the initial datum $(x_0,p_0)$ passes through $x_1$, i.e. M is covered by characteristics.

  The same statements hold for solutions of \eqref{eq:eikchartilde}.
\end{lem}
\begin{proof}
  First of all, we remark that Lemma~\ref{lem:propeik1} (iii) is applicable and in particular the characteristics $x$ do not intersect. Indeed, by smoothness of $p_0$ we have $\lvert p_0(x_{0,1}) - p_0(\tilde{x}_{0,1}) \rvert \leq \lVert \frac{d}{dx_1} p_0 \rVert_{\infty} \lvert x_0 - \tilde{x}_0 \rvert$. Moreover, due to the assumptions, an improved Cauchy-Schwarz inequality holds: By definition of $x_0=(x_{0,1},f(x_{0,1}))$ and monotonicity of $z\mapsto \frac{a+bz}{\sqrt{1+z^2}}$, for $az\leq b$, $a,b,z\geq 0$, we have for $\xi \in \{p_0(x_0),p_0(\tilde{x_0})\}$:
  \begin{align*}
    \MoveEqLeft \frac{\lvert(x_0-\tilde{x}_0)\cdot \xi \rvert}{\sqrt{\lvert x_{0,1} - \tilde{x}_{0,1} \rvert^2 + \lvert f(x_{0,1})-f(\tilde{x}_{0,1})\rvert^2}}
    \leq \frac{\lvert\xi_1\rvert + \lvert\xi_3\rvert \lvert\frac{f(x_{0,1})-f(\tilde{x}_{0,1})}{x_{0,1}-\tilde{x}_{0,1}} \rvert}{\sqrt{1 + \lvert \frac{f(x_{0,1})-f(\tilde{x}_{0,1})}{x_{0,1}-\tilde{x}_{0,1}}\rvert^2}}
    \leq \frac{\lvert\xi_1\rvert + \lvert\xi_3\rvert \lVert f' \rVert_\infty}{\sqrt{1 + \lVert f' \rVert_\infty^2}}
  \end{align*}
  In the case $\lVert f' \rVert_\infty < 1$ we use $\lvert \xi_1 \rvert \leq\tfrac{1}{4} \leq\tfrac{1}{3} \lvert \xi_3 \rvert$, which implies
  \begin{align*}
    \Bigl(\tfrac{\lvert\xi_1\rvert + \lvert\xi_3\rvert \lVert f' \rVert_\infty}{\sqrt{1 + \lVert f' \rVert_\infty^2}}\Bigr)^2 \leq \lvert \xi_3 \rvert^2 \tfrac{(\frac{1}{3}+\lVert f' \rVert_\infty)^2}{1+\lVert f' \rVert_\infty^2}= \lvert \xi_3 \rvert^2 \tfrac{1+\lVert f' \rVert_\infty^2-\frac{8}{9} + \frac{2}{3}\lVert f' \rVert_\infty}{1+\lVert f' \rVert_\infty^2} = \lvert \xi_3 \rvert^2 (1 - \tfrac{8}{9}+\tfrac{2}{3}) = \tfrac{7}{9} \lvert \xi \rvert^2.
  \end{align*}
  For $\lVert f' \rVert_\infty \geq 1$ the improved Cauchy-Schwarz inequality follows from the assumption $4\lVert f' \rVert_\infty \lvert \xi_1 \rvert \leq \lvert \xi_3\rvert$:
  \begin{align*}
    \Bigl(\lvert\xi_1\rvert + \lvert\xi_3\rvert \lVert f' \rVert_\infty\Bigr)^2 \leq \lvert \xi_1 \rvert^2 + (\tfrac{1}{2} + \lVert f' \rVert_\infty^2)  \lvert \xi_3 \rvert^2 \leq (\tfrac{1}{2} + \lVert f' \rVert_\infty^2) \lvert \xi \rvert^2.
  \end{align*}

  We now give a topological argument to prove that the characteristics cover $M$: Denote by $\Psi \colon D \to M$ the continuous map $(t,s) \mapsto x(s)$, where $(x,p)$ is the characteristic corresponding to the initial datum $\bigl( x_0,p_0(x_0)\bigr)$, $x_0=(t,f(t))\in\gamma$. Both the domain $D=\bigcup_{t\in[a,b]} \{t\} \times [T^\text{t}(t),T^\text{b}(t)]$ of $\Psi$ and $M$ are homeomorphic to a disk. Moreover, $\Psi(\partial D)=\partial M$, and the restriction $\Psi\big\vert_{\partial D}$ has topological degree $1$. Therefore $\Psi$ is surjective.

  Indeed, if $R:=\Psi(D)\subsetneq M$, $\Psi$ would induce a continuous map $g \colon \mathbb{D}^2 \to \mathbb{S}^1$ with $g\big\vert_{\mathbb{S}^1} \simeq \id_{\mathbb{S}^1}$, i.e. the sphere were a retract of the disk, which is impossible. Hence $R=M$, and Lemma~\ref{lem:propeik2} is proven.
\end{proof}